\documentclass{amsart}

\def\titlepaper{Minimal nondegenerate extensions}
\def\authorpaper{Theo Johnson-Freyd \& David Reutter}

\usepackage{amsmath}       
\usepackage{amsthm}        
\usepackage{amssymb}       
\usepackage{amsfonts}
\usepackage[all]{xy}
\usepackage{xspace}
\usepackage{calc}
\usepackage{pgfplots}
\usepgflibrary{fpu}
\usepackage{mathtools}
\usepackage{tabularx}
\usepackage{ifthen}		
\usepackage{graphicx}
\usepackage{docmute}
\usepackage{array}
\usepackage{subfloat} 
\usepackage{bbm}            
\usepackage{etoolbox}  
\usepackage{verbatim}
\usepackage{stmaryrd}
\usepackage{thm-restate}
\usepackage{multicol}
\usepackage{nccmath}
\usepackage{stackrel}
\usepackage[T1]{fontenc}

\usepackage[margin=1.5in]{geometry}
\renewcommand{\to}[1][]{\ensuremath{\xrightarrow{#1}}}
\newcommand{\To}[1][]{\ensuremath{\xRightarrow{#1}}}

\allowdisplaybreaks[1]

\theoremstyle{plain} %%% Plain Theorem Styles.

\newtheorem*{mainthm}{Main Theorem}

\newtheorem{theorem}{Theorem}[section]
\newtheorem{lemma}[theorem]{Lemma}
\newtheorem{corollary}[theorem]{Corollary}          
\newtheorem{proposition}[theorem]{Proposition}              
\newtheorem{cor}[theorem]{Corollary}          
\newtheorem{prop}[theorem]{Proposition}   
   
\newtheorem{conjecture}[theorem]{Conjecture}

\theoremstyle{definition} %%%% Definition-like Commands  
\newtheorem{definition}[theorem]{Definition}

\theoremstyle{remark}  %%%% Remark-like Commands

\newtheorem{remark}[theorem]{Remark}
\newtheorem{example}[theorem]{Example}

\newtheorem*{guide*}{Guide}
\newtheorem*{outline*}{Outline}
\newtheorem*{remarkohc*}{Remark on higher categories and the cobordism hypothesis}
\newtheorem*{assumpfield*}{Assumptions on the base field}

\newcommand\MainTheorem{\hyperlink{maintheorem}{Main Theorem}}

\usepackage{bbm}
\newcommand\one{I}
\newcommand\bk{\mathbbm{k}}

\newcommand\fl{\mathrm{fl}}

\DeclareMathOperator\tr{tr}

\DeclareMathOperator\Mod{Mod}

\DeclareMathOperator\sHBA{sHBA}

\DeclareMathOperator\HBA{HBA}
\DeclareMathOperator\cusp{cusp}
\DeclareMathOperator\rot{rot}
\DeclareMathOperator\coeq{coeq}

\newcommand\longto\longrightarrow
\newcommand\isom{\overset\sim\rightarrow}
\newcommand\Isom{\overset\sim\Rightarrow}

\newcommand\Inv{\kappa}

\DeclareMathOperator{\Hom}{Hom}
\DeclareMathOperator{\End}{End}

\newcommand{\id}{\mathrm{id}}

\newcommand{\Aut}{\mathrm{Aut}}

\renewcommand{\Vec}{\mathrm{Vec}}

\newcommand{\ev}{\mathrm{ev}}
\newcommand{\coev}{\mathrm{coev}}

\newcommand{\op}{\mathrm{op}}

\newcommand\st{\text{ s.t.\ }}

\renewcommand\Vec{\cat{Vec}}
\newcommand\sVec{\cat{sVec}}
\newcommand\Rep{\cat{Rep}}

\DeclareMathOperator\Sq{Sq}

\newcommand\coh{\mathrm{coh}}
\newcommand\br{\mathrm{br}}

\newcommand\rev{\mathrm{rev}}

\DeclareMathOperator\SH{SH}
\newcommand\sRep{\cat{sRep}}

\newcommand\quot{/\! /}

\DeclareMathOperator\Witt{Witt}
\DeclareMathOperator\coker{coker}
\DeclareMathOperator\Inn{Inn}
\DeclareMathOperator\Mext{Mext}
\DeclareMathOperator\NBFC{BFC}
\DeclareMathOperator\Obs{Obstr}

\newcommand\rh{\mathrm{h}}

\def\cA{\mathcal A}\def\cB{\mathcal B}\def\cC{\mathcal C}\def\cD{\mathcal D}
\def\cE{\mathcal E}\def\cF{\mathcal F}\def\cG{\mathcal G}
\def\cL{\mathcal L}
\def\cM{\mathcal M}\def\cN{\mathcal N}
\def\cS{\mathcal S}\def\cT{\mathcal T}
\def\cX{\mathcal X}
\def\cY{\mathcal Y}\def\cZ{\mathcal Z}

\newcommand\bC{\mathbb C}

\newcommand\bQ{\mathbb Q}
\newcommand\bR{\mathbb R}
\newcommand\bS{\mathbb S}

\newcommand\bZ{\mathbb Z}

\newcommand\rB{\mathrm B}

\newcommand\rO{\mathrm O}

\usepackage[mathscr]{euscript}

\usepackage[pdftex, colorlinks=true, linkcolor=black, citecolor=black, urlcolor=black, hypertexnames=false,
pdfauthor={\authorpaper},
pdftitle={\titlepaper}
	]{hyperref}

\renewcommand\d{\mathrm{d}}

\newcommand\define[1]{\emph{#1}}
\newcommand\cat[1]{\mathbf{#1}}

\newcommand\ignore[1]{}

\newcommand\Modl[2][]{{#2\text{-}\mathrm{Mod}_{#1}}}
\newcommand\Modr[2][]{{_{#1}\mathrm{Mod}\text{-}#2}}
\newcommand\BrModl[1]{#1\text{-}\mathrm{BrMod}}
\newcommand\BrModr[2][]{{_{#1}}\mathrm{BrMod}\text{-}#2}
\newcommand\Bim[1]{\mathrm{Bim}\text{-}#1}

\newcommand\arXiv[1]{\href{http://arxiv.org/abs/#1}{\nolinkurl{arXiv:#1}}}
\newcommand\MRnumber[2]{}
\newcommand\DOI[1]{\href{http://dx.doi.org/#1}{\nolinkurl{DOI:#1}}}
\newcommand\MAILTO[1]{\href{mailto:#1}{\nolinkurl{#1}}}

\usepackage{tikz}
\usetikzlibrary{matrix}
\usetikzlibrary{cd}
\usetikzlibrary{knots}
\usetikzlibrary{arrows}
\usetikzlibrary{arrows.meta}
\usetikzlibrary{decorations.pathreplacing}
\usetikzlibrary{decorations.markings}
\usetikzlibrary{arrows}
\usetikzlibrary{calc}
\usetikzlibrary{shapes.misc}
\usetikzlibrary{fit}
\usepgflibrary{decorations.pathmorphing}
\usepgflibrary{shapes.geometric}
\usetikzlibrary{arrows.meta}

\newenvironment{tz}[1][]{%
			\begin{tikzpicture}[baseline={([yshift=-.8ex]current bounding box.center)}, #1]%
				}{%
				\end{tikzpicture}%
				}%

\tikzset{arrow data/.style 2 args={
      decoration={
         markings,
         mark=at position #1 with {\arrow[scale=0.7]{#2}}}, 
         postaction=decorate}
}

\pgfdeclarelayer{back}
\pgfdeclarelayer{backback}
\pgfdeclarelayer{front}
\pgfsetlayers{backback,back,main,front}

\makeatletter
\pgfkeys{%
  /tikz/on layer/.code={
    \pgfonlayer{#1}\begingroup
    \aftergroup\endpgfonlayer
    \aftergroup\endgroup
  },
  /tikz/node on layer/.code={
     \gdef\node@@on@layer{%
      \setbox\tikz@tempbox=\hbox\bgroup\pgfonlayer{#1}\unhbox\tikz@tempbox\endpgfonlayer\egroup}
    \aftergroup\node@on@layer
  },
  /tikz/end node on layer/.code={
    \endpgfonlayer\endgroup\endgroup
  }
}
\def\node@on@layer{\aftergroup\node@@on@layer}

\def\black{black}

\tikzset{label/.style = {scale=0.5}}
\tikzset{braid/.style={ preaction={draw=white, line width =3pt, on layer=#1}, draw =black, line width=0.7pt}}
\tikzset{braid/.default=main}
\tikzset{box/.style={line width =0.3pt}}
\tikzset{string/.style={draw=black, line width=0.7pt}}
\tikzset{dinat/.style={draw=red, line width=0.7pt}}
\tikzset{hb/.style={draw=blue,  line width=0.7pt}}

 \def\dotscale{0.3}
\tikzset{dot/.style n args={1}{draw,fill=#1,circle,line width=0.7pt,scale=\dotscale}}
\tikzset{smalldot/.style ={node on layer=front, fill = black, rectangle, line width =0.7pt, scale=1.5*\dotscale}}
\tikzset{smalldotgreen/.style ={node on layer=front, fill = black!20!green,, circle, line width =0.7pt, scale=0.75*\dotscale}}

\tikzset{std/.style = {string, scale=0.55}}
\tikzset{box/.style={rectangle, draw, fill=white, minimum width =#1, label}}

\def\shift{0.075cm}
\tikzset{upshift/.style={yshift=\shift}}
\tikzset{downshift/.style={yshift=-\shift}}

\newcommand\mult[3]{ %mult{base}{width}{height}  
\draw[string] (#1.center) to [out=up, in=-135] +(0.5*#2,#3) to [out=-45, in=up] +(0.5*#2,-#3);
\node[dot=\black,downshift] at ($(#1)+(0.5*#2,#3)$){};
}
\newcommand\braidmult[3]{ %mult{base}{width}{height}  
\draw[braid,on layer = back] (#1.center) to [out=up, in=-135] +(0.5*#2,#3);
\draw[braid = front, on layer=front] ($(#1)+(0.5*#2,#3)$) to [out=-45, in=up] +(0.5*#2,-#3);
\node[node on layer=front,dot=\black,downshift] at ($(#1)+(0.5*#2,#3)$){};
}
\newcommand\braidlact[3]{ %mult{base}{width}{height}  
\draw[braid,on layer = back] (#1.center) to [out=up, in=-135] +(0.5*#2,#3);
\node[node on layer=front,dot=\black,downshift] at ($(#1)+(0.5*#2,#3)$){};
}
\newcommand\braidract[3]{ %mult{base}{width}{height}  
\draw[braid, on layer=front] ($(#1)+(0.5*#2,#3)$) to [out=-45, in=up] +(0.5*#2,-#3);
\node[node on layer=front,dot=\black,downshift] at ($(#1)+(0.5*#2,#3)$){};
}

\usetikzlibrary{3d}
\tikzset{xzplane/.style={canvas is xz plane at y=#1}}
\tikzset{td/.style={
					y={(0.6cm, 0.4cm)},
					x={(1cm, 0cm)},					
                    z  = {(0cm,1cm)},
                    scale = 0.55}}
\tikzset{slice/.style={draw = gray!50, line width = 0.7pt}}              
  \tikzset{obj/.style={scale=0.6, color=gray}}

\tikzset{21braid/.style = {draw=black, line width =0.3pt}}

  \tikzset{surface/.style={ fill=gray!60, opacity=0.8}}
\tikzset{hb3d/.style={smalldot}}
\tikzset{1mor/.style= {string}}

\def\maskpointheight{5pt}
\tikzset{mask point/.style={ transform shape, sloped, 
minimum width=15pt, minimum height=\maskpointheight, inner sep=0pt, ultra thin, font=\tiny, node on layer=#1}}
\tikzset{mask point/.default=main}

\tikzset{
clip even odd rule/.code={\pgfseteorule},
invclip/.style={clip,insert path=[clip even odd rule]{
   [reset cm](-\maxdimen,-\maxdimen)rectangle(\maxdimen,\maxdimen)
    }}} 
\newcommand\clipintersection[1]{(#1.north west) -- (#1.north east) -- (#1.south east) -- (#1.south west) -- (#1.north west)}

\newcommand\cliparoundone[3][main]{\begin{pgfscope}\begin{scope}[overlay]
\path[invclip, on layer=#1] \clipintersection{#2};#3
\end{scope}\end{pgfscope}}

\def\circclip{2pt}
\tikzset{circ mask point/.style={ circle,
minimum width=\circclip, inner sep=0pt, ultra thin, font=\tiny, node on layer=#1}}
\tikzset{circ mask point/.default=main}
\newcommand\circclipintersection[1]{(#1.center) circle (\circclip) }
\newcommand\circcliparoundone[3][main]{\begin{pgfscope}\begin{scope}[overlay]
\path[invclip, on layer=#1] \circclipintersection{#2};#3
\draw[21braid] (#2.center) circle (\circclip);
\end{scope}\end{pgfscope}}

\def\boxmaskpointheight{6pt}
\tikzset{box mask point/.style={ transform shape, sloped, 
minimum width=3.5pt, minimum height=\boxmaskpointheight, inner sep=0pt, ultra thin, font=\tiny}}
\newcommand\boxclipintersection[1]{
(#1.north west)  .. controls ($(#1.west)!0.7!(#1.center)$) and ($(#1.west)!0.7!(#1.center)$) .. (#1.south west) to (#1.south east) .. controls ($(#1.east)!0.7!(#1.center)$) and ($(#1.east)!0.7!(#1.center)$) .. (#1.north east) to (#1.north west)}
\newcommand\boxclipshadow[1]{
\draw[21braid](#1.north west)  .. controls
($(#1.west)!0.7!(#1.center)$) and ($(#1.west)!0.7!(#1.center)$) ..(#1.south west);
\draw[21braid](#1.north east)  .. controls ($(#1.east)!0.7!(#1.center)$) and ($(#1.east)!0.7!(#1.center)$) .. (#1.south east);
}
\newcommand\boxcliparoundone[3][main]{\begin{pgfscope}\begin{scope}[overlay]
\path[invclip, on layer=#1] \boxclipintersection{#2};#3
\end{scope}\end{pgfscope}}

\DeclareMathOperator\Homology{H}
\renewcommand\H{\Homology}

\DeclareMathOperator\rad{rad}

\title{Minimal nondegenerate extensions}

\author{Theo Johnson-Freyd}
\address{Department of Mathematics, Dalhousie University \& \newline \indent Perimeter Institute for Theoretical Physics}
\email{theojf@dal.ca}
\urladdr{http://categorified.net}

\author{David Reutter}
\address{Fachbereich Mathematik, Universit\"at Hamburg}
\email{david.reutter@uni-hamburg.de}
\urladdr{https://www.davidreutter.com}

\begin{document}

\begin{abstract}
We prove that every slightly degenerate braided fusion category admits a minimal nondegenerate extension, and hence that every pseudo-unitary super modular tensor category admits a minimal modular extension.
This completes the program of characterizing minimal nondegenerate extensions of braided fusion categories.

Our proof relies on the new subject of fusion $2$-categories. We study in detail the Drinfel'd centre $\cZ(\Modr{\cB})$ of the fusion $2$-category $\Modr{\cB}$ of module categories of a braided fusion $1$-category $\cB$. 
We show that minimal nondegenerate extensions of $\cB$ correspond to certain trivializations of $\cZ(\Modr{\cB})$. In the slightly degenerate case, such trivializations are obstructed by a class in $\H^5(K(\bZ_2, 2); \bk^\times)$ and we use a numerical invariant --- defined by evaluating a certain two-dimensional topological field theory on a Klein bottle --- to prove that this obstruction always vanishes.

Along the way, we develop techniques to explicitly compute in braided fusion $2$-categories which we expect will be of independent interest.
In addition to the 
model of $\cZ(\Modr{\cB})$ in terms of braided $\cB$-module categories, we develop a 
 computationally useful model in terms of certain algebra objects in~$\cB$. 
We construct an $S$-matrix pairing for any braided fusion $2$-category, and show that it is nondegenerate for $\cZ(\Modr{\cB})$. As a corollary, we identify components of $\cZ(\Modr{\cB})$ with blocks in the annular category of $\cB$ and with the homomorphisms from the Grothendieck ring of the M\"uger centre of $\cB$ to the ground field.
\end{abstract}

\maketitle	

\tableofcontents

\section{Introduction} \label{sec.intro}

\subsection{Motivation and statement of the main result}

This paper completes the long-standing program of classifying minimal nondegenerate extensions of braided fusion categories, a question first raised in \cite{MR1990929} and which has occupied many papers in recent years \cite{MR3039775,MR3022755,foldway,1712.07097,1703.05787,1908.07487,MR3613518,MR3775361,DN,2105.01814}: we prove that there is no obstruction to finding a minimal nondegenerate extension for any slightly degenerate braided fusion category. Throughout this paper we work over a fixed algebraically-closed characteristic-zero field $\bk$.

Braided fusion categories play a key role in the representation theory of quantum groups, Hopf algebras, and vertex algebras~\cite{MR1300632,EGNO,MR2023933}, and in the construction of three-manifold invariants~\cite{RT,MR1186845}. Moreover, braided fusion categories organize the particle excitations in topological 2+1D quantum systems \cite{MR1002038,MR990772,MR1043302}, leading to their great importance in the study of quantum matter and quantum computation~\cite{MR1951039,MR2640343}.
Braided fusion categories are a ``noninvertible'' generalization of finite abelian groups equipped with quadratic forms. Indeed, a braided fusion category in which every simple object is invertible is uniquely determined by its finite abelian group $A$ of invertible objects and  their self-braidings $\beta_{x,x} : x \otimes x \to x \otimes x$ assembled into a quadratic form $q:A \to \bk^\times$~\cite{MR1250465}. The associated symmetric bilinear form $b(x,y):= \frac{q(x+y)}{q(x) q(y)}$ records the full braiding $\beta_{y, x} \circ \beta_{x,y}:x \otimes y \to x \otimes y$ of invertibles.

Suppose that $B \subset C$ is a subgroup of an abelian group $C$ with a symmetric bilinear form $b$. The \define{orthogonal} of this inclusion is the subgroup $\{y \in C~|~ b(x,y) = 1~\forall x \in B\}$. The \define{radical} $\rad(A,b)$ of a symmetric bilinear form $b$ on $A$ is the orthogonal of the identity inclusion $A \subset A$, and $(A,b)$ is \define{nondegenerate} when its radical is the trivial group. When $b$ is the associated bilinear form of a quadratic form $q$, then there is a finer notion of ``radical'' of the quadratic form $q$. Specifically, the restriction $q|_{\rad(A,b)} : \rad(A,b) \to \bk^\times$ is a homomorphism which takes values in $\{\pm 1\}\subset \bk^\times$, and the radical $\rad(A,q)$ is defined as the kernel of $q|_{\rad(A,b)}$. The pair $(A,q)$ is \define{slightly degenerate} when $\rad(A,q)$ is trivial but $\rad(A,b)$ is not: equivalently, when $\rad(A,b) \cong \bZ_2 = \{1,e\}$ and $q(e) = -1$.

In the context of braided fusion categories, the ``orthogonal'' of a fully-faithful inclusion $\cB \subset \cC$ is called the \define{centralizer}, defined to be the full braided subcategory of $\cC$ of objects transparent to those objects of $\cB$:
$$ \cZ_{2}(\cB \subset \cC) := \{y \in \cC \st \beta_{y,x}\circ\beta_{x,y} = \id_{x \otimes y} \; \forall x \in \cB\}.$$
The ``radical'' of a braided fusion category $\cB$ is the centralizer of the identity inclusion, and is called the \define{M\"uger centre} $\cZ_{(2)}(\cB)$ of $\cB$. It is automatically a symmetric monoidal category and is contained as a full subcategory in the centralizer $\cZ_{2}(\cB \subset \cC)$ of any inclusion. A braided fusion category $\cB$ is called \define{nondegenerate} when its M\"uger centre is the ``trivial symmetric fusion category'', i.e.\ when $\cZ_{(2)}(\cB)$ is equivalent to the category $\Vec$ of finite-dimensional vector spaces. A braided fusion category $\cB$ is \define{slightly degenerate} when $\cZ_{(2)}(\cB)$ is the categorical version of $(\bZ_2, q : e \mapsto -1)$ from above, in other words when $\cZ_{(2)}(\cB)$ is equivalent as a symmetric fusion category to the category $\sVec$ of finite-dimensional super vector spaces. 

Let $(A,q)$ be a finite abelian group with a quadratic form $q$ and associated bilinear form $b$. It is natural to try to extend $(A,q)$ to a larger group $A'$ with a quadratic form $q'$ which is nondegenerate. Such an extension is \define{minimal} when the orthogonal of $A \subset A'$ contains nothing more than $\operatorname{rad}(A,b)$. For groups, it turns out that such minimal extensions always exist. The same question makes sense for braided fusion categories: a \define{minimal nondegenerate extension} of a braided fusion category $\cB$ is an inclusion $\cB \subset \cM$ such that $\cM$ is nondegenerate and the canonical inclusion $\cZ_{2}(\cB) \subset \cZ_{2}(\cB \subset \cM)$ is an equivalence. Generalizing the group case, M\"uger asked in \cite{MR1990929} whether every braided fusion category admits a minimal nondegenerate extension.

Somewhat surprisingly, given the group case, the answer to M\"uger's question is ``No,'' as first demonstrated by an explicit counterexample constructed by Drinfel'd in unpublished work. 
Specifically, there is a series of cohomological obstructions, detailed in \cite{1712.07097}, which reduce the minimal nondegenerate extension theory for general braided fusion categories 
to the problem of finding minimal nondegenerate extensions for slightly degenerate braided fusion categories. 
This last problem remained open for many years and appeared for example as Question 5.15 in~\cite{MR3022755}, Conjecture~3.9 in~\cite{foldway}, and Conjecture 1.1 in~\cite{2105.01814}. Our main result is a solution to this problem based on the new theory of fusion 2-categories introduced in~\cite{DR}:
\begin{mainthm}\hypertarget{maintheorem}{}
  Every slightly degenerate braided fusion category admits a minimal nondegenerate extension.
\end{mainthm}
Ultimately, our proof of this theorem is constructive: In Remark~\ref{rem:Z2ext}, we describe how to construct a minimal nondegenerate extension of a slightly degenerate  braided fusion category from certain choices of data. 

Already in the group case, there are multiple minimal nondegenerate extensions, and there can fail to be a canonical choice. This is a hallmark of obstruction theory, and so in hindsight does predict the presence of possible obstructions. For a braided fusion category $\cB$ whose M\"uger centre is \define{Tannakian}, i.e.\ equivalent to the category $\Rep(G)$ of finite-dimensional representations of some finite group $G$, the obstruction to having a minimal nondegenerate extension lives in $\H^4_{\mathrm{gp}}(G; \bk^\times)$~\cite{1712.07097}. Moreover, as we explain in Section~\ref{subsec.completeobstructions}, every class in $\H^4_{\mathrm{gp}}(G; \bk^\times)$ can be realized as this obstruction for some braided fusion category.

In that same section, we also explain a complete ``supercohomological'' obstruction theory for the general non-Tannakian case, thereby answering a question asked by~\cite{OstrikMCA}. One outcome of our paper overall is that when $\cB$ is slightly degenerate, its obstruction to admitting a minimal nondegenerate extension is naturally a class in $\H^5(K(\bZ_2, 2); \bk^\times) \cong \bZ_2$.  The proof of the \MainTheorem\ proceeds by providing a universal calculation of this obstruction and showing that, while its ambient group is nontrivial, the obstruction itself vanishes for all slightly degenerate braided fusion categories.  From this perspective, the slightly degenerate case is in stark contrast to the Tannakian case where all obstruction classes are realizable.

\begin{remark}
  The nontriviality of the obstruction group $\H^5(K(\bZ_2, 2); \bk^\times) \cong \bZ_2$, but the triviality of the obstruction itself, has physical significance. As explained further in \cite{2011.11165}, it implies the existence of a 3+1D gapped topological phase of matter which is ``chiral'' in the sense that it does not admit a gapped topological boundary condition. The extended operators of this phase are described by the braided fusion $2$-category $\cT$ of Theorem~\ref{thm:SandT}.
  The existence of such a phase 
  was predicted in~\cite{MR3321301}, which analyzes a 3+1D gapped quantum field theory with gravitational 't Hooft anomaly $(-1)^{w_2w_3}$.
   However, that anomaly trivializes on framed manifolds, and hence trivializes if one allows Lorentz symmetry to break. Our \MainTheorem\ implies that even if one is willing to break Lorentz symmetry (which can e.g. happen in lattice systems with an ambient magnetic field) then still this phase does not admit a gapped topological boundary condition. \end{remark}

\subsection{Structure of the proof}

Section~\ref{sec.2cats} records a number of results about braided fusion 1- and 2-categories. 
These both set up our proof and may be of independent interest. The first step is to translate the question of classifying minimal nondegenerate extensions into 2-categorical language. Some ideas of this translation appeared in the lecture \cite{DmitriMSRI} and the companion paper \cite{DN}. The version we will use builds on the recent work \cite{DR,GJFcond,2003.06663} setting out the basic theory of fusion 2-categories. 
We review the main elements in \S\ref{subsec.defnitions} and \S\ref{subsec.Zabstract}. Specifically, each braided fusion $1$-category $\cB$ can be ``suspended'' to a fusion $2$-category, whose Drinfel'd centre $\cZ(\Sigma\cB)$ is a braided fusion $2$-category, equivalent to the braided fusion $2$-category $\BrModr{\cB}$ of braided $\cB$-module categories from \cite{MR3248737,MR3874702,DN}. Our first main result is Theorem~\ref{thm.BCs} in \S\ref{subsec.BCthm}, where we show that to build a minimal nondegenerate extension of $\cB$, it suffices to identify $\cZ(\Sigma\cB)$ with $\cZ(\Sigma\cZ_{2}(\cB))$. We then given an explicit model of $\cZ(\Sigma\cB)$ in terms of ``half-braided algebras'' in \S\ref{subsec.halfbraidedalgs}, and use it to conclude a number of facts about $\cZ(\Sigma\cB)$: in particular, in Corollary~\ref{cor.Lgenerates} in \S\ref{subsec.canonicalLagHBA} we show that $\cZ(\Sigma\cB)$ is generated under direct sums and idempotent completions by the 1-categorical Drinfeld centre $\cZ(\cB)$ thought of as a braided $\cB$-module; using this, Theorem~\ref{thm:componentcount} in \S\ref{subsec.compZSB} and Theorem~\ref{thm:invertibleSmatrix} in \S\ref{sec:Smatrix} first count and then enumerate the ``components'' of $\cZ(\Sigma\cB)$, identifying them via a 2-categorical ``$S$-matrix'' with the set of ring homomorphisms $K_0(\cZ_{2}(\cB)) \to \bk$ from the Grothendieck ring of $\cZ_{2}(\cB)$.

With these general results established, we prove our \MainTheorem\ in Section~\ref{sec.proof}. The results in Section~\ref{sec.2cats} significantly constrain the braided monoidal equivalence class of $\cZ(\Sigma\cB)$ for a slightly degenerate braided fusion 1-category $\cB$, and reduce the problem to showing that  $\cZ(\Sigma\cB)$ is in fact independent, up to noncanonical equivalence, of the choice of $\cB$. Theorem~\ref{thm:SandT} in \S\ref{subsec:SandT} classifies the two possible options for $\cZ(\Sigma\cB)$ consistent with the results of Section~\ref{sec.2cats}, and in~\S\ref{subsec:eta} and~\S\ref{subsec:inv} we construct an invariant $\kappa$ of self-dual objects, which is reminiscent of a 2-categorical version of the Frobenius--Schur indicator. We show in \S\ref{subsec:KleinST} that our invariant $\kappa$ can distinguish the two possible options for $\cZ(\Sigma\cB)$.
In~\S\ref{subsec:twistedtrace} and \S\ref{subsec:etaL} we develop the computational techniques needed to calculate $\kappa$ in the half-braided algebra model of $\cZ(\Sigma\cB)$. These pieces are assembled and the proof completed in~\S\ref{subsec:laststep}.

\begin{remark}\label{rem:Z2ext}
Building on a result of Nikshych and Davydov on $\bZ_2$-extensions of braided fusion categories~\cite[Theorem 8.18]{DN}, and upcoming work of Douglas--Schommer-Pries--Snyder~\cite{DSPS-SO3} on a minimal CW decomposition of the $4$-type of $\mathrm{BSO}(3)$, we now give an alternative proof of our \MainTheorem\ leading to an explicit construction of the non-degenerate extension of any braided fusion $1$-category. This alternative proof avoids Theorems~\ref{thm.BCs} and~\ref{thm:SandT} (along with the rest of \S\ref{subsec.BCthm}, \S\ref{subsec:SandT} and \S\ref{subsec:KleinST}), but still relies on Theorems~\ref{thm:componentcount},~\ref{thm:invertibleSmatrix} and the explicit computations in \S\ref{subsec:twistedtrace}--\ref{subsec:laststep}.

Suppose that $\cB$ is a slightly-degenerate braided fusion $1$-category. As is true for any braided fusion $1$-category, $\cB$ determines a braided monoidal 2-category $\BrModr{\cB}$ of braided $\cB$-module categories \cite{MR3248737,MR3874702,DN}. Theorems~\ref{thm:componentcount} and \ref{thm:invertibleSmatrix} describe this 2-category in detail.  In particular, if $\cB$ is slightly degenerate, then there are four equivalence classes of indecomposable braided $\cB$-module categories.
  Two of them are ``magnetically neutral'': the transparent simple ``electron'' $e \in \cZ_{2}(\cB)$ braids with them trivially. The other two are ``magnetically charged'': the braiding with $e$ is $-1$. 
  
  Arbitrarily select one of the two magnetically charged indecomposable braided $\cB$-modules $\cX$.   One way to do this is to decompose the Drinfel'd centre $\cZ(\cB)$ as a sum of indecomposable braided $\cB$-modules: by Corollary~\ref{cor.Lgenerates}, both neutral and magnetically charged summands appear. The fusion $\cX \boxtimes_\cB \cX$ turns out to be equivalent as a braided $\cB$-module to $\cB$ itself; arbitrarily choose such an equivalence $r: \cX \boxtimes_{\cB} \cX \isom \cB$. Up to isomorphism, there are precisely two such choices of $r$.  
  
  These choices determine a bifunctor on the category $\cM := \cB \boxplus \cX$. 
By~\cite[Theorem 8.18]{DN}, extensions of this bifunctor to a full braided monoidal structure  (which restricts to the braided monoidal structure on $\cB$ and to the braided module structure on $\cX$) correspond to extensions of the data $(\cX, r)$ to a braided monoidal $2$-functor $\bZ_2 \to \BrModr{\cB}$ where we think of the commutative group~$\bZ_2$ as a braided monoidal 2-category with only identity 1- and 2-morphisms. Using upcoming results of Douglas--Schommer-Pries--Snyder \cite{DSPS-SO3} to build a minimal CW decomposition of the $4$-type of  $K(\bZ_2,2)$, it follows from Corollary~\ref{cor:CW} that such extensions exist if and only if the self-braiding $\br_{\cX, \cX}$ is isomorphic to the identity on $\cX \boxtimes_{\cB} \cX$ and a certain ``Klein bottle'' obstruction $\Inv(\cX, r) \in \bk^\times $ is trivial. In this case,  there are exactly four such choices of extension corresponding to four admissible choices of isomorphism $\phi: \br_{\cX, \cX} \cong \id_{\cX \boxtimes_{\cB} \cX}$ (see Corollary~\ref{cor:CW} for more details). 

Trivializability of $\br_{\cX, \cX}$ holds for any choice of magnetic simple object $\cX$ and is explained in the first part of the proof of Theorem~\ref{thm:SandT}, while triviality of the Klein obstruction $\Inv(\cX, r)$ follows from the computations in \S\ref{subsec:twistedtrace}--\ref{subsec:laststep}. 

Since $\cX$ is magnetically charged, the braided extension $\cM$ is non-degenerate, resulting in the minimal nondegenerate extension of $\cB$ that we are after. The 16-fold way emphasized in \cite{foldway} arises from the $2 \times 2 \times 4$ choices of $\cX$, $r$ and $\phi$, respectively,  needed in the construction.
\end{remark}

\subsection{Minimal modular extensions for pseudo-unitary categories}

Our \MainTheorem\ may be extended to account for ribbon structures and modularity, provided these structures are (pseudo-)unitary. Recall that, when $\bk = \bC$, a braided fusion category is \define{pseudo-unitary} if it admits a (necessarily unique) ``positive'' ribbon structure for which all quantum dimensions are positive. (For comparison, a \define{unitary} braided fusion category is a braided fusion category equipped with a positive $*$-, or dagger-structure. The underlying braided fusion category of such a category is pseudo-unitary and has a unique positive $*$-structure~\cite{MR3239112, 1906.09710}, but it remains open whether every pseudo-unitary braided fusion category admits a positive $*$-structure.) 
 A pseudo-unitary \emph{super-modular category}, respectively \define{modular category}, is a slightly degenerate, respectively nondegenerate, pseudo-unitary braided fusion category equipped with its unique positive ribbon structure \cite{foldway}. The term \define{spin-modular} is used in \cite{MR3692909,MR2163566,foldway} for a stronger notion than super-modularity: a modular tensor category $\cM$ equipped with a chosen ribbon embedding $\sVec \subset \cM$. The centralizer $\cB = \cZ_{2}(\sVec \subset \cM)$ is therefore super-modular, but, as emphasized in \cite{foldway}, a super-modular category can extend to a spin-modular one in multiple ways.

\begin{corollary} \label{cor.unitary}
Every pseudo-unitary super-modular category admits a pseudo-unitary minimal modular extension. 
\end{corollary}

\begin{proof}  By our \MainTheorem, any slightly degenerate braided fusion category $\cB$ admits a minimal nondegenerate extension $\cM$. By~\cite[Propoposition 8.23]{MR2183279}, a braided fusion category $\cB$ is pseudo-unitary if and only if its global dimension agrees with its Frobenius-Perron dimension.  By~\cite[Theorem~3.10]{DGNO} and~\cite[Theorem~3.14]{DGNO}, respectively, any nondegenerate extension $\cB \subset \cM$ of a braided fusion category $\cB$ has the following global dimension and Frobenius-Perron dimension:
\begin{gather*}
\dim(\cM) =  \dim(\cB)\dim(\cZ_{2}(\cB \subset \cM)), \\ \operatorname{FPdim}(\cM) = \operatorname{FPdim}(\cB) \operatorname{FPdim}(\cZ_{2}(\cB \subset \cM)).\end{gather*}
Hence, if $\cM$ is a nondegenerate extension of a slightly degenerate pseudo-unitary $\cB$, it follows that \[\dim(\cM) = 2 \dim(\cB) = 2 \operatorname{FPdim}(\cB) = \operatorname{FPdim}(\cM)\] and therefore that $\cM$ is also pseudo-unitary. In particular, the unique positive ribbon structure on $\cB$ extends to the unique positive ribbon structure on $\cM$. 
\end{proof}

The more general question (also asked in \cite{MR1990929}) of whether every super-modular category admits a minimal modular extension remains open. We expect it can be settled with similar techniques.

\begin{remark}
  We announced a proof of the special case Corollary~\ref{cor.unitary} of our \MainTheorem\ in November 2020, and one of us presented an outline at the lecture \cite{MPPMlecture}. Our proof at the time was closely inspired by ideas from condensed matter and topological quantum field theory, and involved calculating what in physics would be considered ``anomalies'' of  ``higher symmetries'' and the ``coupling to gravity.''  Our proof did require some positivity statements in order to rule out certain anomalies. Our interpretation of this positivity requirement was that we were using some higher-categorical version of a ``spin-statistics theorem'' (and physicists know well that spin-statistics theorems can fail if unitarity and positivity requirements are dropped).
  
  Our proof in this paper is completely algebraic. We found it by studying our anomaly calculations and discovering that they required less ``orientability'' or ``unitarity'' than we expected on physical grounds. Since we do not have a satisfactory physical/topological explanation of this, we have chosen to present our story emphasizing the role played by (braided) fusion 2-categories, and the parallels between that theory and the more familiar theory of (braided) fusion 1-categories,
 and not to elaborate on the topological quantum field theoretic interpretation.
\end{remark}

\subsection{Acknowledgments}

We thank Corey Jones and Dmitri Nikshych for many helpful conversations related to this work. We are grateful to Noah Snyder for pointing out that the Klein invariant $\Inv$ appears as the attaching map of the $5$-cell in a minimal cell decomposition of $\mathrm{BSO}(3)$, leading to Proposition~\ref{prop:CW} and Corollary~\ref{cor:CW} and the alternative proof in Remark~\ref{rem:Z2ext},
{and to Pavel Etingof, for his interest and his useful comments on an early version of this paper.}
 We furthermore thank the anonymous referees for their careful reading and valuable comments, and for suggesting Remark~\ref{rem:refrequest}.

Parts of this project were completed during the workshop ``Fusion categories and tensor networks'' at the American Institute for Mathematics, and we thank AIM for the generous support. 
This work is further supported by the NSERC Discovery Grant RGPIN-2021-02424, by the Simons Collaboration on Global Categorical Symmetries, and by the Deutsche Forschungsgemeinschaft (DFG, German Research Foundation) – 493608176.
TJF's research at the Perimeter Institute is supported in part by the Government of Canada through the Department of Innovation, Science and Economic Development Canada and by the Province of Ontario through the Ministry of Colleges and Universities. DR is grateful for the hospitality and financial support of the Max Planck Institute for Mathematics where much of this work was carried out.

\section{Fusion 2-categories and their Drinfel'd centres} \label{sec.2cats}

\subsection{Multifusion 2-categories} \label{subsec.defnitions}

Fusion 2-categories were first axiomatized in \cite{DR}, although variants had earlier been considered nonrigorously in the theoretical physics literature \cite{1405.5858,KWZ1,KWZ2,1704.04221,PhysRevX.9.021005}. The theory was then further developed by \cite{GJFcond,2003.06663}. Those latter papers study (fusion) $n$-categories for arbitrary $n$, and when $n$ is large they make  some explicit assumptions about the theory of colimits in enriched $(\infty,1)$-categories which have not been verified in the literature. Although we will lean on those papers for some intuition and notation, we will not require those unverified assumptions: in particular, we will not need fusion $n$-categories for $n>2$, and the theory of fusion $2$-categories of \cite{DR} is fully rigorized within the framework of  monoidal bicategories. Monoidal bicategories date back as far as \cite{MR1402727}, and  a good reference is the appendix to~\cite{SchommerPries}. 

\begin{remark}\label{conv:catno}
  By ``1-category'' we mean the usual strict notion, and by ``2-category'' we mean ``weak 2-category'' or equivalently ``bicategory.'' That said, we will generally suppress associator and unitor data for clarity; this is justified by the appropriate (semi)-strictification theorems for (braided) monoidal bicategories~\cite{MR1261589,MR2717302, MR2770448}. We will use the words ``equivalent'' and ``isomorphic'' interchangeably. 
%  With a few exceptions, 
  We will generally use the following notation to remind the ``category number'' of various constructions. Capital letters usually refer to objects of 2-categories, lower case letters to 1-morphisms in 2-categories and objects of 1-categories, and Greek letters to 2-morphisms in 2-categories and 1-morphisms in 1-categories. 
  Two exceptions to this convention are: unit objects in tensor $n$-categories are called $I$ and unit $k$-morphisms are called $\id$; the braiding natural transformation in a braided monoidal $1$-category is called $\beta$ and the braiding natural transformation in a braided monoidal $2$-category is called $\br$, even when it is evaluated on $1$-morphisms (where it is valued in $2$-morphisms).
  Composition of top-morphisms is generally $\cdot$, and composition of 1-morphisms in 2-categories is $\circ$. Tensor products are $\otimes$ for 1-categories and $\boxtimes$ for 2-categories. We do not use the conventions of \cite{DSPS}: rather, a monoidal 1-category $\cC$ determines a 2-category $\rB\cC$ with one object~$*$ such that $\Hom_{\rB\cC}(*,*) = \cC$ and $x \circ y := x \otimes y$. For braided monoidal categories $\cB$, we write $\cB^{\rev}$ for the same monoidal category $\cB$ with inverse braiding. \end{remark}

Recall that a semisimple $1$-category over our fixed characteristic-zero algebraically-closed field $\bk$ is a 1-category $\cC$ such that:
\begin{itemize}
  \item For every object $x \in \cC$, $\End_\cC(x)$ is a finite-dimensional semisimple $\bk$-algebra. We will henceforth write $\Omega_x\cC := \End_\cC(x)$, thinking of this algebra as the ``based loop space of $\cC$ at $x\in\cC$.'' In many cases there will be a distinguished object $I \in \cC$, in which case we will write simply ``$\Omega\cC$'' for its endomorphisms.
  \item $\cC$ is additively and Karoubi complete. In other words, $\cC$ has direct sums and all idempotents (endomorphisms $\rho \in \Omega_x \cC$ such that $\rho^2 = \rho$) split (factor as $\rho = \gamma\phi$ for some morphisms $\phi : x \to y$ and $\gamma : y \to x$ such that $\phi\gamma = \id_{y}$).
\end{itemize}
The set of \define{components} of a semisimple 1-category $\cC$, denoted $\pi_0 \cC$, is the set of isomorphism classes of simple objects. This is consistent with the notation $\pi_0 \cC$ of components of a semisimple 2-category described in Theorem~\ref{thm.Schur}.

A semisimple $1$-category equipped with a monoidal structure
 is called \define{multifusion} if it has finitely many isomorphism classes of simple objects and 
 all objects have (both left and right) duals, and \define{fusion} if additionally the unit object $I$ is simple, i.e.\ if $\Omega\cC \cong \bk$.

\begin{remark}\label{rem:separabilityrepeatedadvantage}
We will take repeated advantage of the fact that in characteristic zero, semisimple algebraic objects are automatically separable. This statement holds both for associative algebras, for which one needs only that  characteristic-zero fields are perfect, and for multifusion categories, for which it is a subtle theorem of \cite{MR2183279,DSPS}. For comparison, the assumption that $\bk$ is algebraically closed is a convenient simplification --- for example, it allows us to characterize simple objects as those with endomorphism ring $\bk$, and not have to mention division rings --- but it is less vital. For the proof of the \MainTheorem, the characteristic zero assumption will also appear in Proposition~\ref{prop:fundamental}, and is central to the final steps of the proof in~\S\ref{subsec:laststep}.
\end{remark}

The idea of fusion 2-categories is that they are to fusion 1-categories as fusion 1-categories are to semisimple algebras. Specifically, a \define{semisimple 2-category} over $\bk$ is a linear 2-category $\cC$ such that:
\begin{itemize}
  \item For every $X \in \cC$, $\Omega_X\cC := \End_\cC(X)$ is a multifusion 1-category.
  \item $\cC$ is additively and Karoubi complete.
\end{itemize}
Before giving the definition of (multi)fusion 2-category, it is worth elaborating a bit on the second bullet point. 

It is straightforward to define direct sums in a linear 2-category: given objects $X,Y \in \cC$, the \define{direct sum} $X \boxplus Y$, if it exists, is the unique (up to isomorphism which is unique up to 2-isomorphism) object  equipped with 1-morphisms $p_X : X \boxplus Y \to X$, $p_Y : X \boxplus Y \to Y$, $i_X : X \to X \boxplus Y$, $i_Y : Y \to X \boxplus Y$, and a direct sum decomposition $\id_{X \boxplus Y} \cong (i_X \circ p_X) \oplus (i_Y \circ p_Y)$ such that $p_X \circ i_X \cong \id_X$, $p_Y \circ i_Y \cong \id_Y$, $p_X \circ i_Y \cong 0$, $p_Y \circ i_X \cong 0$~\cite[Definition~1.1.2]{DR}. Note that this latter direct sum decomposition takes place in the additive 1-category of endomorphisms of $X \boxplus Y$. (There is an equivalent definition which also chooses the isomorphisms $p_X \circ i_X \cong \id_X$ etc.\ as data, subject to certain coherence conditions.)

In a 2-category with direct sums, if all endomorphism categories are multifusion, then also all 1-morphisms admit both adjoints.

The more interesting part is to define Karoubi completeness of 2-categories. We will give the definition from \cite{DR}, which is tuned to the case of 2-categories in which all 1-morphisms admit adjoints. (A more general definition without this assumption is given in~\cite{GJFcond}; Theorem~3.3.3 therein confirms that the two notions give equivalent definitions of Karoubi completeness when all 1-morphisms have adjoints.) Recall that a \define{monad} on an object $X \in \cC$ is a unital and associative algebra object $p \in \Omega_X\cC$. A monad is \define{separable} when the multiplication map $\mu : p \circ p \to p$, which is a map of $p$-bimodules by associativity, splits $p$-bilinearly. 
Separable monads will be our categorification of idempotents. 

As an example, suppose that $f : X \to Y$ is a 1-morphism with right adjoint $f^* : Y \to X$, and suppose that the counit $\ev_f : f \circ f^* \to \id_Y$ admits a section (i.e. a $2$-morphism $\nu: \id_Y \to f \circ f^*$ such that $\ev_f \cdot \nu = \id_{\id_Y}$). The corresponding monad $p := f^* \circ f$, with multiplication map
$$ p \circ p \cong f^* \circ (f\circ f^*) \circ f \to[\ev_f] f^* \circ \id_Y \circ f \cong p$$
is automatically separable. A separable monad \define{splits} when it factors in this way (see~\cite[Definition 1.3.3]{DR}), and a 2-category with adjoints is \define{Karoubi complete} when its hom-categories are Karoubi complete and all separable monads split.

\begin{example}\label{eg.SigmaB}
  A \define{pointing} of a 2-category $\cC$ is a choice of object $X \in \cC$. The assignment $(\cC, X) \mapsto \Omega_X\cC$ from pointed semisimple 2-categories to multifusion 1-categories has a left adjoint, which we will call $\Sigma$ and think of as a ``suspension'' construction. In other words, for any multifusion $1$-category $\cX$, there is a pointed semisimple 2-category $\Sigma\cX$, constructed by first taking the ``one-object delooping'' $\rB \cX$ (i.e.\ the 2-category with one object $*$ and $\End(*) = \cX$) and then Karoubi-completing $\rB \cX$ by formally splitting all separable monads.
  (In \cite{DR}, this latter Karoubi-completion step is denoted $(-)^\triangledown$, so that $\Sigma\cX = (\rB\cX)^\triangledown$.)
  
   This formal construction can be made much more explicit \cite[\S A.5]{DR}: $\Sigma\cX$ is the ``Morita 2-category'' of separable algebra objects in $\cX$ and their bimodule objects. Theorem~1 of \cite{MR1976459} provides yet another model of $\Sigma\cX$: it is (equivalent to) the 2-category $\Modr{\cX}$ of finite semisimple right $\cX$-module categories \cite[1.3.13]{DR}, pointed by the rank-one free module.
  From this (or any other) explicit description, one finds that the unit $\Omega\Sigma\cX \to \cX$ of the adjunction $\Sigma \dashv \Omega$ is an equivalence of fusion 1-categories.
  
  Note that if $A \in \cX$ is a separable algebra object, then it is the category $\Modl[\cX]{A}$ of \emph{left} $A$-module objects in $\cX$ which is a \emph{right} $\cX$-module, i.e.\ an object of $\Sigma\cX = \Modr{\cX}$. The notation is selected to remind who is acting on which side. When we model $\Sigma\cX$ in terms of Morita categories, we will occasionally write $[A]$ for the Morita class represented by the algebra $A$.
\end{example}

\begin{example}\label{eg.2vec}
  A vitally important example of a finite semisimple 2-category is $2\Vec := \Sigma\Vec$. It has many equivalent descriptions, including as the 2-category of finite semisimple 1-categories and as the 2-category of finite dimensional semisimple algebras and bimodules. Further models can be found in the appendix to \cite{III}.
\end{example}

\begin{remark}
\label{notation:langlerangle} 
By definition, an object $X \in \cC$ is called \define{simple} if $\Omega_X\cC$ is fusion, i.e.\ when $\Omega^2_X\cC := \End_{\Omega_X\cC}(\id_X) \cong \bk$. (There is a more general notion of simplicity of an object in a 2-category developed in \cite{DR}, but it is equivalent to this notion in the case of semisimple 2-categories.) 
Analogous to semisimple $1$-categories, any object $X$ in a semisimple $2$-category decomposes into a finite direct sum of such simples,  and an object is simple if and only if it is indecomposable \cite[Proposition 1.2.17]{DR}.

For a $2$-endomorphism $\alpha \in \End(f)$ of a simple $1$-morphism $f$ in a semisimple $2$-category $\cC$, we define $\langle \alpha \rangle \in \bk$ to be the proportionality factor such that $\langle \alpha \rangle \id_f = \alpha$. In particular, this will appear for $2$-endomorphisms $\alpha \in  \Omega^2_X\cC$ of 
a simple object $X \in \cC$. 
\end{remark}

\begin{remark}
\label{remark.2catDirectSums}
A direct sum decomposition $X\cong \boxplus_{i \in I} X_i$ of an object (into not necessarily simple summands) induces a family of (not necessarily irreducible) ``mutually orthogonal'' idempotents in the semisimple commutative algebra $\Omega^2_X\cC$, i.e.\ idempotents $\{p_i\}_{i \in I}$ such that $p_i p_j =0$ for $i \neq j$ and $\sum_{i \in I} p_i = \id_{\id_X}$. Conversely,  any such family of idempotents in $\Omega^2_X \cC$ induces a direct sum decomposition of $X$ (by splitting the idempotents and then the induced idempotent monads $P_i \in  \Omega_X \cC$, see~\cite[Proposition 1.3.16]{DR}). A summand $X_i$ is simple if and only if the corresponding idempotent $p_i$ is irreducible.

A particular consequence is that the decomposition of an object into simples is unique in a very strong sense (see the proof of~\cite[Proposition 1.2.21]{DR}): For any pair of direct sum decompositions  $\{( X_j, p_j, i_j, \ldots)\}_{j \in J} $ and $\{(X'_j, p'_j, i'_j, \ldots)\}_{j \in J'}$ of $X$ into (not necessarily inequivalent) simples there is a unique bijection $f:J \to J'$ such that $p'_{f(j)}\circ i_j: X_j \to X'_{f(j)}$ is an equivalence and $p'_{k} \circ i_j\cong 0$ for all $k \neq f(j)$. 
In particular, fixing a direct sum decomposition of $X$ into simples $\boxplus_i X_i$, it follows that with respect to this decomposition any $1$-automorphism $f:X \to X$ is the composite of a permutation of summands followed by a ``diagonal matrix of automorphisms'' $\boxplus_i (f_i: X_i \to X_i)$.
\end{remark}

A semisimple 1-category has finitely many isomorphism classes of simple objects exactly when it is equivalent to the category of modules of a finite-dimensional semisimple algebra. By the same token, a semisimple 2-category $\cC$ will be ``finite'' when it is equivalent to $\Sigma\cX$ for some multifusion 1-category $\cX$. Since our ground field $\bk$ is algebraically closed and of characteristic zero, Theorems~1.4.8 and~1.4.9 of \cite{DR} (which rely in turn on the finiteness result of \cite[Corollary~9.1.6]{EGNO}) assert that this finiteness happens exactly when $\cC$ has finitely many isomorphism classes of simple objects.

However, ``finitely many isomorphism classes of simple objects'' is not the morally correct definition of ``finiteness'' for semisimple 2-categories. This is because, if $\bk$ is not algebraically closed, demanding finitely many isomorphism classes of simple objects conflicts with demanding Karoubi completeness. (Indeed, in the simplest case $\Sigma\Vec$, there is an isomorphism class of simple objects for every division ring over $\bk$.) The underlying reason for this is that the usual Schur's lemma, which says that a nonzero map between simple objects is an isomorphism, fails in 2-categories. Rather, we have the following weaker statement \cite[Proposition 1.2.19]{DR}:
\begin{theorem}[2-categorical Schur's Lemma]\label{thm.Schur}
  In a semisimple 2-category $\cC$, the relation ``$X \sim Y$ if $\Hom(X,Y) \neq 0$'' on the set of simple objects is an equivalence relation. The equivalence classes of this relation are the \define{components} of $\cC$, denoted $\pi_0 \cC$. 
  \qed
\end{theorem}

For semisimple 2-categories, the definition of \define{finiteness} that works without assuming algebraic closedness is to demand that there be finitely many components.

\begin{definition}\label{defn:generator}
  A \define{generator} of a finite semisimple 2-category $\cC$ is an object $X \in \cC$ satisfying any, hence all, of the following equivalent conditions:
  \begin{enumerate}
  \item For every non-zero object $Y$ in $\cC$, there exists a non-zero $1$-morphisms $Y\to X$.
\item A direct sum decomposition of $X$ contains indecomposable summands from all components of $\cC$. 
\item The canonical $2$-functor $\Modr{(\End_{\cC}(X))}=\Sigma \Omega_X \cC \to \cC$ is an equivalence. \label{condition:2functor}
  \end{enumerate}
\end{definition}

\begin{example}
  Given a multifusion 1-category $\cX$, the semisimple 2-category $\Sigma\cX \cong \Modr{\cX}$ is generated by the rank-1 free module ``$\cX_\cX$''; condition~\ref{condition:2functor} says that any choice of generator of $\cC$ identifies it with a 2-category of this type  \cite[Theorem 1.4.9]{DR}. In the literature on multifusion 1-categories, the components of $\Sigma\cX$ are called the \define{blocks} of $\cX$. 
\end{example}

\begin{definition}
  A \define{multifusion 2-category} is a finite semisimple 2-category $\cA$ equipped with a monoidal structure for which all objects have duals.
  (There are various ways to present the notion of duals in a monoidal 2-category. We will review one coherent presentation in Definition~\ref{defn:duals} in \S\ref{subsec:inv}, as we will need the details then.)
   A multifusion 2-category is \define{fusion} when the unit object $\one \in \cA$ is simple. A \define{connected fusion 2-category} is a fusion 2-category with only one component.
\end{definition}

Any fusion 2-category $\cA$ has living inside of it a braided fusion 1-category $\Omega\cA$.

\begin{proposition}[{\cite[Construction 2.1.19 and Remark 2.1.22]{DR}}]\label{prop.connectedfusion}
  If $\cB$ is a braided fusion 1-category, then $\Sigma\cB$ is a connected fusion 2-category, and all connected fusion 2-categories are of this form: if $\cA$ is connected fusion, then the counit $\Sigma\Omega\cA \to \cA$ of the adjunction $\Sigma \dashv \Omega$ is an equivalence of fusion 2-categories. \qed
\end{proposition}
In other words, the collection of connected fusion 2-categories, which is a priori a 3-category, is in fact canonically equivalent to the 2-category of braided fusion 1-categories.

\begin{remark}\label{rem:tensorconventions}
We use Proposition~\ref{prop.connectedfusion} to fix our convention on the monoidal structure of $\Modr{\cB} \cong \Sigma \cB$: The tensor product of two separable algebras $A$ and $B$ in $\cB$ is the algebra $A\otimes B$ with multiplication $(\mu_A \otimes \mu_B)\cdot(\id_A \otimes \beta_{A,B}\otimes \id_B)$ (also see~\eqref{eq:tensor}).
\end{remark}

\begin{example} \label{eg:linearization}

An important source of examples of semisimple and fusion 2-categories comes from linearizing (higher) groupoids. Given a finite 1-groupoid $\cG$, we produce a semisimple 2-category $2\Vec[\cG]$ as follows: Consider $\cG$ as a ``discrete'' $2$-groupoid with only identity $2$-morphisms and $\bk$-linearize the $2$-morphisms sets, i.e.\ allow $\bk$-linear multiples of identities and set the $2$-Hom spaces between non-isomorphic $1$-morphisms to be zero. 
Now (1- and 2-categorically) Karoubi and direct sum complete this 2-category. Equivalently, $2\Vec[\cG]$ is the 2-category of $2$-functors $\cG \to 2\Vec$ (where we consider $\cG$ as a discrete $2$-groupoid as above).

  This construction can be ``twisted'' by Postnikov classes. Suppose that $\alpha$ is (a cocycle for) a class in $\H^3(\cG; \bk^\times)$. This choice selects a 2-groupoid $\widetilde\cG^\alpha$ which is built as an extension $K(\bk^\times,2) \to \widetilde\cG^\alpha \to \cG$. Take this 2-groupoid, and produce a $\bk$-linear 2-category by extending the 2-morphisms from $\bk^\times$ to $\bk$ (see~\cite[Construction 2.1.16]{DR} for details); then (1- and 2-categorically) Karoubi and direct sum complete. We will call the result $2\Vec^\alpha[\cG]$.
   Up to a sign convention on how to handle Postnikov extension data, $2\Vec^\alpha[\cG]$ can also be described as the 2-category of functors $F : \widetilde\cG^\alpha \to 2\Vec$ which are linear for the action of $\bk^\times$ on 2-morphisms.
  As an example, take $\cG = \rB G$ for a finite group $G$. Then  the cocycle~$\alpha$ selects a fusion 1-category $\Vec^\alpha[G]$, and $2\Vec^\alpha[\rB G] = \Sigma(\Vec^\alpha[G])$.
  
  Suppose now that $\cG$ is a monoidal groupoid. Then the 2-category $2\Vec[\cG]$ will naturally inherit a monoidal structure. It will be fusion provided $\cG$ is \define{groupal}: the monoid structure on $\pi_0 \cG$ should be a group. Groupal monoidal structures on $\cG$ are the same as choices of delooping $\rB \cG$: i.e.\ $\rB \cG$ should be a pointed connected homotopy 3-type equipped with an equivalence $\Omega\rB \cG \isom \cG$.
  
  For the twisted version, recall that the functor $\Omega$ of taking based loops induces a desuspension map $\Omega : \H^\bullet(\rB\cG; \bk^\times) \to \H^{\bullet - 1}(\cG; \bk^\times)$. A \define{delooping} of a Postnikov class is a choice of preimage for it along this desuspension map. Specifically, suppose that we have chosen a delooping $\rB \cG$ of $\cG$ and a class $\alpha \in \H^4(\rB\cG; \bk^\times)$. Then there is a multifusion 2-category $2\Vec^\alpha[\cG]$ whose underlying linear 2-category is $2\Vec^{\Omega\alpha}[\cG]$. (Note the mild abuse of notation: we will write the full delooped Postnikov class in the superscript, but write plain $\cG$ rather than $\rB\cG$ in brackets.)
  
  To enhance to a braided monoidal structure requires simply delooping further: one must choose a pointed connected simply-connected homotopy 4-type $\rB^2\cG$ such that $\Omega^2\rB^2\cG \cong \cG$, and a class $\alpha \in \H^5(\rB^2\cG; \bk^\times)$.

  It is worth emphasizing that linearization does nothing to components: the canonical map $\pi_0 \cG \to \pi_0 2\Vec^\alpha[\cG]$ is a bijection. However,  typically $2\Vec[\cG]$ will have more isomorphism classes of simple objects than just the objects of $\cG$ (and in particular will have non-invertible simple $1$-morphisms between these objects) \cite[Remark 1.4.23]{DR}. Moreover, $\cG$ is usually not canonically determined by its linearization. To give a fusion 2-category $\cC$ a groupal presentation, every component of $\cC$ must contain at least one invertible object (inducing a canonical group structure on $\pi_0 \cC$), $\pi_0 \Omega\cC$ must be a group, and moreover one must choose a splitting of the group homomorphism $\{$invertible objects up to equivalence$\} \to \pi_0 \cC$ \cite[Remark 2.1.17]{DR}.

  If $\cC$ is braided and admits a groupal presentation monoidally, then it admits one braided monoidally, because the
 inclusion $\widetilde\cG^\alpha \to 2\Vec^\alpha[\cG]$ is fully faithful on \emph{invertible} 1- and 2-morphisms. 
  \end{example}

\begin{example}\label{eg:sVeclin}
  The symmetric fusion 1-category $\sVec$ arises from linearization since its simple objects $I$ and $e$ are both invertible. Indeed, $\sVec = \Vec^\alpha[\bZ_2]$, where the twisting~$\alpha$ is the stable cohomology class $(-1)^{\Sq^2} \in \H^{\bullet + 2}(K(\bZ_2,\bullet); \bk^\times)$. The stability of $\alpha$ reflects the symmetry of $\sVec$.
  It follows that $\Sigma\sVec = 2\Vec^\alpha[\rB\bZ_2]$ for the same twisting $\alpha = (-1)^{\Sq^2}$. 
  
  Although $\Sigma\sVec$ is connected, it contains two simple objects up to equivalence. 
  To study them, we model $\Sigma\sVec$ as the 2-category of finite-dimensional separable superalgebras and their bimodules (see Example~\ref{eg.SigmaB}). Then the simple objects are the unit object $I = [\bk]$ and the Clifford algebra $C := [\operatorname{Cliff}(1,\bk)]$. (Since $\bk$ is algebraically closed and of characteristic zero, these are the only two simple super algebras up to Morita equivalence.) The 1-categories $\Hom_{\Sigma\sVec}(I,C)$ and $\Hom_{\Sigma\sVec}(C,I)$ are both equivalent to $\Vec$, and we will, slightly abusively, write $v$ for the simple object in both categories; it corresponds to the unique simple (left or right) $\operatorname{Cliff}(1,\bk)$-supermodule.
    
  The fusion rules in $\Sigma\sVec$ are:
  \begin{equation*}
    C^2 \cong I, \qquad
    e^2 \cong \id, \qquad ev \cong ve \cong v, \qquad v^2 \cong \id \oplus e.
  \end{equation*}
  To choose the first of these isomorphisms requires selecting $\sqrt{-1} \in \bk$. In fact, in the Morita category of real superalgebras, $[\operatorname{Cliff}(1,\bR)]$ famously has order~$8$, not~$2$. (The two choices of $\sqrt{-1}$ provide isomorphisms $C^2 \cong I$ which differ by a factor of $e$. )
  
  We will not list choices for all of the associator and higher coherence data, but we will record the self-braidings of the object $C$ and the $1$-morphism $e$:
   \begin{align}
   \label{eqn.halfbraide} \beta_{e,e} =  -1~\id_{e\otimes e} &: e \otimes e \to e \otimes e ,\\
   \label{eqn.halfbraidc} \br_{C,C} \cong e \boxtimes\id_{C\boxtimes C}  &: C \boxtimes C \to C \boxtimes C .
  \end{align}
  The braiding \eqref{eqn.halfbraide} is a restatement that the twisting is $\alpha = (-1)^{\Sq^2}$ in the linearization presentation of $\sVec$. The braiding \eqref{eqn.halfbraidc} is a fun exercise in superalgebra (a solution can be found in \cite[\S5.5]{MR3978827}). 
\end{example}

\subsection{Drinfel'd centres of fusion 2-categories}\label{subsec.Zabstract}

As in the 1-categorical case, an important tool when analyzing a fusion 2-category, and hence, by Proposition~\ref{prop.connectedfusion}, a braided fusion 1-category, is the Drinfel'd centre. Drinfel'd centres of monoidal 2-categories were first introduced in \cite{MR1402727}. An excellent summary appears in Section~2.3 of~\cite{DN}.

Let us briefly outline the abstract definition. Suppose that $\cA$ is a monoidal 2-category. The \define{Drinfel'd centre} $\cZ(\cA)$ is the braided monoidal 2-category of endomorphisms of $\cA$ thought of as an $\cA$-bimodule. It has the following explicit description. An object of $\cZ(\cA)$ consists of an object $A \in \cA$ together with a \define{half-braiding} isomorphism $h_A|_X : A \boxtimes X \isom X \boxtimes A$ which depends naturally and monoidally on $X \in \cA$. Note that both naturality and monoidality are data: for example, naturality is the data of an iso-2-morphism $h_A|_X \circ (f \boxtimes \id_A) \cong (\id_A \boxtimes f) \circ h_A|_Y$ for every 1-morphism $f : X \to Y$, and this data must be compatible for composition of 1-morphisms; after suppressing associator and unitor data, monoidality consists of, for every pair of objects $X,Y \in \cA$, an iso-2-morphism $h_A|_{X \boxtimes Y} \cong (\id_X \boxtimes h_A|_Y) \circ (h_A|_X \boxtimes \id_Y)$, which must satisfy its own monoidality and naturality conditions. A 1-morphism $(A,h_A) \to (B,h_B)$ consists of a 1-morphism $f : A \to B$ in $\cA$ together with a natural and monoidal iso-2-morphism $\gamma_f|_X : (\id_X \boxtimes f) \circ h_A|_X \cong h_B|X \circ (f \boxtimes  \id_X)$, and a 2-morphism $(f,\gamma_f) \Rightarrow (g,\gamma_g)$ is a 2-morphism $f \Rightarrow g$ satisfying the obvious compatibility condition. (For 1-morphisms $(f,\gamma_f)$, naturality and monoidality of $\gamma_f$ are properties, not data, and for 2-morphisms they are no conditions at all.)

We are primarily interested in the case when the monoidal 2-category $\cA = \Sigma\cB$ is connected fusion, corresponding to a braided fusion 1-category $\cB$.
In this case, Davydov and Nikshych \cite{DN} provide a rather explicit description of $\cZ(\Sigma\cB)$ (and we will give an even more hands-on description in \S\ref{subsec.halfbraidedalgs}). 
Specifically, \cite{MR3248737,MR3874702} introduce, up to  left/right conventions, a $2$-category 
$\BrModr{\cB}$
 of \define{(right) braided module categories} over $\cB$: 
\begin{itemize}
\item A braided right module category of $\cB$ is a finite semisimple right $\cB$-module category $(\cM, \;* : \cM \boxtimes \cB \to \cM, \;\dots)$ (the $\dots$ stands for unitor and associator data that we henceforth suppress) equipped with a \define{$\cB$-module braiding}: a natural isomorphism $\{\sigma_{m, x}: m * x \to m * x\}_{x\in \cB, m \in \cM}$ such that $\sigma_{m, I}$ is trivialized by the unitor data and satisfying the following coherence conditions:
\[\begin{tikzcd}
	{(m*x)*y \cong m*(x\otimes y)} && {m* (y \otimes x) \cong (m*y) *x} \\
	{ (m*x) *y \cong m*(x\otimes y)} && {m*(y \otimes x) \cong (m*y) *x }
	\arrow["{m*\br_{x,y}}", from=1-1, to=1-3]
	\arrow["{\sigma_{m,y}*x}", shift left=10, from=1-3, to=2-3]
	\arrow["{\sigma_{m*x, y}}"',shift right=10, from=1-1, to=2-1]
	\arrow["{m*\br^{-1}_{y,x}}"', from=2-1, to=2-3]
\end{tikzcd}\]
\[\begin{tikzcd}[%
   %,row sep = 0ex
    %,/tikz/column 1/.append style={anchor=base east}
    ,/tikz/column 3/.append style={anchor=base east}
    ]
	m*(x\otimes y)  &m*(y\otimes x) \cong (m*y)*x& (m*y)*x \\
	{ m*(x\otimes y)} && {m*(y \otimes x) \cong (m*y) *x }
	\arrow["{m*\br^{-1}_{y,x}}", from=1-1, to=1-2]
	\arrow["\sigma_{m,y}*x", from= 1-2, to=1-3]
	\arrow["{\sigma_{m*y,x}}", from=1-3, to=2-3, shift right=10,start anchor =south east, end anchor=north east]
	\arrow["{\sigma_{m,x\otimes y}}"', from=1-1, to=2-1]
	\arrow["{m*\br_{x,y}}"', from=2-1, to=2-3]
\end{tikzcd}\]
We refer the reader to \cite[Section 5.2]{MR3248737} for graphical interpretation of these data.
\item A \emph{braided module functor} $(\cM, *, \sigma) \to (\cM' , *', \sigma')$ is a $\cB$-module functor $(f:\cM \to \cM', \phi_{m,x}: f(m * x) \cong f(m) *' x)$ such that $\phi_{m,x} \cdot f(\sigma_{m,x}) = \sigma'_{f(m),x}\cdot \phi_{m,x} $.
\item A transformation of braided module functors is simply a natural transformation of the underlying $\cB$-module functors.   
\end{itemize}
For example, any braided monoidal functor of braided fusion 1-categories $F:\cB \to \cC$ equips $\cC$ with the structure of a braided $\cB$-module category with module braiding \[\sigma_{c, b} = \beta^\cC_{F(b), c} \cdot \beta^\cC_{c, F(b)}:  c\otimes F(b) \to c\otimes F(b),\] where $\beta^\cC$ denotes the braiding of $\cC$.

The relationship between braided modules and the 2-categorical Drinfel'd centre is:
\begin{theorem}[{\cite[Theorem 4.10]{DN}}]\label{thm.zmod}
  $\cZ(\Sigma\cB) \cong \BrModr\cB$.
\end{theorem}
For another perspective on this equivalence, see Theorem~\ref{thm:HBADrinfeld}.
\begin{proof}
  The paper \cite{DN} works with left $\cB$-module categories, and proves that $\cZ(\Modl \cB) = \BrModl\cB$, the braided monoidal 2-category of braided left $\cB$-modules, whereas for our conventions, $\Sigma\cB$ is equivalent to the 2-category $\Modr\cB$ of finite semisimple right $\cB$-module categories.
  Hence, Theorem~\ref{thm.zmod} either follows by mirroring Davydov-Nikshych's proof for right modules, or by using the braiding to produce a monoidal equivalence $\Modr{\cB}\cong \Modl{\cB}$ which lifts to a braided monoidal equivalence $\BrModr{\cB} \cong \BrModl{\cB}$ (where we use our conventions from Remark~\ref{rem:tensorconventions} for the monoidal structure of $\Modr{\cB}$ and Davydov-Nikshych's conventions~\cite[\S 3.3]{DN} for the monoidal structure of $\Modl{\cB}$). \end{proof}

As an example, the unit object in $\cZ(\Sigma\cB)$ corresponds to the ``rank-1 free module'' $\cB$ as a braided $\cB$-module. The action is the tensor product $m * x = m \otimes x$, and the module braiding is the full braiding $\sigma_{m,x} = \beta_{x,m} \circ \beta_{m,x} : m \otimes x \to x \otimes m \to m \otimes x$. In particular, an endomorphism of $\cB$-as-a-braided-$\cB$-module is an element of $\cB$ (i.e.\ an endomorphism of $\cB$ as an unbraided $\cB$-module) which is invisible to the full-braiding. In other words:

\begin{lemma} \label{lemma.centres}
  If $\cB$ is a braided fusion 1-category, then $\Omega\cZ(\Sigma\cB) = \cZ_{2}(\cB)$. \qed
\end{lemma}

\begin{remark}\label{rem:brmodfullyfaithful}
  The forgetful functor $\cZ(\Sigma\cB) \cong \BrModr{\cB} \to \Modr \cB \cong \Sigma\cB$ is fully faithful on $2$-morphisms: being a braided module functor is a property of a module functor and every module natural transformation between these is allowed. 
  In particular, any $\cB$-module summand of a braided $\cB$-module category is also a braided $\cB$-module category and a  braided $\cB$-module category is indecomposable iff it is indecomposable as a $\cB$-module category.
  
  For general non-connected fusion $2$-categories $\cA$, the canonical map $\cZ(\cA) \to \cA$ is faithful on $2$-morphisms but not full.
  The additional connectedness of the map $\cZ(\Sigma\cB) \to \Sigma\cB$ comes from connectedness of $\Sigma \cB$ itself. 
\end{remark}

For fusion 1-categories $\cC$, $\cZ(\cC)$ being fusion is equivalent to separability of $\cC$. Reiterating Remark~\ref{rem:separabilityrepeatedadvantage}, we will repeatedly use that separability for fusion 1-categories is automatic in characteristic zero. In particular, we expect that $\cZ(\cA)$ will be a fusion 2-category whenever $\cA$ is. For the purposes of this paper we will need only:

\begin{lemma} $\cZ(\Sigma \cB)$ is a (braided) fusion $2$-category. 
\end{lemma}
\begin{proof}  Additivity and Karoubi completeness, as well as dualizability of objects and $1$-morphisms, follows directly from the respective properties of $\Sigma \cB$. Remark~\ref{rem:brmodfullyfaithful} shows that $\cZ(\Sigma \cB)$ is locally finite semisimple.
 It therefore suffices to show that $\pi_0\cZ(\Sigma \cB)$ is finite. This follows from Corollary~\ref{cor.Lgenerates} which shows that the braided module category $\cZ(\cB)$ is a generator for $\cZ(\Sigma \cB)$.
\end{proof}

We will study $\pi_0 \cZ(\Sigma\cB)$ in more detail in \S\ref{subsec.compZSB}--\ref{sec:Smatrix}, which occur after the proof of Corollary~\ref{cor.Lgenerates}. We will not use finiteness of $\pi_0\cZ(\Sigma \cB)$ until then.

\subsection{Higher Drinfel'd centres control nondegenerate extensions}\label{subsec.BCthm}

We henceforth adopt the following language. Let $\cB$ be a braided monoidal (1-)category. A braided monoidal category \define{under} $\cB$ is a braided monoidal category $\cM$ equipped with a braided monoidal functor $i: \cB \to \cM$. 
Of course, it is an \define{extension} when $i$ is fully faithful, and a \define{nondegenerate extension} when additionally $\cM$ is nondegenerate. An isomorphism of braided monoidal categories $i: \cB \to \cM$ and $i': \cB \to \cM'$ under $\cB$  is a braided monoidal equivalence $F:\cM \to \cM'$ and a monoidal natural isomorphism $F\circ i  \cong i' $. An isomorphism of nondegenerate extensions is an isomorphism between the underlying braided monoidal categories under $\cB$. 

The connection between 2-categorical Drinfel'd centres and (minimal nondegenerate) extensions is explained by the following theorem, which forms the main result of this section. Versions of this connection were proposed in~\cite{Kong2020}, in the lecture \cite{DmitriMSRI}, and motivated \cite{DN}.

\begin{theorem} \label{thm.BCs}
  Let $\cB$ and $\cC$ be braided fusion 1-categories. If $\cZ(\Sigma\cB) \cong \cZ(\Sigma\cC)$, then there is a nondegenerate extension $i : \cB \hookrightarrow \cM$ and an equivalence $\cC \cong \cZ_{2}(\cB \subset \cM)^{\rev}$ (in which case there is also an equivalence $\cB \cong \cZ_{2}(\cC\subset \cM)^{\rev}$ by \cite[Theorem~8.21.1]{EGNO}).
\end{theorem}

\begin{remark} \label{rem:BCs}
A higher Morita categorical explanation of Theorem~\ref{thm.BCs} is given in~\cite[\S2.3]{2011.11165}. Indeed, by working with  Morita 5-categories, \cite{2011.11165} argues that the 2-groupoid consisting of the equivalences $\cZ(\Sigma\cB) \cong \cZ(\Sigma\cC)$ is equivalent to the 2-groupoid consisting of the nondegenerate extensions $\cB \subset \cM$ with $\cC \cong \cZ_{2}(i : \cB \to \cM)^{\rev}$ as in the Theorem. In particular, the Theorem is if and only if.
For our \MainTheorem\ we only use the ``if'' direction of Theorem~\ref{thm.BCs}, and we will give an elementary proof which only uses braided fusion $2$-category theory.
\end{remark}

To set up the proof, recall that  (braided) monoidal objects, aka $E_1$ ($E_2$)-algebras, may be defined in any (braided) monoidal $2$-category $\cC$. 
Unpacked, this amounts to an object $A$ in $\cC$ equipped with multiplication and unit $1$-morphisms $m:A\boxtimes A \to A$  and $u:I \to A$ and associator, unitor and braiding $2$-isomorphisms 
\begin{gather*}
 \alpha: m\circ (m\boxtimes \id_A) \To m \circ (\id_A \boxtimes m) ,\qquad \lambda: m \circ (u \boxtimes \id_A) \To m ,\\ \rho: m\circ(\id_A \boxtimes u) \To m ,\qquad  \beta: m\circ \br_{A,A} \To m ,
\end{gather*}
fulfilling the familiar coherence conditions. Note that any (braided) monoidal structure on an object $A \in \cC$ induces a (braided) monoidal structure on the $1$-category $\Hom_{\cC}(I, A)$. 

\begin{definition} \label{defn:rigidobject}
For a monoidal object $A$ in a monoidal $2$-category, the associator $2$-isomorphism equips the multiplication $1$-morphism $m:A\boxtimes A\to A$ with the structure of an $A$--$A$ bimodule $1$-morphism. 
We say that $A$ is \define{rigid} if $m$ has a right adjoint $\Delta:A \to A \boxtimes A$ as an $A$--$A$ bimodule $1$-morphism. This name is suggested by \cite[Appendix D]{MR3381473} and \cite[Definition-Proposition 1.3]{BJS}; see also \cite[Proposition 4.8]{III}, which explains the relation to the usual notion of rigidity.
\end{definition}
\begin{remark} There is also a notion of \emph{separability} of a monoidal object $A$ in a monoidal $2$-category: $A$ is rigid and the counit $\ev_{m} : m\circ \Delta \To \id_A$ of the adjunction $m\dashv \Delta$  admits a section as an $A$--$A$ bimodule $2$-morphism. In characteristic zero, we expect that every rigid monoidal object in a fusion $2$-category is automatically separable, generalizing the fact that every multifusion $1$-category is automatically separable. Hence, in our context, ``rigidity'' and ``separability'' should be treated as synonymous. 
\end{remark}
\begin{lemma}\label{lemma:rigidextensions}
Consider the case where $\cC = \Bim{\cX} \cong \Modr{(\cX \boxtimes \cX^{\mathrm{mp}})}$ (where $\cX^{\mathrm{mp}}$ denotes $\cX$ with the opposite monoidal structure) is the multifusion $2$-category of finite semisimple $\cX$-bimodules for some multifusion $1$-category $\cX$. (That this semisimple 2-category is multifusion follows from \cite[Theorem 3.4.3]{DSPS}.) Monoidal objects $A \in \Bim{\cX}$ correspond to finite semisimple monoidal categories $\cY$ under $\cX$. 
  We claim that $A$ is rigid in the sense of Definition~\ref{defn:rigidobject} if and only if $\cY$ is rigid as a monoidal category.
\end{lemma}

\begin{proof}
Unpacked, we are claiming the following. Consider the balanced tensor product $\cY \boxtimes_\cX \cY$. It is an $\cX$-bimodule, and the multiplication bifunctor $m : \cY \boxtimes \cY \to \cY$ factors through the quotient map $q : \cY \boxtimes \cY \to \cY \boxtimes_\cX \cY$ and
   an ``$\cX$-balanced multiplication'' $m_\cX : \cY \boxtimes_\cX \cY \to \cY$. The claim is that $m$ has a $\cY$-bilinear right adjoint if and only if $m_\cX$ does.
   
  In one direction, suppose that $m$ has a right adjoint as a $\cY$-bimodule map, so that $\cY$ is rigid and hence multifusion. Then $\Bim{\cY}$ is a multifusion $2$-category, and the 1-morphism $m_\cX$ therein must be adjunctible. 
  
  In the other direction, for any modules $\cM_\cX \in \Modr{\cX}$ and $_\cX\cN \in \Modl{\cX}$, \cite[Proposition 3.4.1]{DSPS} provides a right adjoint to the quotient map $\cM \boxtimes \cN \to \cM \boxtimes_\cX \cN$ which is bilinear for the left and right actions by the multifusion categories $\End(\cM_\cX)$ and $\End(_\cX\cN)$. In particular, by restricting this bilinearity data along $\cY \to \End(\cY_\cX)$ or $\End(_\cX\cY)$, we find a $\cY$-bilinear right adjoint for $q : \cY \boxtimes \cY \to \cY \boxtimes_\cX \cY$. But $m \cong m_\cX \circ q$, and so if $m_\cX$ is $\cY$-bilinearly right-adjunctible, then $m$ is as well with $m^R \cong q^R \circ m_\cX^R$.
\end{proof}

\begin{definition}\label{def:internalmueger}
Let $A$ be a braided monoidal object in a braided monoidal $2$-category $\cC$. Each $1$-morphism $f:I \to A $ determines a $1$-endomorphism of $A$ by left multiplication, namely $m\circ (f \boxtimes \id_A) : A\to A$. 
Then there are two interesting 2-isomorphisms $m\circ (f \boxtimes \id_A) \cong m \circ \br_{A,A} \circ \br_{A,A} \circ (f \boxtimes \id_A)$:
\begin{itemize}
 \item The (unitor data for the) ambient braided monoidal structure on $\cC$ provides an isomorphism $f \boxtimes \id_A \cong \br_{A,A} \circ (\id_A \boxtimes f) \cong \br_{A,A} \circ \br_{A,A} \circ (f \boxtimes \id_A)$. \item The braiding on $A$ provides an isomorphism $m \cong m \circ \br_{A,A} \cong m \circ \br_{A,A} \circ \br_{A,A}$. \end{itemize} We will say that $f$ is \define{transparent} if these two isomorphisms are equal.
We define the \emph{M\"uger centre} of a braided monoidal object $A$ in $\cC$ to be the full subcategory of $\Hom_{\cC}(I, A)$ on the transparent $1$-morphisms. 
\end{definition}

\begin{remark}\label{rem:braidedmoduleobject}
In fact, the M\"uger centre of $A$ is the endomorphisms of the tensor unit in the braided monoidal $2$-category $\BrModr[\cC]{A}$ of \emph{braided right module objects} of $A$, i.e.\ module objects $\mathrm{act}: M \boxtimes A \to M$ in $\cC$ equipped with a $2$-isomorphism $ \sigma: \mathrm{act} \circ \br_{A,M} \circ \br_{M , A} \To \mathrm{act}$ fulfilling coherence conditions analogous to \S\ref{subsec.Zabstract}. Braided module objects categorify the \define{dyslectic modules} (sometimes also called ``local modules'') of a commutative algebra in a braided monoidal $1$-category~\cite{MR1315904}.\end{remark}

Note that the M\"uger centre of the braided monoidal object $A$ (in the sense of Definition~\ref{def:internalmueger}) is a full subcategory of the M\"uger centre of the braided monoidal $1$-category $\Hom_{\cC}(I, A)$. The latter M\"uger centre typically has more objects than the former, because the latter tests transparency only against 1-morphisms $I \to A$, whereas the former tests against all 1-morphisms with codomain $A$.

\begin{definition}\label{defn.Lag2alg}
A \emph{Lagrangian braided monoidal object} in a braided fusion $2$-category $\cC$ is a rigid braided monoidal object $A$ in $\cC$, fulfilling:
\begin{enumerate}
\item \label{Lag2alg:connectivity} it is \emph{strongly connected}: its unit $u:I \to A$ is a fully faithful $1$-morphism. 
(A 1-morphism $f : X \to Y$ in a 2-category is \define{fully faithful} if the induced functor $f\circ : \Hom(W,X) \to \Hom(W,Y)$ is fully faithful for all objects $W$.)
Equivalently, $u: I\to A$ is a simple summand of $A$ in a direct sum decomposition. 
\item it is \emph{nondegenerate}: the M\"uger centre of $A$ is trivial, or equivalently, any simple transparent $1$-morphism $f:I \to A$ is isomorphic to the unit $u$.  \label{Lag2alg:conditionnondeg}
\end{enumerate} \end{definition}
An isomorphism of Lagrangian braided monoidal objects is a braided monoidal isomorphism of the 
 braided monoidal objects. 
Note that Lagrangian braided monoidal objects are preserved under braided monoidal equivalences.

\begin{remark}\label{rem:Lagrangian}
The notion of a ``Lagrangian braided monoidal object'' in a braided fusion $2$-category is a categorification of the notion of a ``Lagrangian algebra'' in a braided fusion 1-category~\cite[Section 4.2]{MR3039775}. In fact, as in the 1-categorical setting, a braided monoidal object as in Definition~\ref{defn.Lag2alg} should only  be called ``Lagrangian'' if the braided fusion $2$-category $\cC$ is \emph{nondegenerate} in the sense of Theorem~\ref{thm:invertibleSmatrix}. We expect nondegeneracy of $\cC$ to be equivalent to triviality of its M\"uger sylleptic centre. In other words, $\cC$ should present a ``3+1D topological order'' as axiomatized in~\cite{1405.5858,KWZ1,2003.06663}. In this nondegenerate case, we expect condition~\ref{Lag2alg:conditionnondeg} of Definition~\ref{defn.Lag2alg} to be equivalent to the stronger condition of triviality of the $2$-category of braided $A$-module objects of Remark~\ref{rem:braidedmoduleobject}. 
%Algebras as in Definition~\ref{defn.Lag2alg} should only be called ``Lagrangian'' if the braided fusion $2$-category $\cC$ is \emph{nondegenerate} in the sense of Theorem~\ref{thm:invertibleSmatrix}. We expect nondegeneracy of $\cC$ to be equivalent to triviality of its M\"uger sylleptic centre. In other words, $\cC$ should present a ``3+1D topological order'' as axiomatized in~\cite{1405.5858,KWZ1,2003.06663}. In this nondegenerate case, we expect condition~\ref{Lag2alg:conditionnondeg} of Definition~\ref{defn.Lag2alg} to be equivalent to the stronger condition of triviality of the $2$-category of braided $A$-module objects of Remark~\ref{rem:braidedmoduleobject}. 

(The term ``Lagrangian'' comes from analogy with symplectic geometry: The braiding on $\cC$ is analogous to a presymplectic form on a vector space $V$, its nondegeneracy  is analogous to nondegeneracy of the presymplectic form, hence defining a symplectic form. In this analogy, algebra objects correspond to subspaces and braided algebra objects to isotropic subspaces $I \subset V$. The $2$-category of braided module objects is analogous to the isotropic reduction $V \sslash I := I^\perp/I$, and an isotropic subspace is Lagrangian if this reduction is trivial.)

For general Lagrangian objects, condition~\ref{Lag2alg:connectivity} is too strong: There is a filtration of connectivity conditions that one can place on a braided monoidal object $A \in \cC$. For example, condition~\ref{Lag2alg:connectivity} is strictly stronger than the request that $u : I \to A$ be simple as a 1-morphism. (In Proposition~\ref{prop.undercategories}, this weaker condition corresponds to asking that $\cM$ be fusion but without the request that $\cB \to \cM$ be injective.) In Remark~\ref{rem:strongerversionLag}, we will further address such more general Lagrangian objects with simple but not necessarily fully faithful unit.
\end{remark}

\begin{prop} \label{prop.undercategories}
Let $\cB$ be a braided fusion $1$-category.
There is a bijective correspondence between isomorphism classes of rigid braided monoidal objects $A \in \cZ(\Sigma\cB)$ and braided multifusion categories $i: \cB \to \cM$ under $\cB$. Moreover:
\begin{enumerate} 
\item The functor $i: \cB \to\cM$ is fully faithful if and only if the unit $1$-morphism $u:I \to A$ is fully faithful in $\cZ(\Sigma\cB)$ (in which case $\cM$ is fusion).
\item The centralizer $\cZ_{2}(i : \cB \to \cM)$ is equivalent to $\Hom_{\cZ(\Sigma\cB)}(I,A)$, and the M\"uger centre of $\cM$ is the M\"uger centre of $A$ in the sense of Definition~\ref{def:internalmueger}.
\end{enumerate}
In particular, Lagrangian braided monoidal objects in $\BrModr{\cB}$ correspond to nondegenerate braided extensions $\cB \subseteq \cM$.
\end{prop}

\begin{proof} 
Using the equivalence $\cZ(\Sigma\cB) \cong \BrModr{\cB}$ from Theorem~\ref{thm.zmod} and unpacking definitions,
 (isomorphism classes of) braided monoidal objects $A$ correspond to finite semisimple braided monoidal categories $\cM$ equipped with braided monoidal functors $\cB \to \cM$ (up to isomorphism of braided monoidal categories under $\cB$), cf.~\cite[Table 1]{DN}. Specifically, the underlying category of $\cM$ is the underlying category of $A$ interpreted as a braided $\cB$-module, which is to say $\cM = \Hom_{\Modr{\cB}}(U(I), U(A))$, where $U: \BrModr{\cB} \to \Modr{\cB}$ forgets the braidings, and $I \in \BrModr{\cB}$ is $\cB$ as a braided $\cB$-module. This forgetful functor $U$ also has a universal description as the canonical map $\cZ(\Sigma\cB) \to \Sigma\cB$.
 
Using the braiding on $\cB$, any right $\cB$-module is naturally a $\cB$-bimodule; the corresponding functor $\Modr{\cB} \to \Bim{\cB}$ is monoidal. Since rigidity is preserved by monoidal functors, Lemma~\ref{lemma:rigidextensions} implies that if $A$ is rigid, then so is $\cM$, and hence $\cM$ is multifusion. But if $\cM$ is multifusion, then $A$ is rigid, because (just like in the proof of Lemma~\ref{lemma:rigidextensions}) the rigidity of $A$ unpacks to the adjunctibility of some $\cM$-bilinear map.

Fully faithfulness of the unit $u:I \to A$ is equivalent to $u:I \to A$ being the inclusion of a simple summand into $A$, and hence $i: \cB \to \cM $ being the inclusion of an indecomposable braided module category into the braided module category $\cM$, or equivalently, $i: \cB \to \cM$ being a fully faithful braided monoidal functor.

Lastly, note that for any braided module category $\cN$ with module braiding $\{\sigma_{n,b}:  n*b \to n*b\}_{ n \in \cN, b \in \cB}$, braided module functors $F: \cB \to \cN$ correspond to objects $n:= F(I_{\cB})$ in $\cN$ such that $\sigma_{n,b}$ is the identity for all $b\in \cB$. In particular, the braided monoidal $1$-category $\Hom_{\BrModr{\cB}}(I, A)$ of  $1$-morphism into our braided monoidal object $A$ is equivalent to $\cZ_{2}(\cB \subseteq \cM)$.
Unpacking Definition~\ref{def:internalmueger}, the M\"uger centre of $A$ can then be seen to be the further full subcategory $\cZ_{2}(\cM)$ of $\cZ_{2}( \cB \subseteq \cM)$. 
\end{proof}

\begin{example} 
\label{eg:canonicalLag}
The braided monoidal $2$-category $\cZ(\Sigma\cB)$ has a canonical Lagrangian braided monoidal object $\cL$ given by the nondegenerate extension $\cB \hookrightarrow \cZ(\cB)$. 
\end{example}

\begin{proof}[Proof of Theorem~\ref{thm.BCs}] 
Let $\cL$ denote the canonical Lagrangian object in $\cZ(\Sigma\cC)$ corresponding to the nondegenerate extension $\cC \hookrightarrow \cZ(\cC)$.
An equivalence $F : \cZ(\Sigma\cC) \cong \cZ(\Sigma\cB)$ supplies a Lagrangian object $F(\cL) \in \cZ(\Sigma\cB)$ and hence a nondegenerate extension $\cB \hookrightarrow \cM$. Moreover, $F$ induces an equivalence 
\[\cZ_{2}( \cB \subseteq \cM) = \Hom_{\cZ(\Sigma \cB)}(I, F(\cL)) \stackrel{F^{-1}}{ \cong} \Hom_{\cZ(\Sigma \cC)} (I, \cL) = \cZ_{2}( \cC \subseteq \cZ(\cC)) \cong \cC^{\rev}. \qedhere\]
\end{proof}

\begin{remark} \label{rem:strongerversionLag}
Following Remark~\ref{rem:Lagrangian}, we expect the following stronger version of Theorem~\ref{thm.BCs} to be true: For any nondegenerate braided fusion $2$-category $\mathcal{C}$, there is an equivalence between the $2$-groupoid of fusion $2$-categories $\cA$ equipped with an equivalence $\cZ(\cA) \cong \cC$ and the $2$-groupoid of ``general'' Lagrangian braided monoidal objects in $\cC$ (i.e.\ 
rigid braided monoidal objects $A$ fulfilling condition~\ref{Lag2alg:conditionnondeg} of Definition~\ref{defn.Lag2alg} but instead of condition~\ref{Lag2alg:connectivity} only requiring the unit $u:I \to A$ to be simple). The sub-groupoid of braided fusion 1-categories $\cB$ with equivalences $\cZ(\Sigma \cB) \cong \cC$ should be equivalent to the sub-groupoid of the ``strongly connected'' Lagrangian objects whose unit $u:I \to A$ is fully faithful. 

In particular, if $\cC$ is itself connected, i.e.\ of the form $\Sigma \cB$ for a braided fusion $1$-category $\cB$, then this stronger statement would more generally imply that there is an equivalence between the $2$-groupoid of braided fusion categories $\cM$ equipped with a (not necessarily fully faithful) braided monoidal functor $i: \cB \to \cM$ and the $2$-groupoid of fusion $2$-categories $\cA$ equipped with an equivalence $\cZ(\Sigma \cB) \cong \cZ(\cA)$. 
\end{remark}

\subsection{Half-braided algebras and bimodules}
\label{subsec.halfbraidedalgs}

Theorem~\ref{thm.zmod} provides a model for the 2-categorical Drinfel'd centre $\cZ(\cC)$ in the case when $\cC = \Sigma\cB$ is connected fusion. In this section we present an even more explicit model. Throughout this and the next section (and also \S~\ref{subsec:twistedtrace} and~\ref{subsec:etaL}), we use string diagrams for computations in braided monoidal $1$-categories, in which morphisms go up the page and the braiding $\beta_{x, y}:x \otimes y \to y \otimes x$ is drawn as: 
  \[\begin{tz}[std]
 \draw[braid, arrow data={0.8}{>}] (0,0) to [out=up, in=down]  (1,2);
 \draw[braid, arrow data={0.8}{>}] (1,0) to [out=up, in=down]  (0,2);
 \node[label, below] at (0,0) {$x$};
 \node[label, below] at (1,0) {$y$};
 \node[ below] at (0.5, -0.25) {$\beta_{x,y}$};
 \end{tz} \]
  
Recall that the Drinfel'd centre $\cZ(\Sigma\cB)$ consists of objects $A \in \Sigma\cB$ equipped with a half-braiding $\{h_X : X \boxtimes A \isom A \boxtimes X\}_{X \in \Sigma \cB}$, which is natural and monoidal in $X$. 
As explained in Example~\ref{eg.SigmaB}, objects of $\Sigma\cB$ are represented by separable algebras in $\cB$. Moreover, since $\Sigma\cB$ is connected, a natural transformation like $h_{-} : (-)\boxtimes A \isom A \boxtimes (-)$ is determined by its values on the unit object $I_{\Sigma\cB}$ and its endomorphisms, or in other words on the objects of~$\cB$. This motivates:

\begin{definition} \label{defn.HBA}
A \emph{half-braided algebra} $(A, \gamma)$ in a braided fusion $1$-category $\cB$ is a unital associative algebra object $A$ (depicted as a black line, with multiplication $A \otimes A \to A$ depicted as a black dot), whose underlying object is equipped with a half-braiding $\left(\gamma_{b}: b\otimes A \to A\otimes b\right)_{b\in \cB}$ (depicted as a black square) fulfilling:
\begin{equation}\label{eq:hb}
\begin{tz}[std]
\node (base) at (0,0.5){};
\draw[string, hb, on layer= back] (-0.5,0) to (2.3,2);
\braidmult{base}{1}{1}
\draw[string] (0,0) to (0,0.5); 
\draw[string] (0.5, 1.5) to (0.5, 2);
\draw[string]  (1,0) to (1,0.5) ;
\node[smalldot] at (0,0.35){};
\node[label, below] at (-0.5,0) {$b$};
\node[label, below] at (0., 0) {$A$};
\node[label, below] at (1,0) {$A$};
\end{tz}
=
\begin{tz}[std]
\node (base) at (0,0.5){};
\braidmult{base}{1}{1}
\draw[string, on layer=back] (0,0) to (0,0.5); 
\draw[string]  (1,0) to (1,0.5) ;
\draw[braid=back,hb, on layer=front] (-0.5,0) to (0.89,1) to (2.3,2) ;
\node[smalldot] at (0.9,1){};
\draw[string] (0.5, 1.5) to (0.5, 2);
\node[label, below] at (-0.5,0) {$b$};
\node[label, below] at (0., 0) {$A$};
\node[label, below] at (1,0) {$A$};
\end{tz}
=
\begin{tz}[std]
\node (base) at (1.25,0){};
\braidmult{base}{1}{1}
\draw[string] (1.75, 1) to (1.75,2); 
\draw[braid=back, hb] (-0.5,0) to (0.89,1) to node[smalldot, pos=0.61]{} (2.3,2) ;
\node[label, below] at (-0.5,0) {$b$};
\node[label, below] at (1.25, 0) {$A$};
\node[label, below] at (2.25,0) {$A$};
\end{tz}
\end{equation}
Here and in the rest of the paper, when working with string diagrams for natural transformations,
we will generally colour objects in which a morphism is natural in blue. This is merely a visual guidance and has no mathematical content.

A \emph{homomorphism of half-braided algebras} $(A, \gamma) \to (B, \zeta)$ is a unital algebra homomorphism $f:A\to B$ which intertwines the half-braiding: $\mu_b \cdot (\id_b \otimes f) = (f\otimes \id_b) \cdot \gamma_b$. We will need this notion only for isomorphisms.

We say that a half-braided algebra is \emph{separable} if the underlying algebra object in $\cB$ is separable. 
\end{definition}

Half-braided algebra can equivalently be expressed as algebra objects in $\cB$ equipped with a \emph{quantum moment map} in the sense of~\cite{MR3874702, Safronov}, see Remark~\ref{rem:quantummomentmap} for more details. 

\begin{example} The braiding of $\cB$ defines a half-braided algebra structure on any algebra object in the M\"uger centre $\cZ_{2}(\cB)$  of $\cB$. 

On the other hand, half-braided algebra structures on the trivial algebra $I$ correspond to monoidal natural automorphisms of $\cB$. 
\end{example}

In Proposition~\ref{prop:annular} and the surrounding discussion, we will show that (separable) half-braided algebras are nothing but (separable) algebra objects in the Drinfel'd center $\cZ(\cB)$ of the braided fusion $1$-category $\cB$, but equipped with an unusual monoidal structure. 

\begin{remark} \label{remark.pointedmoritaobject}
The algebra object $A \in \cB$ encodes both an object $[A] \in \Sigma \cB$ (given by the algebra as an object of the Morita category of algebras in $\cB$; recall from Example~\ref{eg.SigmaB} that $\Sigma\cB$ is the Morita category of separable algebra objects in $\cB$) and a $1$-morphism $_{A}A: I \to[]  [A]$ (the algebra as a left module for itself) with right adjoint $A_{A}: [A] \to I$. The algebra object $A$ in $\cB$ may be recovered as the composite $A_A \otimes_A {_A A} : I \to I$ with its induced ``pair of pants'' algebra structure.  A half-braiding~$h$ (in the sense of $\cZ(\Sigma \cB)$) on the object $[A]$ induces natural $2$-isomorphisms $h_b : b \otimes \id_{[A]} \to \id_{[A]} \otimes b$. In the graphical calculus of the monoidal $2$-category $\Sigma \cB$ (see~\cite[\S~2.1.2]{DR}), this can be depicted as follows:
   
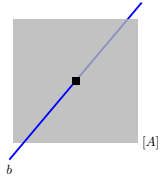
\[
     \begin{tz}[td]
  \begin{scope}[xzplane=1.]
 \path[surface] (0,0) rectangle (3,3);
%  \node[obj, above left] at (-0.1,-0.1){$B$};
 \end{scope}
 \draw[string,hb ] (0.5,0,0) to (1.5, 1, 1.5);
 \draw[string,hb,on layer =back] (1.5, 1, 1.5) to  (2.5,2, 3);
 \node[hb3d] at (1.5, 1, 1.5){};
 \node[label, right] at (3,1,0) {$[A]$};
 \node[label, below] at (0.5,0,0) {$b$};
 \end{tz}
 \]
The half-braiding $\gamma$ on $A$ from Definition~\ref{defn.HBA} arises as the composite
  
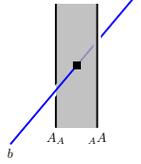
\[\begin{tz}[td]
  \begin{scope}[xzplane=1.]
 \path[surface] (0.25,0) rectangle (1.25,3);
 \end{scope}
  \draw[string, hb] (-0.25,0,0) to node[mask point,  pos=0.7](A){} (0.75, 1, 1.5);
\draw[1mor] (1.25, 1,0) to node[mask point, pos=0.7](B){}(1.25, 1,3);
  \cliparoundone{A}{\draw[1mor] (0.25, 1,0) to (0.25, 1,3);}
 \cliparoundone[back]{B}{\draw[string, hb,on layer =back] (0.75, 1, 1.5) to  (1.75,2, 3);}
 \node[hb3d] at (0.75, 1, 1.5){};
  \node[label, below] at (-0.25,0,0) {$b$};
  \node[label, below] at (0.25, 1,0) {$A_{A}$};
  \node[label, below] at (1.25,1,0) {$_{A}A$};
 \end{tz}
 \]
which makes equations~\eqref{eq:hb} evident. Even though we will not explicitly use this translation, for most computations and proofs in this section, this geometric intuition is very helpful. \end{remark}

\begin{lemma}\label{lem:inducedhb}
Let $m$ be a left, respectively right, 
 module of an algebra $A \in \cB$. Then, any half-braiding $\{\gamma_b : b\otimes A \to A\otimes b\}_{b \in \cB}$ on $A$ fulfilling~\eqref{eq:hb} equips the object $m \in \cB$ with a half-braiding $\{\gamma_{b, m}: b\otimes m \to m \otimes b\}_{b \in \cB}$ defined as follows, respectively:
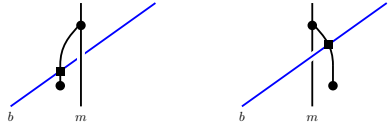
\[
\begin{tz}[std]
\node (base) at (0,0.5){};
\draw[braid, hb, on layer=back] (-1.2 , -0.5)to (-0.5,0) to (2.3,2); %1.4 per 1
\braidlact{base}{1}{1}
\draw[braid] (0.5,-0.5) to (0.5,1.5);
\draw[string] (0,0) to (0,0.5); 
\draw[string] (0.5, 1.5) to (0.5, 2);
\node[smalldot] at (0,0.35){};
\node[below, label] at (0.5,-0.5) {$m\vphantom{b}$};
\node[below, label] at (-1.2,-0.5) {$b$};
\node[dot] at (0,0){};
\end{tz}
\hspace{1cm}
\begin{tz}[std]
\node (base) at (0,0.5){};
\braidract{base}{1}{1}
\draw[string]  (1,0) to (1,0.5) ;
\draw[string, on layer=back] (0.5,-0.5) to (0.5, 1.5);
\draw[braid=back, hb, on layer=front] (-1.2, -0.5) to (-0.5,0) to  (2.3,2) ;
\node[smalldot] at (0.9,1){};
\draw[string] (0.5, 1.5) to (0.5, 2);
\node[below, label] at (0.5,-0.5) {$m\vphantom{b}$};
\node[below, label] at (-1.2,-0.5) {$b$};
\node[dot] at (1,0){};
\end{tz}
\]
For left modules (and analogously for right modules), 
this construction extends to a $\cB$-module functor $\Modl[\cB]{A} \to \cZ(\cB)$, where $\cB$ acts on the right on the Drinfel'd centre $\cZ(\cB)$ via the (braided) monoidal functor $\cB \to \cZ(\cB)$, $b \mapsto (b, \br_{-, b})$. 
\end{lemma}
\begin{proof} Immediate.
\end{proof}

\begin{definition} \label{def:hbbimod}A \emph{half-braided bimodule} $(A, \gamma) \to (B, \zeta)$ between half-braided algebras $(A, \gamma)$ and $(B, \zeta)$ is a unital $B$-$A$ bimodule $_{B}m_A$ in $\cB$ for which the action $B\otimes m \otimes A \to m$ is compatible with the half-braidings $\gamma$ and $\mu$  as follows:
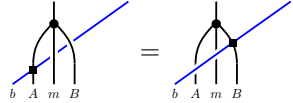
\[
\begin{tz}[std]
\node (base) at (0,0.5){};
\draw[braid, hb, on layer=back] (-0.5,0) to (2.3,2);
\braidmult{base}{1}{1}
\draw[braid] (0.5,0) to (0.5,1.5);
\draw[string] (0,0) to (0,0.5); 
\draw[string] (0.5, 1.5) to (0.5, 2);
\draw[string]  (1,0) to (1,0.5) ;
\node[smalldot] at (0,0.35){};
\node[below, label] at (-0.5,0) {$b$};
\node[below, label] at (0,0) {$A$};
\node[below, label] at (0.5,0) {$m\vphantom{A}$};
\node[below, label] at (1,0) {$B$};
\end{tz}
=
\begin{tz}[std]
\node (base) at (0,0.5){};
\braidmult{base}{1}{1}
\draw[string, on layer=back] (0,0) to (0,0.5); 
\draw[string]  (1,0) to (1,0.5) ;
\draw[string, on layer=back] (0.5,0) to (0.5, 1.5);
\draw[braid=back, hb,  on layer=front] (-0.5,0) to (0.89,1) to (2.3,2) ;
\node[smalldot] at (0.9,1){};
\draw[string] (0.5, 1.5) to (0.5, 2);
\node[below, label] at (-0.5,0) {$b$};
\node[below, label] at (0,0) {$A$};
\node[below, label] at (0.5,0) {$m\vphantom{A}$};
\node[below, label] at (1,0) {$B$};
\end{tz}
\]
Equivalently, a half-braided bimodule is one for which the two half-braidings induced by $\gamma$ and $\mu$ as in Lemma~\ref{lem:inducedhb} agree. 
\end{definition}
\begin{remark}
While half-braided algebras are algebras equipped with the additional \emph{structure} of a compatible half-braiding, it is a mere \emph{property} for a bimodule to be half-braided. In particular, morphisms of half-braided bimodules are just morphisms of bimodules without any additional compatibilities.  
In fact, on a half-braided bimodule $_{B}m_A$, the induced half-braiding $\gamma_{b, m_A} = \gamma_{b, {}_{B}m}: b\otimes m \to m\otimes b$ on $m$ from Lemma~\ref{lem:inducedhb} is the unique one fulfilling:
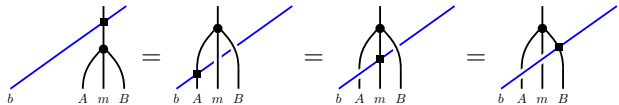
\[
\begin{tz}[std]
\node (base) at (1.25,0){};
\braidmult{base}{1}{1}
\draw[string] (1.75, 1) to (1.75,2); 
\draw[string] (1.75, 0) to (1.75, 1);
\draw[braid=back, hb ] (-0.5,0) to (0.89,1) to node[smalldot, pos=0.61]{} (2.3,2) ;
\node[below, label] at (-0.5,0) {$b$};
\node[below, label] at (1.25,0) {$A$};
\node[below, label] at (1.75,0) {$m\vphantom{A}$};
\node[below, label] at (2.25,0) {$B$};
\end{tz}
=
\begin{tz}[std]
\node (base) at (0,0.5){};
\draw[braid, hb, on layer=back] (-0.5,0) to (2.3,2);
\braidmult{base}{1}{1}
\draw[braid] (0.5,0) to (0.5,1.5);
\draw[string] (0,0) to (0,0.5); 
\draw[string] (0.5, 1.5) to (0.5, 2);
\draw[string]  (1,0) to (1,0.5) ;
\node[smalldot] at (0,0.35){};
\node[below, label] at (-0.5,0) {$b$};
\node[below, label] at (0,0) {$A$};
\node[below, label] at (0.5,0) {$m\vphantom{A}$};
\node[below, label] at (1,0) {$B$};
\end{tz}
=
\begin{tz}[std]
\node (base) at (0,0.5){};
\braidmult{base}{1}{1}
\draw[string, on layer=back] (0,0) to (0,0.5); 
\draw[string]  (1,0) to (1,0.5) ;
\draw[string, on layer=main] (0.5,0) to (0.5, 1.5);
\draw[braid=back, hb, on layer=back] (-0.5,0) to node[pos=0.355, smalldot]{} (2.3,2) ;
\draw[string] (0.5, 1.5) to (0.5, 2);
\node[below, label] at (-0.5,0) {$b$};
\node[below, label] at (0,0) {$A$};
\node[below, label] at (0.5,0) {$m\vphantom{A}$};
\node[below, label] at (1,0) {$B$};
\end{tz}
=
\begin{tz}[std]
\node (base) at (0,0.5){};
\braidmult{base}{1}{1}
\draw[string, on layer=back] (0,0) to (0,0.5); 
\draw[string]  (1,0) to (1,0.5) ;
\draw[string, on layer=back] (0.5,0) to (0.5, 1.5);
\draw[braid=back,hb,  on layer=front] (-0.5,0) to (0.89,1) to (2.3,2) ;
\node[smalldot] at (0.9,1){};
\draw[string] (0.5, 1.5) to (0.5, 2);
\node[below, label] at (-0.5,0) {$b$};
\node[below, label] at (0,0) {$A$};
\node[below, label] at (0.5,0) {$m\vphantom{A}$};
\node[below, label] at (1,0) {$B$};
\end{tz}
\]
\end{remark}
\begin{example} \label{eg:phistar}
Just as any (iso)morphism $\phi : A\to B$ of algebras induces a $B$-$A$ bimodule structure on $B$, so too
any (iso)morphism $ \phi: (A, \gamma) \to (B, \zeta)$ of half-braided algebras induces a half-braided bimodule, which we will denote $_{B}(\phi_{*})_{A}$. 

Given a right, respectively left, $B$-module $m_B$, respectively $_B m$, we will later use the notation $m_{\phi A}$, respectively $_{A\phi}m$ to denote its restriction along the (iso)morphism $\phi$; in this notation, ${_{B}(\phi_{*})_{A}} = {_B B_{\phi A}}$, but we will have times when this latter notation will be cumbersome.
\end{example}

Given half-braided bimodules $n:(A, \gamma) \to (B, \zeta)$ and $m:(B, \zeta) \to (C, \delta)$, their relative tensor product $_{C}m\otimes_B n_A$ is again a half-braided bimodule. Hence, (separable) half-braided algebras, half-braided bimodules and bimodule maps may be assembled into a $2$-category $\mathrm{(s)HBA}(\cB)$.

Recall from Lemma~\ref{lem:inducedhb} that a half-braiding $\{\gamma: b\otimes A \to A\otimes b\}_{b \in \cB}$ on an algebra $A$ fulfilling~\eqref{eq:hb} induces a half-braiding $\gamma_{b, _{A}m}: b\otimes m\to m \otimes b$ on any left $A$-module $_{A}m$. 
The following Lemma is a half-braided algebra version of~\cite[Corollary 4.11]{MR3874702}. 
\begin{lemma}\label{lem:hbvsBmodbraiding}
Let $A$ be an algebra in $\cB$ with associated right $\cB$-module category $\Modl[\cB]{A}$. 
For any half-braiding $\gamma$ on $A$ fulfilling~\eqref{eq:hb}, the natural isomorphism
\[\{\sigma_{_{A}m, b}:=  \gamma_{b, _{A}m} \cdot \beta_{m, b}: m\otimes b \to m \otimes b\}_{b \in \cB,\,{_{A}m} \in \Modl[\cB]{A}}\]
defines a $\cB$-module braiding on $\Modl[\cB]{A}$. This construction defines a bijection between the set of half-braidings on $A$ fulfilling~\eqref{eq:hb} and the set of $\cB$-module braidings on $\Modl[\cB]{A}$. 

 Moreover, if $_{B}m_A$ is a bimodule between half-braided algebras $(A, \gamma)$ and $(B, \zeta)$, then the $\cB$-module functor $_{B}m \otimes_A - : \Modl[\cB]{A} \to \Modl[\cB]{B}$ is a braided module functor if and only if the bimodule $_{B}m_A$ is half-braided in the sense of Definition~\ref{def:hbbimod}. 
\end{lemma}
\begin{proof} Verifying that $\sigma$ defines a $\cB$-module braiding is a direct consequence of the axioms. To show that this assignment is a bijection, we construct an explicit inverse: Given a $\cB$-module braiding $\{\sigma_{_{A}m, b}\}_{b \in \cB, _{A}m \in \Modl[\cB]{A}}$ we set $\gamma_b := \sigma_{_{A}A, b}\cdot \beta^{-1}_{A, b}: b \otimes A \to A \otimes b$ and note that $\gamma_b$ is a half-braiding on $A$ fulfilling~\eqref{eq:hb}. A direct computation shows that these two functions are inverse to one another. 
\end{proof}

\begin{remark}The $\cB$-module category $\Modl[\cB]{A}$ is the category $\Hom_{\Sigma \cB}(I, [A])$ 
with right $\cB$-module structure induced by the identification $\Hom_{\Sigma \cB}(I, I) = \Omega \Sigma \cB \cong \cB$. Following the diagrammatic notation from Remark~\ref{remark.pointedmoritaobject} for the (2-categorical) half-braiding on the object $[A]$ in $\Sigma \cB$, the induced $\cB$-module braiding $\sigma_{_{A}m, b}: m*b \to m*b$ on $\Modl[\cB]{A}$ arises as follows: 
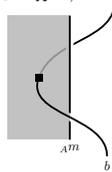
\[
   \begin{tz}[td]
  \begin{scope}[xzplane=1.]
 \path[surface] (0,0) rectangle (1.5,3);
 \end{scope}
 \draw[string] (3,0,0) to [out=up, in=-110]node[mask point, pos=0.5](A){} (0.75, 1, 1.5);
  \cliparoundone{A}{\draw[1mor] (1.5, 1,0) to node[mask point=back, pos=0.75](B){} (1.5, 1,3);}
\cliparoundone[back]{B}{ \draw[string,on layer =back] (0.75, 1, 1.5) to [out=70, in=down]  (2,2, 3);}
 \node[smalldot] at (0.75, 1, 1.5){};
 \node[label, below] at( 1.5, 1, 0) {$_{A}m$};
 \node[label, below] at (3,0,0){$b$};
 \end{tz}
\]
\end{remark}

\begin{theorem} \label{thm:HBADrinfeld}The $2$-category $\sHBA(\cB)$ of separable half-braided algebras is equivalent to $\BrModr\cB$  and hence to the Drinfel'd centre $\cZ(\Sigma \cB)$. 
\end{theorem}
\begin{proof}The $2$-functor $\sHBA(\cB) \to \BrModr\cB$ is defined as follows:
To a half-braided algebra $(A, \gamma)$ in $\cB$, we assign the the right $\cB$-module category $\Modl[\cB]{A}$ of left $A$ modules with $\cB$-module braiding induced from the half-braiding $\gamma$ as in Lemma~\ref{lem:hbvsBmodbraiding}. To a half-braided bimodule $_{B}m_A$ we assign the induced $\cB$-module functor $_{B}m \otimes_A- : \Modl[\cB]{A} \to \Modl[\cB]{B}$ (see Lemma~\ref{lem:hbvsBmodbraiding}), and to a bimodule morphism $_{B}m_A \to {_{B}m'_A}$ we assign the corresponding $\cB$-module natural transformation $f \otimes_A -:  {}_{B}m \otimes_A - \To {}_{B}m' \otimes_A -$. 
Since $\Modl[\cB]{(-)}: \Sigma \cB \to \Modr \cB$ is an equivalence (see Example~\ref{eg.SigmaB}), it follows from Lemma~\ref{lem:hbvsBmodbraiding} that this $2$-functor $ \sHBA(\cB) \to \BrModr\cB$ is an equivalence.  \end{proof}

In particular, the braided monoidal structure on $\cZ(\Sigma \cB)$ induces a braided monoidal structure on $\sHBA(\cB)$. 
For the reader's convenience, we explicitly unpack some pieces of this induced braided monoidal structure. 
The tensor unit of $\sHBA(\cB)$ is the trivial algebra $I$ with half braiding $\gamma_b = \id_b: b\to b$ (omitting unitors from the notation).
The tensor product $(A, \gamma) \boxtimes (B, \zeta)$ of two half-braided algebras $(A,\gamma)$ and $(B, \zeta)$ is the half-braided algebra with underlying object $A\otimes B$ and multiplication and half-braiding $\gamma \boxtimes \mu : b\otimes (A\otimes B) \to (A\otimes B) \otimes b$ given as follows: 
\begin{equation}\label{eq:tensor}
\begin{tz}[std]
\node (base) at (0,0){};
\node (2nd) at (0.5,0.){};
\draw[string](0.5, 1) to (0.5, 1.5);
\draw[string] (1, 1) to (1, 1.5);
\mult{2nd}{1}{1}
\braidmult{base}{1}{1}
\node[below, label] at (0.,0) {$A$};
\node[below, label] at (0.5,0) {$B$};
\node[below, label] at (1,0) {$A$};
\node[below, label] at (1.5,0) {$B$};
\end{tz}
\hspace{0.3\linewidth}
 \begin{tz}[std]
 \draw[string, hb] (0,0) to [out=up, in=down]node[smalldot, pos=0.385]{} node[smalldot,pos=0.615]{}(1.5,1.5);
 \draw[string] (0.5,0) to (0.5, 1.5);
 \draw[string] (1,0) to (1,1.5);
\node[below, label] at (0,0) {$b\vphantom{A}$};
\node[below, label] at (0.5,0) {$A$};
\node[below, label] at (1,0) {$B$};
\end{tz}
\end{equation}
Similarly, the tensor product $(_{A}m_B)\boxtimes (_{C}n_D)$ of two half-braided bimodules $_{A}m_B$ and $_{C}n_D$ is the $(A\otimes C)$--$(B\otimes D)$ bimodule $m\otimes n$ with action given as follows:
\[\begin{tz}[std]
\node (base1) at (1.25,0){};
\node (base2) at (1.75,0){};
\braidmult{base1}{2}{1}
\draw[string] (2.25, 0) to node[mask point,pos=0.6] (A){} (2.25, 1);
\cliparoundone[back]{A}{\braidmult{base2}{2}{1}}
\draw[string] (2.25, 1) to (2.25,2); 
\draw[string] (2.75, 1) to (2.75,2); 
\draw[string] (2.75, 0) to (2.75, 1);
\node[below, label] at (1.25,0) {$A$};
\node[below, label] at (1.75,0) {$B$};
\node[below, label] at (2.25,0) {$m\vphantom{A}$};
\node[below, label] at (2.75, 0) {$n\vphantom{A}$};
\node[below, label] at (3.25, 0) {$C$};
\node[below, label] at (3.75, 0) {$D$};
\end{tz}
\]
\begin{lemma}\label{lem:algebraiso}Let A and B be algebra objects in $\cB$ and suppose $\gamma$ is a half-braiding on~$A$  fulfilling~\eqref{eq:hb}. Then, $\gamma_B: B\otimes A \to A \otimes B$ is an algebra isomorphism. Moreover, if $B$ admits a half-braiding $\mu$, then $\gamma_B$ is an isomorphism of half-braided algebras $(B, \zeta) \boxtimes (A, \gamma) \to (A, \gamma) \boxtimes (B, \zeta) $.
\end{lemma}
\begin{proof} This follows directly from naturality of $\gamma_B$ in $B$ and its compatibility with the algebra structure on $A$.
\end{proof}
In the braided monoidal $2$-category $\sHBA(\cB)$, the braiding of two half-braided algebras $(A, \gamma) \boxtimes (B, \zeta) \to (B, \zeta) \boxtimes (A, \gamma)$ is the half-braided bimodule $\left(\mu_A\right)_*$ represented by the half-braided algebra isomorphism $\mu_A: A\otimes B \to B \otimes A$ as in Lemma~\ref{lem:algebraiso}.

\begin{remark}
It is often convenient to present braided monoidal $2$-categories in terms of braided monoidal double categories (see~\cite{Shulman2010,1910.09240} on notations and details of this presentation). 

In our case, the full braided monoidal structure on the $2$-category $\HBA(\cB)$ (and analogously on the subcategory $\sHBA(\cB)$) can be efficiently described in terms of the braided monoidal double category $\mathbb{H}\mathrm{BA}(\cB)$ whose groupoid of objects $\mathbb{H}\mathrm{BA}(\cB)_0$ is the braided monoidal $1$-groupoid of half-braided algebras and half-braided algebra isomorphisms with monoidal structure and braiding given as in~\eqref{eq:tensor} and Lemma~\ref{lem:algebraiso}. Its  braided monoidal $1$-category $\mathbb{H}\mathrm{BA}(\cB)_1$ of $1$-morphisms is the braided monoidal category whose objects are triples $(A,m,B)$ of half-braided algebras $A $ and $B$ and a half-braided bimodule $_{A}m_B$ and whose morphisms $(A, m, B) \to (C, n, D)$ are triples $(\phi, f, \psi)$ of half-braided algebra isomorphisms $\phi:A \to C$ and $\psi:B \to D$ and an $A$-$B$ bimodule map $f: (\phi_*) \otimes_A m \to  n \otimes_{D} ( \psi _*)$. 

The monoidal structure on $\mathbb{H}\mathrm{BA}(\cB)_1$ has unit $(I, I,I)$ and tensor product $(A, m, B) \boxtimes (C, n,D) = (A\otimes C, m \otimes n, B \otimes D)$ where we use the tensor product of half-braided algebras from above and give $m \otimes n$ the obvious bimodule structure. The braiding isomorphism $ (A,  m,B) \otimes (C,n, D) \to (C,n, D) \otimes (A, m,B)$ in  $\mathbb{H}\mathrm{BA}(\cB)_1$ is given by the triple $(\gamma_A, \zeta_{m, n_D}, \zeta_B)$, where $\gamma_A: A\otimes C \to C \otimes A$ is the half-braided algebra isomorphism given by the half-braiding $\gamma$ of $C$, $\zeta_B: B\otimes D \to D \otimes B$ the one induced by the half-braiding $\zeta$ of $D$ and where $\zeta_{m,n_D}:m \otimes n \to n \otimes m $ is the half-braiding on $n$ induced by the half-braiding $\zeta$ of $D$  as in Lemma~\ref{lem:inducedhb} (as $n$ is a half-braided bimodule, $\xi_{m, n_{D}}$ equals the half-braiding $\gamma_{m, _{C}n}$ induced by $\gamma$ of $C$). This indeed defines a $(C \otimes A)$--$(B \otimes D)$ bimodule isomorphism $ (\gamma_A )_* \otimes_{A \otimes C} (m\otimes n) \to  (n\otimes m) \otimes_{D \otimes B} (\zeta_B)_*$ as required. \end{remark}

\begin{example} As discussed in Lemma~\ref{lemma.centres}, $\Omega \cZ(\Sigma \cB) \cong \cZ_{2}(\cB)$. This is easily see when describing $\cZ(\Sigma \cB)$ in terms of half-braided algebras: the tensor unit of $\sHBA(\cB)$ is the trivial algebra in $\cB$; any object $b$ of $\cB$ admits a unique $I$--$I$ bimodule structure; this bimodule structure is half-braided in the sense of Definition~\ref{def:hbbimod} if and only if $b$ is transparent in $\cB$. 
\end{example}

Combining Theorem~\ref{thm:HBADrinfeld} and Lemma~\ref{lem:inducedhb}, it follows that the (underlying object of the) canonical Lagrangian braided monoidal object $\cB \to \cZ(\cB)$ (see Example~\ref{eg:canonicalLag}) is a generating object of $\BrModr\cB \cong \cZ(\Sigma \cB)$. 

\begin{corollary} \label{cor.Lgenerates}
The braided $\cB$-module category $\cZ(\cB)$ with braided $\cB$-module structure induced by the braided monoidal functor $\cB \to \cZ(\cB)$ is a generating object (in the sense of Definition~\ref{defn:generator}) of $\BrModr\cB$. In particular, the $2$-functor $\Sigma\Omega_{\cZ(\cB)} \BrModr\cB \to\BrModr\cB$ is an equivalence of $2$-categories. 
\end{corollary}
\begin{proof} By Theorem~\ref{thm:HBADrinfeld}, any braided right $\cB$-module category $\cM$ is equivalent to the category  $\Modl[\cB]{A}$ of left modules for a half-braided algebra $(A, \gamma)$, with $\cB$-module braiding constructed as in Lemma~\ref{lem:hbvsBmodbraiding}. With this $\cB$-module braiding, the $\cB$-module functor $\Modl[\cB]{A}\to \cZ(\cB)$ from Lemma~\ref{lem:inducedhb} is in fact a braided $\cB$-module functor, and is evidently non-zero (if $A$ is non-zero). \end{proof}

\subsection{The canonical Lagrangian object in \texorpdfstring{$\cZ(\Sigma \cB)$}{Z(Sigma B)} as a half-braided algebra}\label{subsec.canonicalLagHBA}

In Example~\ref{eg:canonicalLag} we observed that, for a braided fusion 1-category $\cB$, the braided fusion 2-category $\cZ(\Sigma\cB) \cong \BrModr\cB$ has a canonical Lagrangian braided monoidal object $\cL$, which comes from the nondegenerate extension $\cB \hookrightarrow \cZ(\cB)$. We now present an explicit half-braided algebra representative $(L, \lambda)$ of the underlying object of this canonical Lagrangian braided monoidal object. We will not give a half-braided algebra description of the braided monoidal structure of $\cL$, as we will not require it for our proof.

To give the definition, recall that the \emph{coend} $\int^{\cB}F = \int^{b\in \cB} F(b,b)$ of a functor $F:\cB  \times \cB^{\mathrm{op}} \to \cB$ is the universal object in $\cB$ equipped with a dinatural \cite{MR0274550} transformation
 $\iota: F(-,-) \to \int^{\cB} F$. In other words, for each $x \in \cB$ there is a map $\iota_x : F(x,x) \to \int^\cB F$, and for each morphism $\alpha : x \to y$, 
 $$ \iota_x \cdot F(\alpha, \id_y) = \iota_y \cdot F(\id_x, \alpha) : F(x,y) \to \int^\cB F,$$
 and $\int^\cB F$ is the universal object in $\cB$ receiving maps with these compatibilities.
Of course, as $\cB$ is semisimple, the map $\bigoplus_{b_i \in \pi_0 \cB} \iota_{b_i}: \bigoplus_{b_i \in \pi_0 \cB} F(b_i, b_i) \to \int^{b \in \cB} F(b, b)$ is an isomorphism, but we prefer the coend formulation with its defining universal property, as this will make various explicit computations cleaner. For example, for objects $x,z \in \cB$, to define a map $x \otimes \int^{b \in \cB} F(b, b) \to z$, it suffices to write down a map $x \otimes F(b,b) \to z$ which is dinatural in $b$.

\begin{definition} \label{defn.lagHBA}
The \emph{canonical Lagrangian half-braided algebra} $(L, \lambda)$ in $\sHBA(\cB)$ is the half-braided algebra $L:=\int^{b \in \cB} b \otimes b^*$ with unit $\iota_{I}: I \cong I \otimes I^* \to L$ and multiplication $L\otimes L \to L$ defined in terms of the following dinatural (in $b$ and $c$) transformation $b\otimes b^* \otimes (c\otimes c^*) \to L$:
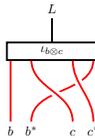
\[\begin{tz}[std]
\clip (-0.5, -0.5) rectangle (2.5, 3);
\draw[dinat] (0,0) to (0,1.5);
\draw[dinat] (0.5, 0) to [out=up, in=down] (2,1.5);
\draw[braid=main, dinat] (1.5,0) to [out=up, in=down] (0.5,1.5);
\draw[braid=main, dinat] (2,0) to [out=up, in=down] (1.5,1.5);
\draw[string] (1, 1.7) to (1,2.5);
\node[box=2.4cm] at (1, 1.7) {$\iota_{b\otimes c}$};
\node[label, below] at (0,0) {$b$};
\node[label, below] at (0.5,0) {$b^*$};
\node[label, below] at (1.5, 0) {$c\vphantom{c^*}$};
\node[label, below] at (2,0) {$c^*$};
\node[label, above] at (1,2.5) {$L$};
\end{tz}
\]
The half-braiding $\{\lambda_x:x \otimes L \to L \otimes x\}_{x\in \cB}$ is given by the following dinatural (in $b$) transformation $x\otimes (b \otimes b^*) \to L \otimes x$:
\begin{equation}\label{eq:Laghb}
\begin{tz}[std]
\draw[string,hb, arrow data={0.5}{>}] (0,0) to [out=up, in=down] (0.5, 1);
\draw[string, dinat, arrow data={0.5}{>}] (1,0) to (1,1);
\draw[string, dinat, arrow data={0.5}{<}] (1.5, 0) to +(0,1);
\draw[string, hb, arrow data={0.15}{>}] (2, 1) to [out=down, in=down, looseness=2] (3, 2);
\draw[string] (1.25, 1.5) to (1.25, 2);
\node[box=2.4cm] at (1.25, 1.25) {$\iota_{x\otimes b}$};
\node[label, above] at (1.25, 2) {$L$};
\node[label, below] at (1, 0) {$b$};
\node[label, below] at (1.5, 0) {$b^*$};
\node[label, below] at (0, 0) {$x\vphantom{b^*}$};
\node[label, above] at (3, 2) {$x$};
\end{tz}
\end{equation}
\end{definition}
In this Definition~\ref{defn.lagHBA} and in the rest of the paper, whenever we work with a string diagram expression for a dinatural transformation, we will colour the objects in which the transformation is dinatural in {red}. This is merely a visual guidance and has no mathematical content. 

\begin{lemma} \label{lem:lagworks}
The functor $\Modl[\cB]{L} \to \cZ(\cB)$ induced from the half-braiding on $L$ as in Lemma~\ref{lem:inducedhb} is an equivalence of braided module categories. 
\end{lemma}
The Lemma confirms that $(L, \lambda)$ as described in Definition~\ref{defn.lagHBA} is indeed a half-braided algebra representative of the canonical Lagrangian object $\cL$ from Example~\ref{eg:canonicalLag}. In particular, it follows that the algebra $L$ is separable and is therefore indeed an object of $\sHBA(\cB) \cong \cZ(\Sigma \cB)$ and not merely of $\HBA(\cB)$.

\begin{proof}An explicit inverse $\cZ(\cB) \to \Modl[\cB]{L} $ maps an object $(x, \gamma)$ to the object $b$ with left $L$ action defined by the following dinatural transformation $b \otimes b^* \otimes x \to x$:
\[
\begin{tz}[std]
\clip (-1.2, -0.4) rectangle (0.75, 2);
\draw[string, dinat] (-1,0) to [out=up, in=-135] (0, 1.5) to  +(45:0.2cm) coordinate(x); 
\draw[string,dinat, arrow data={0.9}{>}] (x) arc (135:-45:0.3cm) to [out=-135, in=up] (-0.5,0); 
\draw[braid=main] (0,0) to (0,1);
\draw[string](0,1) to (0,2);
\node[smalldot] at (0,1.5) {};
\node[label, below] at (0,0) {$x\vphantom{b^*}$};
\node[label, below] at (-0.5,0) {$b^*$};
\node[label, below] at (-1,0) {$b$};
\end{tz} \qedhere
\]
\end{proof}

In Corollary~\ref{cor.Lgenerates}, we showed that the Lagrangian object $\cL$, represented by the half-braided algebra $(L, \lambda)$ or the braided module category $\cZ(\cB)$, respectively, generates the $2$-category $\cZ(\Sigma \cB) \cong \sHBA(\cB) \cong \BrModr{\cB}$. In particular, the $2$-functor $\Sigma \Omega_{\cL} \cZ(\Sigma \cB) \to \cZ(\Sigma \cB)$ is an equivalence. In other words, ignoring the braided monoidal structure, the semisimple $2$-category $\cZ(\Sigma \cB)$ is the category of finite semisimple module categories of the multifusion $1$-category $\End_{\cZ(\Sigma \cB)}(\cL)$. Although not strictly necessary for the proof of our \MainTheorem, we now show that this multifusion category is equivalent to the Drinfel'd centre $\cZ(\cB)$ of $\cB$ equipped with an unusual monoidal structure.

For every object $(x, \gamma)$ in the Drinfel'd center $\cZ(\cB)$ we define the morphisms $\mathrm{l}_{x, \gamma}:L \otimes x \to x$ and  $\mathrm{r}_{x, \gamma}: x\otimes L \to x$  in $\cB$ in terms of the following dinatural transformations, respectively (the square denotes the half-braiding $\gamma_b: b\otimes x \to x \otimes b$ of $x$):
\begin{equation}\label{eq:LLbimod}
\begin{tz}[std]
\clip (-1.2, -0.4) rectangle (0.75, 2);
\draw[string, dinat] (-1,0) to [out=up, in=-135] (0, 1.5) to  +(45:0.2cm) coordinate(x); 
\draw[string,dinat, arrow data={0.9}{>}] (x) arc (135:-45:0.3cm) to [out=-135, in=up] (-0.5,0); 
\draw[braid=main] (0,0) to (0,1);
\draw[string](0,1) to (0,2);
\node[smalldot] at (0,1.5) {};
\node[label, below] at (0,0) {$x\vphantom{b^*}$};
\node[label, below] at (-0.5,0) {$b^*$};
\node[label, below] at (-1,0) {$b$};
\end{tz}
\hspace{4cm}
\begin{tz}[std]
\clip (-0.5, -0.4) rectangle (1.2, 2);
\draw[string] (0,0) to (0,2);
\draw[braid, dinat] (0.5,0) to [out=up, in=down] (-0.4, 1);
\draw[string, dinat, arrow data={0.7}{>}] (-0.4,1)  to [out=up, in =-135] (0, 1.5) to [out=45,  in=up, out looseness=2](1,0);
\node[smalldot] at (0,1.5) {};
\node[label, below] at (0,0) {$x\vphantom{b^*}$};
\node[label, below] at (0.5,0) {$b$};
\node[label, below] at (1,0) {$b^*$};
\end{tz}
\end{equation}
Observe that $\mathrm{l}_{x, \gamma}$ and $\mathrm{r}_{x, \gamma}$ define commuting left and right actions of the algebra object $L$, and hence an $L$--$L$ bimodule structure on $x$. This construction defines a fully faithful functor $\cZ(\cB) \hookrightarrow \mathrm{Bim}(L)$. 

\begin{definition}\label{def:annular}
The \define{annular tensor product} $(x, \gamma) \otimes_{\mathrm{ann}} (y, \zeta)$ of two objects in $\cZ(\cB)$ is the coequalizer of $\id_x \otimes \mathrm{l}_{y, \zeta}$ and $\mathrm{r}_{x,\gamma} \otimes \id_y: x \otimes L \otimes y \to x \otimes y$ in $\cB$ equipped the half-braiding 
\begin{equation}
\label{eq:annularhb}
\begin{tz}[std]
\draw[string, on layer=main] (0.3333,0) to (0.33333, 2);
\draw[braid=back, hb, on layer=back] (-0.5,0) to node[pos=0.33, smalldot](A){} (2.,2) ;
\draw[braid]  (1.16666,0) to (1.16666,2) ;
\node[below, label] at (-0.5,0) {$b\vphantom{y}$};
\node[below, label] at (0.333333,0) {$\vphantom{yb}x$};
\node[below, label] at (1.1666,0) {$y\vphantom{b}$};
\node[above left, label] at (A){$\gamma_b$};
\end{tz}
\hspace{0.75cm} \text{or equivalently} \hspace{0.75cm} \begin{tz}[std]
\draw[string, on layer=main] (0.3333,0) to (0.33333, 2);
\draw[string]  (1.16666,0) to (1.16666,2) ;
\draw[braid=main, hb] (-0.5,0) to node[pos=0.66666, smalldot](A){} (2.,2) ;
\node[below, label] at (-0.5,0) {$b\vphantom{y}$};
\node[below, label] at (0.333333,0) {$\vphantom{yb}x$};
\node[below, label] at (1.1666,0) {$y\vphantom{b}$};
\node[below right, label] at (A){$\zeta_b$};
\end{tz}
\end{equation}
descended to the coequalizer. In other words, the annular tensor product $(x, \gamma) \otimes_{\mathrm{ann}} (y, \zeta)$ may be understood as coequalizing the two half-braidings~\eqref{eq:annularhb} on $x \otimes y $.
\end{definition}

\begin{prop} \label{prop:annular}
The subcategory $\cZ(\cB) \hookrightarrow \mathrm{Bim}(L)$ is precisely the subcategory of those bimodules which are half-braided in the sense of Definition~\ref{def:hbbimod}, and hence is equivalent to $\End_{\sHBA(\cB)}(L,\lambda)$. The induced tensor product on $\cZ(\cB)$ is the annular tensor product. 
\end{prop}
In particular, a (separable) half-braided algebra in the sense of Definition~\ref{defn.HBA} is precisely a separable algebra object in the monoidal category $(\cZ(\cB), \otimes_{\mathrm{ann}})$, yielding another proof of Corollary~\ref{cor.Lgenerates} that $\Sigma (\cZ(\cB), \otimes_{\mathrm{ann}}) \cong \sHBA(\cB)$.

\begin{proof}Analogous to Lemma~\ref{lem:lagworks}, the data of an $L$--$L$ bimodule structure on an object $x$ of $\cB$ is equivalent to the data of two commuting
 half-braidings $\gamma, \zeta$. The bimodule induced by such a $(x, \gamma, \zeta)$ is half-braided in the sense of Definition~\ref{def:hbbimod} precisely when the two half-braidings $\gamma$ and $\zeta$ agree. The induced tensor product is the relative tensor product over $L$ which agrees with the annular tensor product of Definition~\ref{def:annular}. 
\end{proof}

\begin{remark} \label{rem:quantummomentmap} Half-braided algebras $(A, \gamma)$ are algebra objects in $(\cZ(\cB), \otimes_{\mathrm{ann}})$. 
As in the proof of Proposition~\ref{prop:annular}, identify $\cZ(\cB)\cong \Modl[\cB]{L}$ and embed $\Modl[\cB]{L} \hookrightarrow \operatorname{Bim}(L)$ as those $L$-bimodules for which (the half braidings that encode) the left and right $L$-actions agree.
It follows that algebra objects in $(\cZ(\cB), \otimes_{\mathrm{ann}})$ can be identified with algebra objects in $\operatorname{Bim}(L)$, or equivalently, algebra objects $A$ in $\cB$ equipped with an algebra homomorphism $L \to A$, which fulfill the additional compatibility condition from~\cite[Definition 3.1]{Safronov} (which ensures that the induced $L$--$L$ bimodule arises from an $L$-module). 
Hence, the data of a compatible half-braiding $\gamma$ on an algebra $A$ in $\cB$ is equivalent to the data of a \emph{quantum moment map} on $A$ in the sense of~\cite[Definition 3.1]{Safronov}. (The original definition of quantum moment map from~\cite{MR3874702} is missing the compatibility condition, cf.~\cite[Remark~3.5]{Safronov}.) \end{remark}

\begin{remark} As $L$ is a separable algebra, the annular tensor product on $\cZ(\cB)$ may also be described by splitting certain idempotents. 
Explicitly, given objects $(x,\gamma), (y,\zeta) \in \cZ(\cB)$, their annular tensor product $(x, \gamma) \otimes_{\mathrm{ann}} (y, \zeta)$ may be obtained by splitting the following idempotent on $x\otimes y$: 
\begin{equation}\label{eq:idempotent}
\sum_{b \in \pi_0 \cB} ~ \frac{d_-\!(b)}{\mathcal{D}}
\begin{tz}[std, scale=1.5]
\draw[string] (0,0) to (0,0.5);
\draw[hb, arrow data={0.2}{>}] (0,1) to [out=45, in=45, looseness=3.5] (1,1);
\draw[braid, hb] (1,1) to [out=-135, in=-135, looseness=3.5] node [pos=0.15] (A){}(0,1);
\draw[string] (1,0) to (1, 1.5);
\draw[braid] (1,1.5) to (1,2);
\draw[string] (0,0.5) to (0,2);
\node[smalldot] at (0,1){};
\node[smalldot] at (1,1){};
\node[dot] at (A){};
\node[label, below right] at (A) {$\alpha_b$};
\node[label, above left] at (0,1) {$\gamma_b$};
\node[label,below right] at (1,1) {$\zeta_{b^*}$};
\node[label,below] at (0,0) {$x\vphantom{yf}$};
\node[label,below] at (1,0) {$y\vphantom{f}$};
\end{tz}\end{equation}
Here, the sum is over a set of representatives of the isomorphism classes of simple objects of $\cB$, and we have chosen right and left duals and an isomorphism $\alpha_b: {}^*b \to b^*$ for every one of them. The scalar factor $d_-\!(b) = \ev_{{}^*b} \cdot (\alpha_b^{-1} \otimes \id_b) \cdot \coev_b$ is as defined in equation~\eqref{eqn:dimensions} in \S\ref{sec:Smatrix} below, and $\cD:= \sum_{b \in \pi_0 \cB} d_{-}(b) d_+(b)$ is the \emph{global dimension} of $\cB$ (which is itself independent of these choices). The idempotent~\eqref{eq:idempotent} is independent of all these choices.

Versions of the annular tensor product for ribbon braided fusion categories have appeared in~\cite{MR4074586} and~\cite{2004.09611} under the name ``reduced tensor product.''
\end{remark}

\begin{remark}  \label{rem.ann}
The multifusion $1$-category $(\cZ(\cB), \otimes_{\mathrm{ann}})$ is itself the (1-categorical) idempotent completion of the \emph{annular category} $\mathrm{Ann}(\cB)$ whose object are the objects of $\cB$, with $\Hom_{\mathrm{Ann}(\cB)}(x, y) = \int^{b \in \cB} \Hom_{\cB}(b \otimes x,x \otimes b^{**})$, and with monoidal structure directly induced by $\cB$ (defining the tensor product of morphisms in $\mathrm{Ann}(\cB)$ requires the braiding of $\cB$). The inclusion $\mathrm{Ann}(\cB) \to \cZ(\cB)$ maps an object $x$ to the object $\int^b b\otimes x \otimes b^*$ with half-braiding defined analogously to~\eqref{eq:Laghb}, and conversely, any object $(y, \gamma) \in \cZ(\cB)$ gives rise to an idempotent in $\mathrm{Ann}(\cB)$ on $y$ seen as an object of $\mathrm{Ann}(\cB)$. 

In other words, the sequence of $2$-functors $\rB\mathrm{Ann}(\cB) \to \rB\cZ(\cB) \to \cZ(\Sigma \cB)$ may be understood as successively completing with respect to ordinary idempotents and then $2$-idempotents, and the whole underlying $2$-category of $\cZ(\Sigma \cB)$ can be recovered from $\mathrm{Ann}(\cB)$. 
See~\cite{2004.09611} for a detailed construction (in the ribbon setting) of the equivalence between $(\cZ(\cB),\otimes_{\mathrm{ann}})$ and the Karoubi completion of $\mathrm{Ann}(\cB)$.

More geometrically, the annular category $\mathrm{Ann}(\cB)$ is equivalent to the \emph{skein category} $\int_{S^1_b} \cB$ of $\cB$, whose objects are appropriately framed $\cB$-labelled points on a framed annulus $S^1 \times I$  and whose morphisms are $\cB$-string diagrams in $S^1 \times I \times I$. From this perspective, our monoidal structure $\otimes_{\mathrm{ann}}$ comes from ``stacking'' annuli, hence the name. In the oriented/ribbon case, these various perspectives on $(\cZ(\cB), \otimes_{\mathrm{ann}})$ are discussed in detail in~\cite{2004.09611}. (More formally, $\mathrm{Ann}(\cB)$ is the factorization homology of the $E_2$-category $\cB$ on $S_b^1$ with its bounding $2$-framing, and the monoidal structure $\otimes_{\mathrm{ann}}$ is the remaining $E_1$-structure induced by the $E_2$-structure of $\cB$.)
\end{remark}

\subsection{Components of \texorpdfstring{$\cZ(\Sigma \cB)$}{Z(Sigma B)}}\label{subsec.compZSB}
In this section, we prove the following theorem. 

\begin{theorem} \label{thm:componentcount}
Let $\cB$ be a braided fusion $1$-category. Then, the number of components of $\cZ(\Sigma \cB)$ agrees with the number of simple transparent objects in $\cB$, i.e.  $|\pi_0 \cZ(\Sigma \cB)| = | \pi_0 \cZ_{2}(\cB)|$. 
\end{theorem}
In fact, versions of Theorem~\ref{thm:componentcount} hold for all nondegenerate braided fusion $2$-categories, and more generally for fusion $n$-categories $\cC$ with trivial centre in the sense of~\cite{2003.06663}. This owes to the existence of a nondegenerate \emph{$S$-matrix pairing} $\widetilde S: \pi_0 \Omega^k \cC \times \pi_0 \Omega^{n-k} \cC \to \bk$, which we will develop in future work~\cite{Smatrix}.
To keep our proofs in this paper elementary, we will somewhat reverse the narrative and conclude Theorem~\ref{thm:componentcount} from well-known results about braided fusion $1$-categories, and then use it in \S\ref{sec:Smatrix} to show that the $S$-matrix of $\cZ(\Sigma \cB)$ is nondegenerate.

\begin{proof}[Proof of Theorem~\ref{thm:componentcount}]
Consider the braided module category $\cZ(\cB)$ over $\cB$ induced by the braided monoidal functor $\cB \to \cZ(\cB)$ (i.e.\ the canonical Lagrangian object $\cL$ in $\cZ(\Sigma \cB)$ from Example~\ref{eg:canonicalLag}). It follows from Corollary~\ref{cor.Lgenerates} that $\cZ(\cB)$ is a generating object of $ \BrModr{\cB} \cong \cZ(\Sigma \cB)$, and hence that its direct sum decomposition into indecomposable braided $\cB$-module categories contains at least one simple object from each component of $\BrModr{\cB}$. 

Recall from \cite[Corollary 4.4]{MR3022755} that the centralizer of the symmetric subcategory $\cZ_{2}(\cB) \hookrightarrow \cB \hookrightarrow \cZ(\cB)$ in $\cZ(\cB)$ is the monoidal subcategory $\cB^\rev\boxtimes_{\cZ_{2}(\cB)}\cB \subseteq \cZ(\cB)$ induced by the mutually centralizing subcategories $\cB \hookrightarrow \cZ(\cB), b\mapsto (b, \beta_{-, b})$ and $\cB^{\rev} \hookrightarrow \cZ(\cB), b\mapsto (b, \beta^{-1}_{b,-})$ of $\cZ(\cB)$.
Define the following equivalence relations on the indecomposable $\cB$-module summands of $\cZ(\cB)$:
\begin{enumerate}
\item $\cM\sim_1 \cN$ if $\cM$ and $\cN$ are in the same component of $\BrModr{\cB}$.
\item $\cM \sim_2 \cN $ if there exists a braided $\cB$-module functor $F: \cZ(\cB) \to \cZ(\cB)$ whose $\cM \to \cN$ component does not vanish. 
\item $\cM \sim_3 \cN$ if $\cM$ and $\cN$ are in the same $\cB\boxtimes_{\cZ_{2}(\cB)} \cB^{\rev}$-module summand of $\cZ(\cB)$. 
\end{enumerate}
We claim that all these equivalence relations agree. Indeed, equivalence of $\sim_1$ and $\sim_2$ is obvious, as any non-zero braided $\cB$-module functor $\cM \to \cN$ may be extended by zero to a braided $\cB$-module functor $\cZ(\cB) \to \cZ(\cB)$. 

To see the equivalence of $\sim_2$ and $\sim_3$ note that for any object $b\in \cB^\rev \subseteq \cZ(\cB)$, tensoring $b \otimes -: \cZ(\cB) \to \cZ(\cB)$ defines a braided $\cB$-module functor. By definition, if $\cM \sim_3 \cN$, i.e. $\cM, \cN \subseteq \cZ(\cB)$ are sub-$\cB$-modules in the same $\cB^{\rev}\boxtimes_{\cZ_{2}(\cB)}\cB$-module summand of $\cZ(\cB)$, then there exists an object $b \in \cB^{\rev}$ such that $(b \otimes \cM ) \cap \cN \neq 0$. (This notation is short-hand for: for every object $m \in \cM \subseteq \cZ(\cB)$, there exists an object $n \in \cN \subseteq \cZ(\cB)$ such that $\Hom_{\cZ(\cB)}(b \otimes m, n) \neq 0$.) In particular, $b \otimes -: \cZ(\cB) \to \cZ(\cB)$ is a braided $\cB$ module functor whose $\cM \to \cN$ component is non-zero and hence $\cM \sim_{2} \cN$.

Conversely, suppose that $\cM\sim_2 \cN$, i.e. that there exists a braided $\cB$-module functor $F:\cZ(\cB) \to \cZ(\cB)$ with non-zero $\cM\to \cN$ component. 
Passing through the equivalence between braided module categories and half-braided algebras, it follows from Proposition~\ref{prop:annular} that the monoidal functor
 \[\left( \cZ(\cB), \otimes_{\mathrm{ann}} \right) \to \End_{\BrModr{\cB}}( \cZ(\cB))\hspace{1cm} (b, \gamma) \mapsto ( (b, \gamma) \otimes_{\mathrm{ann}} -)\] is an equivalence. Recall from Remark~\ref{rem.ann} that $\left( \cZ(\cB), \otimes_{\mathrm{ann}} \right)$ is the (1-categorical) idempotent completion of the annular category $\mathrm{Ann}(\cB)$. It follows that any braided $\cB$-module functor  $F:\cZ(\cB) \to \cZ(\cB)$ arises from splitting an idempotent on a braided $\cB$-module functor of the form $ b \otimes - : \cZ(\cB) \to \cZ(\cB)$ for $b \in \cB^{\rev} \subseteq \cZ(\cB)$. Hence, if there is a functor $F: \cZ(\cB) \to \cZ(\cB)$ with a non-zero $\cM \to \cN$ component, then there is a $b\in \cB^{\rev}$ such that $(b \otimes \cM)\cap \cN \neq 0 $ and hence $\cM \sim_3 \cN$. 

As $\cZ(\cB)$ is generating in $\BrModr{\cB}$, the set of equivalence classes of $\sim_1$ is the set of components $\pi_0 \BrModr{\cB}$. On the other hand, as $\cB^{\rev}\boxtimes_{\cZ_{2}(\cB)}\cB \subseteq \cZ(\cB)$ is the centralizer of $\cZ_{2}(\cB) \subseteq \cZ(\cB)$, it follows from~\cite[Corollary 3.6]{DGNO} that $\sim_3$ has precisely $|\pi_0 \cZ_{2}(\cB)|$ equivalence classes. 
\end{proof}

\begin{remark} \label{rem:blocksofmultifusion}
As $\Sigma(\cZ(\cB), \otimes_{\mathrm{ann}}) \cong \cZ(\Sigma \cB)$, there is a bijection between components of $\cZ(\Sigma \cB)$ and blocks of the multifusion category $(\cZ(\cB), \otimes_{\mathrm{ann}})$.
Therefore, Theorem~\ref{thm:componentcount} is the statement that the multifusion category $(\cZ(\cB), \otimes_{\mathrm{ann}})$ has precisely as many blocks as there are simple transparent objects in $\cB$. This generalizes~\cite[Theorem 5.3]{2004.09611} which observes that for non-degenerate $\cB$, $(\cZ(\cB), \otimes_{\mathrm{ann}})$ is equivalent to the `matrix multifusion category' $\End(\cB)$ of endo-functors of $\cB$. 
\end{remark}

\subsection{The framed $S$-matrix of \texorpdfstring{$\cZ(\Sigma \cB)$}{Z(Sigma B)}}\label{sec:Smatrix}

The $S$-matrix plays a prominent role in the theory of modular tensor categories. Defined as the evaluation of a Hopf link --- two unknotted linked circles in $3$-space ---  it defines a pairing on simple objects which records their mutual braiding statistics. In this section, we discuss a categorification of this concept. Intuitively, in an appropriately ``oriented'' (or ``ribbon'') braided fusion $2$-category $\cC$, one can define a pairing between objects $A$ of $\cC$ and endomorphisms $b$ of the tensor unit by linking a $b$-labelled circle with an $A$-labelled $2$-sphere in ambient $4$-space. When restricted to simple objects and simple endomorphisms of the tensor unit, this results in an ``$S$-matrix'' for $\cC$ which one expects to be nondegenerate (as a pairing between components of $\cC$ and simple endomorphisms of the tensor unit) if and only if $\cC$ is.

A first difficulty in carrying out this construction, is that the braided fusion $2$-categories in this paper are not equipped with any sort of ribbon structure or orientation. In the setting of braided fusion $1$-categories $\cB$, this issue is overcome by defining a \emph{framed $\widetilde{S}$-matrix}  as follows (see~\cite[Section 2.8.1]{DGNO} for more details).
For each simple object $b$ of $\cB$ choose right and left duals 
\[\ev_b: b\otimes b^* \to I, \qquad \coev_b: I \to b^* \otimes b, \qquad \ev_{{}^*b}: {}^*b \otimes b \to I, \qquad \coev_{{}^*b}: I \to b \otimes {{}^*b},
\]
depicted as right- and left- turning cups and caps as in Remark~\ref{conv:catno}.
Semisimplicity of $\cB$, together with simplicity of the unit object $I$, guarantees that there exists an isomorphism $\alpha_b:{}^*b\to b^*$ and moreover that for any choice of such an isomorphism $\alpha_b$ the following scalars $d_+\!(b)$ and $d_{-}(b)$ are nonzero (see e.g.~\cite[\S 2.1]{MR2183279}): 
\begin{equation}\label{eqn:dimensions}
d_+\!(b):= \begin{tz}[std]
\draw[string, arrow data={0.5}{<}] (0,0) circle (1cm);
\node[dot] at (1,0){};
\node[label, right, scale=1.5] at (1,0){$\alpha_b$};
\node[label, left, scale=1.5] at (-1,0) {$b$};
\end{tz}
\hspace{1cm}
d_-\!(b):= \begin{tz}[std]
\draw[string, arrow data={0}{>}] (0,0) circle (1cm);
\node[dot] at (-1,0){};
\node[label, left, scale=1.5] at (-1,0){$\alpha_b^{-1}$};
\node[label, right, scale=1.5] at (1,0){$b$};
\end{tz}
\end{equation}

For $b, b'$ in $\cB$, define
\begin{equation}\label{eq:framedS1cat}
\widetilde{S}_{b, b'} :=\frac{1}{d_-\!(b) \, d_+\!(b')}~~
\begin{tz}[std, yscale=-1]
\draw[string, arrow data={1}{<}] (0.5,0) to [out=up, in=down] (0,1); 
\draw[braid = main, arrow data={0.5}{<}](0,0) to [out=up, in=down] (0.5, 1) to [out=up, in=down] (0,2);
\draw[braid=main] (0,1) to [out=up, in=down] (0.5,2);
\draw (0.5, 2) to [out=up, in=up] (1.5,2) to (1.5,0) to [out=down, in=down] (0.5,0);
\draw (0,2) to [out=up, in=up] (-1,2) to (-1,0) to [out=down, in=down] (0,0);
\node[label, scale=1.5, left] at (0,1) {$b'$};
\node[label, scale=1.5, right] at (0.5,1) {$b$};
\node[dot] at (1.5, 1) {};
\node[dot] at (-1,1) {};
\node[label, scale=1.5, right] at (1.5, 1) {$\alpha_{b'}$};
\node[label, scale=1.5, left] at (-1, 1) {$\alpha_{b}^{-1}$};
\end{tz}
\end{equation}
Note that while $d_+$ and $d_-$ both depend on the choice of duality data and isomorphism $\alpha$, these choices cancel in~\eqref{eq:framedS1cat}, and $\widetilde{S}_{b, b'}$ only depends on the isomorphism class of $b$ and $b'$.  Equivalently, without having to choose duality data for $b$,  $\widetilde{S}_{b,b'}$ may also be defined as the proportionality factor:
\[
\left \langle 
\frac{1}{d_+\!(b')}~
\begin{tz}[std, yscale=-1]
\draw[string, arrow data={1}{<}] (0.5,0) to [out=up, in=down] (0,1); 
\draw[braid = main, arrow data={0.5}{<}](0,0) to [out=up, in=down] (0.5, 1) to [out=up, in=down] (0,2);
\draw[braid=main] (0,1) to [out=up, in=down] (0.5,2);
\draw (0.5, 2) to [out=up, in=up] (1.5,2) to (1.5,0) to [out=down, in=down] (0.5,0);
\draw (0,2) to (0,2.5);
\draw (0,0) to (0,-0.5);
\node[label, scale=1.5, left] at (0,1) {$b'$};
\node[label, scale=1.5, right] at (0.5,1) {$b$};
\node[dot] at (1.5, 1) {};
\node[label, scale=1.5, right] at (1.5, 1) {$\alpha_{b'}$};
\end{tz}
\right \rangle 
\]

In the same way, we can define a framed $\widetilde{S}$-matrix for a braided fusion $2$-category $\cC$.
Given an object $A \in \cC$ and a $1$-endomorphism $b:I \to I$ in $\cC$, there are over- and under-braiding $2$-isomorphisms $\br_{A, b}: \id_A\boxtimes b \to b \boxtimes \id_A$ and $\br_{b,A}: b\boxtimes \id_A \to \id_A \boxtimes b$ of $b$ and $A$, which arise as naturality data for the braiding $1$-isomorphisms $\br_{-, A}$ and $\br_{A,-}$.
Graphically expressed as movies in $3$-space (projected to $\bR^2)$, these $2$-isomorphisms correspond to the following motion of a particle $b$ around a string $A$:
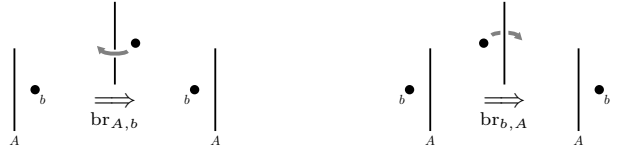
\[
\begin{tz}[std]
\draw[string] (0,0) to (0,2);
\node[dot] at (0.5,1){};
\node[label, below] at (0,0) {$A$};
\node[label, below right] at (0.5,1) {$b$};
\end{tz}
\hspace{1em}
\stackrel[\br_{A,b}]{
\begin{tz}[std]
\clip (-0.75,-0.25) rectangle (0.75, 2);
\node[dot] at (0.5,1){};
\draw[string] (0,0) to (0,2);
\draw[braid=main,-{Triangle[width=3pt,length=3pt]},gray, line width=1.5pt, shorten <= 0.15cm, shorten >= 0.05cm] (0.5, 1) to [out=-135, in=-45, looseness=1.] (-0.5,1);
\end{tz}
}{\Longrightarrow}
\hspace{1em}
\begin{tz}[std]
\draw[string] (0,0) to (0,2);
\node[dot] at (-0.5,1){};
\node[label, below] at (0,0) {$A$};
\node[label, below left] at (-0.5,1) {$b$};
\end{tz}
\hspace{1ex}
\qquad\qquad\qquad
\begin{tz}[std]
\draw[string] (0,0) to (0,2);
\node[dot] at (-0.5,1){};
\node[label, below] at (0,0) {$A$};
\node[label, below left] at (-0.5,1) {$b$};
\end{tz}
\hspace{1em}
\stackrel[\br_{b,A}]{
\begin{tz}[std]
\clip (-0.75,-0.25) rectangle (0.75, 2);
\node[dot] at (-0.5,1){};
\draw[string,-{Triangle[width=3pt,length=3pt]},gray, line width=1.5pt,shorten <= 0.15cm, shorten >= 0.05cm] (-0.5, 1) to [out=45, in=135, looseness=1.] (0.5,1);
\draw[braid=main] (0,0) to (0,2);
\end{tz}
}{\Longrightarrow}
\hspace{1em}
\begin{tz}[std]
\draw[string] (0,0) to (0,2);
\node[dot] at (0.5,1){};
\node[label, below] at (0,0) {$A$};
\node[label, below right] at (0.5,1) {$b$};
\end{tz}
\hspace{1ex}
,
\]
As diagrams in 4-dimensional ``spacetime'' (projected to $\bR^3$), we depict these $2$-morphisms as follows, respectively (where $A$ now appears as a surface and $b$ as a line):
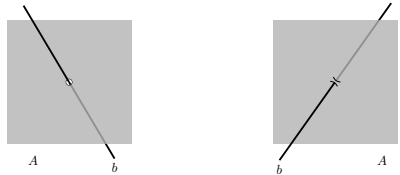
\[
    \begin{tz}[td]
    \draw[string, on layer=back] (2,2,-0.75) to node[circ mask point, pos=1](A){} (1.5, 1, 1.5);
  \begin{scope}[xzplane=1.]
\circcliparoundone{A}{ \path[surface] (0,0) rectangle (3,3);}
 \end{scope}
 \draw[string, on layer=front] (1.5, 1, 1.5) to   (1.,0, 3.75);
\node[label, below] at (2,2,-0.75) {$b$};
\node[label, right] at (1,0,0){$A$};
 \end{tz}
 \hspace{2cm}
   \begin{tz}[td]
\path (0.75, 0,0) to node[box mask point, pos=0.5] (A){} (2.25, 2,3); %\node[box mask point]  at (1.5, 1, 1.5) (A) {};
  \begin{scope}[xzplane=1.]
\path[surface] (0,0) rectangle (3,3);
 \end{scope}
  \boxcliparoundone{A}{\draw[string,] (0.75,0,0) to (1.5, 1, 1.5);}
\boxcliparoundone[back]{A}{ \draw[string, on layer =back] (1.5, 1, 1.5) to  (2.25,2, 3);}
\boxclipshadow{A}
\node[label, below] at (0.75, 0,0) {$b$};
\node[label, right] at (3,0,0){$A$};
 \end{tz}
\]
The \define{full braid} $\id_A \boxtimes b \To \id_A \boxtimes b$  is the composite $\br_{b, A} \cdot \br_{A, b}$.
Assume now that $b$ is simple, and choose duality data for $b$ and a $2$-isomorphism $\alpha_b: {{}^*b} \To b^*$. 
Then we may ``trace out'' $b$, 
 defining the following $2$-endomorphism  $\widetilde{R}_{A, b}: \id_A \To \id_A$ (as usual omitting coherence isomorphisms):
 \begin{equation*}
 d_+\!(b)^{-1} \left( \vphantom{\frac{a}{b}}(\id_A \boxtimes \ev_b) \cdot\left((\br_{b, A} \cdot \br_{A, b} )\boxtimes\alpha_b\right) \cdot (\id_A \boxtimes \coev_{*b})) \right) ,
 \end{equation*}
 or graphically: 
\begin{equation}\label{eq:R}
\widetilde{R}_{A, b} = 
d_+\!(b)^{-1} 
 \begin{tz}[td]
 \clip (-0.25,0,0.15) rectangle (3.5,2,3);
    \path[ on layer=back] (2.,2,0) to node[circ mask point, pos=1](A){} (1.5, 1, 1.5);
    \path (1.5, 1,2.5) to node[box mask point, pos=0.] (B){} (2, 2,3); 
  \begin{scope}[xzplane=1.]
\circcliparoundone{A}{ \path[surface] (-0.25,0) rectangle (3.,3.2);}
 \end{scope}
\boxcliparoundone[front]{B}{ \draw[string, on layer=front, arrow data={0.5}{>}] (1.5, 1, 1.5) to [out=135, in=-135] (1.5, 1, 2.5) ;}%    (0.75,0, 3);
\boxcliparoundone[back]{B}{ \draw[string, on layer=back, arrow data={0.4}{>}] (1.5, 1, 2.5) to [out=45, in=up, looseness=2.5] (3.5,1,2) to [out=down, in=-45, looseness=2.5] (1.5, 1, 1.5);}
\boxclipshadow{B}
\node[dot] at (3.5, 1,2 ){};
\node[label, right] at (3.5, 1,2) {$~\alpha_b$}; \end{tz}
\end{equation}
As before, the normalization factor $d_+\!(b)^{-1}$ ensures that this endomorphism is independent of the choice of duality data for $b$ and the isomorphism $\alpha_b :{}^*b \to b^*$. 

\begin{example}\label{ex:Rhba} Given a half-braided algebra $(A,\gamma)$ in a braided fusion $1$-category $\cB$ and a transparent simple object $b \in \cZ_{2}(\cB) = \Omega \cZ(\Sigma \cB)$, and using the braided monoidal structure on $\sHBA(\cB)$ described in \eqref{eq:tensor} and Lemma~\ref{lem:algebraiso}, the bimodule endomorphism $\widetilde{R}_{(A,\gamma), b} : {}_{A}A_{A} \To {}_{A}A_{A}$ unpacks to 
\begin{equation}\label{eq:Rhba}
d_+\!(b)^{-1}~ \left(\vphantom{\frac{a}{b}}(\id_A \otimes \ev_b) \cdot ((\gamma_{b,A} \cdot \br_{A,b}) \otimes \alpha_b) \cdot (\id_A \otimes \coev_{{}^*b})\right) ~= ~ d_+\!(b)^{-1}~
\begin{tz}[std]
\clip (-0.75, -0.5) rectangle (1.75, 2.5);
\draw[string] (0,0) to (0,2);
\draw[braid, ] (1,1) to [out=down, in=-45, looseness=2] (0, 0.5) to [out=135, in=down] (-0.4, 1);
\draw[string, , arrow data={0.01}{>}, arrow data={0.8}{>} ] (-0.4,1)  to [out=up, in =-135] (0, 1.5)  to [out=45,  in=up, looseness=2] (1, 1);
\node[smalldot] at (0,1.5) {};
\node[label,left] at (-0.5,0.9) {$b$};
\node[label,below] at (0,-0.) {$A$};
\node[dot] at (1,1){};
\node[label, right] at (1,1) {$~\alpha_b$};
\end{tz}.
\end{equation}
\end{example}

\begin{example} Analogous to Example~\ref{ex:Rhba}, given a braided $\cB$-module category $\cM_{\cB} \in \BrModr{\cB}$ and a transparent simple $b\in \cZ_{2}(\cB)$, the coefficients $(\widetilde{R}_{\cM, b})_m :m \to m$ of the $\cB$-module natural endomorphism $\widetilde{R}_{\cM, b}: \id_\cM \To \id_{\cM}$ are given as follows:
 \begin{multline}
 \label{eq:Rbrmod}
  d_+\!(b)^{-1}
   \left( m \cong m * I \to[m *  \coev_{{{}^*b}} ] m* (b\otimes {{}^*b})   \right.  
   \\
   \cong (m * b )* {{}^*b} \to[\sigma_{m, b} * \alpha_b]
   (m * b )* {b^*}
   \\
    \left. \cong m* (b\otimes {b^*}) 
   \to[m * \ev_{b} ] m*I \cong m   \right).
 \end{multline}
\end{example}

If the object $A$ in~\eqref{eq:R} is also simple, then $\widetilde{R}_{A,b}$, being an endo-$2$-morphism of the simple 1-morphism $\id_A$, is proportional to the identity, and we define $\widetilde{S}_{A,b} := \langle \widetilde{R}_{A,b} \rangle \in \bk$ to be the proportionality factor (so that $\widetilde{R}_{A,b} = \widetilde{S}_{A,b} \, \id_A$; compare Remark~\ref{notation:langlerangle}).

 \begin{lemma}\label{lemma.Sonlycomp}
 $\widetilde{S}_{A,b}$ only depends on the component of $A$ and descends to a function 
 \[\widetilde{S}:  \pi_0 \cC \times \pi_0 \Omega \cC   \to \bk
 \]
 which we call the \emph{framed $\widetilde{S}$-matrix} of $\cC$. \qed
 \end{lemma}
 \begin{proof}
  This is a direct consequence of the fact that for any $1$-morphism $f:A\to B$ and any $1$-morphism $b:I \to I$, naturality of the braiding implies that 
 \[
    \begin{tz}[td]
\path (0.75, 0,0) to node[box mask point, pos=0.5] (A){} (2.25, 2,3); %\node[box mask point]  at (1.5, 1, 1.5) (A) {};
  \begin{scope}[xzplane=1.]
\path[surface] (-0.,0) rectangle (3.,3);
 \end{scope}
 \draw[1mor] (2,1, 0) to node[mask point, pos=0.73](X){} (2,1,3);
  \boxcliparoundone{A}{\draw[string,] (0.75,0,0) to (1.5, 1, 1.5);}
\boxcliparoundone[back]{A}{ \cliparoundone[back]{X}{\draw[string, on layer =back] (1.5, 1, 1.5) to  (2.25,2, 3);}}
\boxclipshadow{A}
\node[label, below] at (0.75, 0,0) {$b$};
\node[label, below] at (2,1,0){$f$};
%\draw[1mor](0,0,0) to (0,0,3);
%\node[below] at (1.5, 1, -0.75) {\footnotesize$\br_{f,A}$};
 \end{tz}
 \qquad
\hspace{0.6cm}=\hspace{1.3cm}
     \begin{tz}[td]
\path (0.75, 0,0) to node[box mask point, pos=0.5] (A){} (2.25, 2,3); %\node[box mask point]  at (1.5, 1, 1.5) (A) {};
  \begin{scope}[xzplane=1.]
\path[surface] (-0.,0) rectangle (3.,3);
 \end{scope}
  \boxcliparoundone{A}{\draw[string] (0.75,0,0) to node[mask point, pos=0.63](X){} (1.5, 1, 1.5);}
\boxcliparoundone[back]{A}{ \cliparoundone[back]{X}{\draw[string, on layer =back] (1.5, 1, 1.5) to  (2.25,2, 3);}}
\boxclipshadow{A}
 \cliparoundone{X}{\draw[1mor] (1,1, 0) to  (1,1,3);}
\node[label, below] at (0.75, 0,0) {$b$};
\node[label, below] at (1,1,0){$f$};
%\draw[1mor](0,0,0) to (0,0,3);
%\node[below] at (1.5, 1, -0.75) {\footnotesize$\br_{b,A}$};
 \end{tz}~~~~~.\qedhere
 \]
 \end{proof}

It turns out that $\cC$ is a nondegenerate braided fusion $2$-category if and only if this $\widetilde{S}$-matrix is invertible \cite{Smatrix}. 
As a special case, we will prove:
 
 \begin{theorem} \label{thm:invertibleSmatrix}The framed $\widetilde{S}$-matrix $\widetilde{S}: \pi_0 \cZ(\Sigma \cB) \times  \pi_0 \Omega \cZ(\Sigma \cB)  \to \bk$  of $\cZ(\Sigma \cB)$ is invertible. \end{theorem}

 \begin{proof} It follows from Theorem~\ref{thm:componentcount} that the framed $\widetilde{S}$-matrix of $\cZ(\Sigma \cB)$ is a square matrix. Therefore, it suffices to show that it has full rank $|\pi_0 \cZ_{2}( \cB)|$. We will prove this in terms of braided module categories. 
 
For any braided $\cB$-module category $\cM_{\cB}$, we define a \emph{braided $\cB$-module S-matrix} $\widetilde{S}^{\cM}_{m, b}$ indexed by simple objects $m\in \cM$ and simple transparent objects $b\in \cZ_{2}(\cB)$ as the proportionality factor $\widetilde{S}^{\cM}_{m, b} := \langle (\widetilde{R}_{\cM,b})_m \rangle$ where $(\widetilde{R}_{\cM,b})_m: m\to m$ is defined as in~\eqref{eq:Rbrmod}. (Note that while $\widetilde{S}^{\cM}_{m, b}$ and $(\widetilde{R}_{\cM,b})_m$ can also be defined for possibly non-transparent $b \in \cB$, the $(\widetilde{R}_{\cM, b})_m$ do not assemble into a $\cB$-module transformation in that case.) As $\widetilde{R}_{\cM, b}$ is a $\cB$-module transformation it follows that if $\cM_{\cB}$ is an indecomposable $\cB$-module category, then $\widetilde{R}_{\cM, b}$ is proportional to the identity natural transformation, and hence that the coefficient $\widetilde{S}_{\cM, b}$  of the framed S matrix of $\cZ(\Sigma \cB)$ for indecomposable $\cM_{\cB}$ and simple transparent $b \in \cZ_{2}(\cB)$ can be computed in terms of the braided $\cB$-module S-matrix of $\cM_{\cB}$: 
\[\widetilde{S}_{\cM, b} = \widetilde{S}^{\cM}_{m, b} = \widetilde{S}^{\cM}_{m', b} \text{ for all simple $m,m' \in \cM$.}\]
 In fact, it follows from Lemma~\ref{lemma.Sonlycomp} that if $m \in \cM$ and $m' \in \cM'$ are simple objects in indecomposable braided $\cB$-module categories $\cM$ and $\cM'$ which are in the same connected component of $\BrModr{\cB}$, then $\widetilde{S}^\cM_{m, b} = \widetilde{S}^{\cM'}_{m', b}$. 
 
 By Corollary~\ref{cor.Lgenerates}, the braided $\cB$-module category $\cZ(\cB)$ generates $\BrModr{\cB}$ and hence decomposes into a direct sum of indecomposable $\cB$-module categories, at least one from every component of $\BrModr{\cB}$. It follows from the previous paragraph that for simple transparent $b \in \cZ_{2}(\cB)$ and simple object $m\in \cZ(\cB)$ in an indecomposable $\cB$-module summand $\cM\subseteq \cZ(\cB)$ of $\cZ(\cB)$ which belongs to a component $C$ of $\BrModr{B}$ we have that $\widetilde{S}_{C, b} = \widetilde{S}^{\cZ(\cB)}_{m, b}$. Hence, the $|\pi_0(\cZ(\cB))| \times |\pi_0(\cZ_{2}(\cB))|$ matrix $\widetilde{S}^{\cZ(\cB)}_{ m \in \cZ(\cB), b \in \cZ_{2}(\cB)}$ is obtained from the $ |\pi_0 \cZ(\Sigma \cB)| \times |\pi_0(\cZ_{2}(\cB))| $ matrix $\widetilde{S}_{ C \in \pi_0 \cZ(\Sigma\cB), b \in \cZ_{2}(\cB)}$ by repeating columns. In particular, these two matrices have the same rank.

But $\widetilde{S}^{\cZ(\cB)}_{m \in \cZ(\cB), b \in \cZ_{2}(\cB)}$ is the framed $\widetilde{S}$-matrix (as in~\eqref{eq:framedS1cat}) of the braided fusion $1$-category $\cZ(\cB)$ with columns restricted to $\cZ_{2}(\cB) \subseteq \cZ(\cB)$. As $\cZ(\cB)$ is a nondegenerate braided fusion $1$-category, it follows from~\cite[Theorem 3.4 and Corollary 3.6]{DGNO} that $\mathrm{rk}(\widetilde{S}^{\cZ(\cB)}) = | \pi_0(\cZ_{2}(\cB))|$. 
\end{proof}

\begin{remark}\label{rem:Rinvertible} If $b\in \cZ_{2}(\cB)$ is an \emph{invertible} object, then 
\[\widetilde{R}_{A, b}:= \left(\id_A \boxtimes \ev_b\right) \cdot \left(\left(\br_{b,A}\cdot \br_{A,b}\right) \boxtimes \id_b\right) \cdot \left(\id_A \boxtimes(\ev_b)^{-1}\right) :  \id_A \To \id_A
\] for any choice of right dual $\ev_b: b\otimes b^* \to I$ and hence it follows that for invertible $b, c \in \cZ_{2}(\cB)$
\[\widetilde{R}_{A, b} \cdot \widetilde{R}_{A, c} = \widetilde{R}_{A, b\otimes c}.\]
 In particular,  if every simple object in $\cZ_{2}(\cB)$ is invertible --- inducing an abelian group structure on $\pi_0 \cZ_{2}(\cB)$ --- then $\widetilde{S}_{A, -}: \pi_0 \cZ_{2}(\cB) \to \bk$ is a group homomorphism and it follows from Theorem~\ref{thm:invertibleSmatrix} that the assignment $A\mapsto \widetilde{S}_{A,-}$ defines a bijection 
 \[\pi_0 \cZ(\Sigma \cB) \to  \Hom(\pi_0 \cZ_{2}(\cB), \bk^\times)
 \] between the set of components of $\cZ(\Sigma\cB)$ and the set of group homomorphisms $\pi_0 \cZ_{2}(\cB) \to \bk^\times$. 
 Explicitly, a simple object $A \in \cZ(\Sigma \cB)$ is in the component corresponding to a group homomorphism $\phi: \pi_0 \cZ_{2}(\cB) \to \bk^\times$ if and only if the full braid of $A$ and $b$ is proportional to the identity with proportionality factor $\phi(b)$:
 \[\br_{b,A} \cdot \br_{A, b}  = \phi(b)~ \id_{\id_A \boxtimes b}: \id_{A} \boxtimes b \To \id_{A} \boxtimes b.\]
\end{remark}

\begin{remark}Generalizing Remark~\ref{rem:Rinvertible}, any choice of ribbon structure on $\cZ_{2}(\cB)$ (with possibly non-invertible objects) induces a canonical normalization $R_{A,b}: \id_A \To \id_A$ defined as in~\eqref{eq:R} with $\alpha_b: {}^*b \to b$ the canonical induced pivotal structure but \emph{without} the normalization factor $d_+\!(b)^{-1}$. In other words, $R_{A, b} = \dim(b) \widetilde{R}_{A, b}$. With this normalization, the proportionality factors $S_{A, b} = \langle R_{A, b} \rangle = \dim(b) \widetilde{S}_{A, b}$ induce, for every simple object $A\in \cZ(\Sigma \cB)$, a ring homomorphism $S_{A,-}: K_0(\cZ_{2}(\cB))\to \bk$ from the Grothendieck ring of $\cZ_{2}(\cB)$ to the ground field $\bk$. As $K_0(\cZ_{2}(\cB)) \otimes_{\bZ} \bk$ is a commutative semisimple $\bk$-algebra~\cite[Corollary 3.7.7]{EGNO} it follows from Theorem~\ref{thm:invertibleSmatrix} that this induces a bijection (dependent on the choice of ribbon structure) between the set $\pi_0\cZ(\Sigma \cB)$ of components and the set $\mathrm{Spec}_\bk(K_0(\cZ_{2}(\cB)))$ of ring homorphisms $K_0(\cZ_{2}(\cB)) \to \bk$. Under this bijection, the identity component is sent to the ring homomorphism $\dim: K_0(\cZ_{2}(\cB)) \to \bk$.
\end{remark}

\section{Proof of the Main Theorem} \label{sec.proof}
We turn now to the proof of the \MainTheorem. 
Our strategy is directed by Theorem~\ref{thm.BCs}: in order to find a minimal nondegenerate extension of a slightly degenerate braided fusion category $\cB$, it suffices to find an equivalence of braided fusion $2$-categories $\cZ(\Sigma\cB) \cong \cZ(\Sigma\,\sVec)$.

\subsection{The braided fusion $2$-categories \texorpdfstring{$\cS$}{S} and \texorpdfstring{$\cT$}{T}}\label{subsec:SandT}

Suppose that $\cB$ is a slightly degenerate braided fusion 1-category. The results in Section~\ref{sec.2cats} constrain the structure of the braided fusion 2-categories which can arise as $\cZ(\Sigma\cB)$. The goal of this section is to classify the possible equivalence classes of braided fusion 2-categories consistent with these constraints.

For this, we will need a few facts and notation about the cohomology of the Eilenberg--MacLane spaces $K(\bZ_2,n)$:
\begin{itemize}
\item The cohomology ring $\H^\bullet( K(\bZ_2, n); \bZ_2)$ is freely generated by a degree-$n$ generator $t_n$
under cup products and the action of the Steenrod operators $\Sq^i$, obeying the usual compatibility relations~\cite{MR60234}; explicitly, as a commutative $\bZ_2$-algebra, it is the polynomial ring on expressions of the form $\Sq^{k_1} \cdots \Sq^{k_m} t_n \in \H^{\sum_j k_j + n} (K(\bZ_2, n); \bZ_2)$ with $k_j \geq 2 k_{j+1}$ and with $\sum_{j=1}^m (k_j - 2 k_{j+1}) < n$  (where we set $k_{m+1}=0$).
\item The desuspension map $\Omega : \H^{\bullet+1}(K(\bZ_2, n+1); \bZ_2) \to \H^{\bullet}(K(\bZ_2, n); \bZ_2)$ takes $t_{n+1} \mapsto t_n$ and commutes with the Steenrod operators $\Sq^i$, but it is not a ring homomorphism.
\item Since $K(\bZ_2, n)$ is rationally acyclic, its cohomology with $\bk^\times$ coefficients is independent of our choice of algebraically closed characteristic zero field $\bk$: any choice of inclusion $\overline\bQ \to \bk$ induces an isomorphism $\H^\bullet(K(\bZ_2,n); \overline\bQ^\times) \to \H^\bullet(K(\bZ_2,n);\bk^\times)$ in positive degrees $\bullet>0$. 
\item Most of the low-degree cohomology groups $\H^\bullet(K(\bZ_2,n);\bk^\times)$ that we will use are computed in \cite{MR65162}, and conveniently summarized in Examples~2.2 and~2.3 of \cite{DN}. (That paper uses the notation $\H^i_{\mathrm{br}}(A; -) := \H^{i+1}(K(A,2); -)$ and $\H^i_{\mathrm{syl}}(A; -) := \H^{i+2}(K(A,3); -)$ for $i>0$.) Those descriptions in particular allow us to read off the map on coefficients $t \mapsto (-1)^t : \H^\bullet(K(\bZ_2,n);\bZ_2) \to \H^\bullet(K(\bZ_2,n);\bk^\times)$ and hence to give names to elements in the $\bk^\times$-cohomology.
\end{itemize}

\begin{lemma}\label{lem:aux} Let $K(\bZ_2, 3) \to X \to K(\bZ_2,2)$ denote the non-trivial extension classified by $\Sq^2t_2 = t_2^2\in \H^4(K(\bZ_2, 2); \bZ_2) \cong \bZ_2$. Then, the induced maps on $\bk^\times$-cohomology form a short exact sequence 
\[0 \to \H^5(K(\bZ_2, 2); \bk^\times) \to \H^5(X; \bk^\times) \to \H^5(K(\bZ_2, 3) ;\bk^\times)\to 0.
\]
Similarly, the extension $K(\bZ_2,4) \to \rB X \to K(\bZ_2,3)$, classified by $\Sq^2 t_3 \in \H^5(K(\bZ_2, 3); \bZ_2)$, which characterizes the unique delooping of $X$ induces a short exact sequence
\[0 \to \H^6(K(\bZ_2, 3); \bk^\times) \to \H^6(\rB X; \bk^\times) \to \H^6(K(\bZ_2, 4) ;\bk^\times)\to 0.
\]
Since the desuspension morphisms
\[\Omega: \H^6(K(\bZ_2, 3); \bk^\times) \to \H^5(K(\bZ_2, 2); \bk^\times)\hspace{0.75cm} \Omega: \H^6(K(\bZ_2, 4); \bk^\times) \to \H^5(K(\bZ_2, 3); \bk^\times)\]
are isomorphisms, so is the desuspension morphism $\Omega: \H^6(\rB X; \bk^\times) \to \H^5(X; \bk^\times)$.
\end{lemma}
\begin{proof}
Let us inspect the Serre spectral sequence for $X$:
$$ E_2^{i,j} = \H^i(K(\bZ_2,2); \H^j(K(\bZ_2,3); \bk^\times)) \Rightarrow \H^{i+j}(X; \bk^\times). $$
The $E_2$ page of this spectral sequence begins:
\begin{equation}\label{eqn:E22}
\begin{tikzpicture}[anchor=base,xscale=1.5, baseline=(name.base)]
\path
(0,0) node {$\bk^\times$} ++(.75,0) node {$0$} ++(.75,0) node {$\bZ_2$} ++(.75,0) node {$0$} ++(.75,0) node {$\bZ_4$} ++(.75,0) node {$\bZ_2$}   ++(.75,0) node {$\bZ_2$}   
(0,.5) node {$0$} ++(.75,0) node {$0$} ++(.75,0) node {$0$} ++(.75,0) node {$0$} ++(.75,0) node {$0$} ++(.75,0) node {$0$}
(0,1) node {$0$} ++(.75,0) node {$0$} ++(.75,0) node {$0$} ++(.75,0) node {$0$} ++(.75,0) node {$0$} 
(0,1.5) node {$\bZ_2$} ++(.75,0) node {$0$} ++(.75,0) node {$\bZ_2$} ++(.75,0) node {$\bZ_2$}
(0,2) node {$0$} ++(.75,0) node {$0$} ++(.75,0) node {$0$}
(0,2.5) node {$\bZ_2$} ++(.75,0) node {$0$}
;
\draw[->] (-.75,-.125) -- ++(6,0);
\draw[->] (-.4,-.5) -- ++(0,3.5);
\path  (0,-.5) node {$0$} ++(.75,0) node {$1$} ++(.75,0) node {$2$} ++(.75,0) node {$3$} ++(.75,0) node {$4$} ++(.75,0) node {$5$} ++(.75,0) node {$6$} ++(.75,0) node {$i$}
(-.75,0) node {$0$} ++(0,.5) node {$1$} ++(0,.5) node {$2$} ++(0,.5) node {$3$} ++(0,.5) node {$4$} ++(0,.5) node {$5$} ++(0,.5) node {$j$}
;
\path (-1,1.5) node[anchor=east] (name) {$E_2^{i,j} \quad = \quad  $};
\end{tikzpicture}
\end{equation}
These groups are generated by the following elements:
\begin{equation}\label{eqn:E22names}
\begin{tikzpicture}[anchor=base, xscale=2.475, baseline=(name.base)]
\path
(0,0) node {$*$} ++(.75,0) node {$\cdot$} ++(.75,0) node {$(-1)^{t_2}$} ++(.75,0) node {$\cdot$} ++(.75,0) node {$\sqrt{({-1})^{t_2^2}}$} ++(.75,0) node {~\,\,$(-1)^{\Sq^2\Sq^1 t_2}$}   ++(.75,0) node {$(-1)^{t_2^3}$}   
(0,.5) node {$\cdot$} ++(.75,0) node {$\cdot$} ++(.75,0) node {$\cdot$} ++(.75,0) node {$\cdot$} ++(.75,0) node {$\cdot$} ++(.75,0) node {$\cdot$}
(0,1) node {$\cdot$} ++(.75,0) node {$\cdot$} ++(.75,0) node {$\cdot$} ++(.75,0) node {$\cdot$} ++(.75,0) node {$\cdot$} 
(0,1.5) node {$(-1)^{t_3}$} ++(.75,0) node {$\cdot$} ++(.75,0) node {$(-1)^{t_3t_2}$} ++(.75,0) node {$(-1)^{t_3 \,{\Sq^1}t_2}$}
(0,2) node {$\cdot$} ++(.75,0) node {$\cdot$} ++(.75,0) node {$\cdot$}
(0,2.5) node {$(-1)^{\Sq^2 t_3}$} ++(.75,0) node {$\cdot$}
;
\draw[->] (-.5,-.125) -- ++(5.5,0);
\draw[->] (-.4,-.5) -- ++(0,3.5);
\path  (0,-.5) node {$0$} ++(.75,0) node {$1$} ++(.75,0) node {$2$} ++(.75,0) node {$3$} ++(.75,0) node {$4$} ++(.75,0) node {$5$} ++(.75,0) node {$6$} ++(.5,0) node {$i$}
(-.5,0) node {$0$} ++(0,.5) node {$1$} ++(0,.5) node {$2$} ++(0,.5) node {$3$} ++(0,.5) node {$4$} ++(0,.5) node {$5$} ++(0,.5) node {$j$}
;
\path (1,1.5) node[anchor=east] (name) {$\phantom{E_2^{i,j}}$};
\end{tikzpicture}
\end{equation}
The name in bidegree $(i,j) = (4,0)$ is justified because the order-2 element in $E_2^{4,0} \cong \bZ_4$ is $(-1)^{t_2^2}$.
In bidegree $(5,0)$, note that $(-1)^{\Sq^2 \Sq^1 t_2} = (-1)^{t_2 \Sq^1 t_2} \in \H^5(K(\bZ_2,2);\bk^\times)$. (In other words, the map $\bZ_2=\H^4(K(\bZ_2, 2); \bZ_2) \to \H^4(K(\bZ_2,2); \bk^\times) = \bZ_4$ is not surjective, while the map $\bZ_2^2=\H^5(K(\bZ_2, 2); \bZ_2) \to \H^5 (K(\bZ_2, 2); \bk^\times) = \bZ_2$ is not injective.)

The nontrivial extension class $t_2^2 \in \H^4(K(\bZ_2,2); \bZ_2)$ classifying $X$ manifests as a nontrivial $d_4$ differential $t_3 \mapsto t_2^2$ on cohomology with $\bZ_2$ coefficients, which in turn gives nontrivial differentials $(-1)^{t_3} \mapsto (-1)^{t_2^2}$ and $(-1)^{t_3 t_2} \mapsto (-1)^{t_2^3}$. 
In particular, the $\bZ_2$ in bidegree $(i,j) = (2,3)$ does not survive to $E_\infty$. 

The $\bZ_2$ in bidegree $(5,0)$ survives just because of all of the $0$s in total degree $4$. The $\bZ_2$ in bidegree $(0,5)$ also survives. This is not automatic for a spectral sequence with this $E_2$ page. Rather, we know it is true because 
  the edge map in total degree~$5$ is the restriction map $\H^5(X; \bk^\times) \to E_2^{0,5}=\H^5(K(\bZ_2,3); \bk^\times)$, and we know that there is a class $ \alpha \in \H^5(X; \bk^\times)$ whose image in $\H^5(K(\bZ_2,3); \bk^\times) \cong \bZ_2$ is non-trivial: Indeed, consider the symmetric monoidal $2$-category $\Sigma \sVec$ as a braided monoidal $2$-category and observe that the induced delooping of its Picard groupoid $(\Sigma\sVec)^\times$ is a homotopy $4$-type $K(\bk^\times, 4)\cdot K(\bZ_2,3) \cdot K(\bZ_2, 2)$ whose quotient $K(\bZ_2, 3) \cdot K(\bZ_2,2)\cong X$ is the non-trivial extension (since $\br_{C,C} \cong e \boxtimes \id_{C\boxtimes C}$), and which is therefore classified by a class $\alpha \in \H^5(X; \bk^\times)$. The image of this class in $\H^5(K(\bZ_2, 3), \bk^\times)$ classifies the Picard groupoid $\Omega (\Sigma \sVec)^\times = \sVec^\times$ and hence is non-trivial.
  
  The argument for $\rB X $ is completely analogous: The only entries in total degree $6$ in its spectral sequence
  $$ E_2^{i,j} = \H^i(K(\bZ_2,3); \H^j(K(\bZ_2,4); \bk^\times)) \Rightarrow \H^{i+j}(\rB X; \bk^\times). $$
  are $E_2^{6,0} = \H^6(K(\bZ_2,3);\bk^\times) \cong \bZ_2$ generated by $(-1)^{\Sq^2 t_4}$
and
$E_2^{0,6} = \H^6(K(\bZ_2,4);\bk^\times) \cong \bZ_2$ generated by $(-1)^{\Sq^2\Sq^1 t_3}$. The former survives to $E_\infty$ simply because in this case the $E_2$ page already vanishes in total degree $5$. The latter survives to $E_\infty$ because we know that there is a class $ \alpha \in \H^6(\rB X;\bk^\times)$ whose restriction to $\H^6(K(\bZ_2,4);\bk^\times) \cong \bZ_2$ is nontrivial: this time, consider the symmetric monoidal 2-category $\Sigma\sVec$ as a sylleptic monoidal 2-category, restrict to its Picard groupoid $(\Sigma\sVec)^\times$, and observe that its delooping to a homotopy 5-type is an extension $K(\bk^\times, 5) \cdot K(\bZ_2,4) \cdot K(\bZ_2,3)$, classified by a class in $\H^6(\rB X;\bk^\times)$, whose restriction to $\H^6(K(\bZ_2, 4); \bk^\times)$ characterizes the symmetric braiding on $\sVec^\times$ and is hence non-trivial. 

The desuspension morphisms 
\[\bZ_2\langle (-1)^{\Sq^2\Sq^1 t_3} \rangle = \H^6(K(\bZ_2, 3); \bk^\times) \to \H^5(K(\bZ_2, 2); \bk^\times) = \bZ_2 \langle (-1)^{\Sq^2\Sq^1 t_2}\rangle 
\]
\[\bZ_2\langle (-1)^{\Sq^2 t_4}  \rangle = \H^6(K(\bZ_2, 4); \bk^\times) \to \H^5(K(\bZ_2, 3); \bk^\times) = \bZ_2 \langle (-1)^{\Sq^2 t_3}\rangle 
\]
send generators to generators and are therefore isomorphisms.
\end{proof}

We now turn to the main result of this section. To set it up, let us recall the constraints from Section~\ref{sec.2cats} on the structure of $\cC = \cZ(\Sigma\cB)$ for a slightly degenerate braided fusion 1-category $\cB$:
\begin{enumerate}
\item \label{condition:lines} By Lemma~\ref{lemma.centres}, there is an equivalence of symmetric fusion $1$-categories $\Omega\cC \cong  \sVec$. 
\item \label{condition:components} By Theorem~\ref{thm:componentcount}, $\cC$ has exactly two components. 
\item \label{condition:charges} Theorem~\ref{thm:invertibleSmatrix} relates $\pi_0 \cC$ and $\pi_0 \Omega\cC \cong \pi_0 \sVec$. Specifically, write $e$ for (a choice of) the non-identity simple object in $\sVec$; we will call it ``$e$'' because we will think of it as being \define{electrically charged} (under the gauge group $\bZ_2$ which arises from $\sVec$ under Tannakian reconstruction). By Remark~\ref{rem:Rinvertible}, the two components of $\cC$ are distinguished by their full braiding with $e$: objects in the identity component are transparent to $e$, whereas objects $X$ in the non-identity component enjoy \[\br_{e, X} \cdot \br_{X, e} = (-1) ~\id_{\id_X \boxtimes e}.\]
A physical object is a \define{magnetic monopole} if a test electron, when moved all the way around the object, comes back to its original position with some nontrivial phase.
We will therefore refer to the non-identity component of $\cC$ as the \define{magnetic component}, and its objects as \define{magnetically charged}.
\end{enumerate}

The following result was first proved in \cite{2011.11165}, relying on non-degeneracy of the braiding of $\cZ(\Sigma \cB)$ and on some higher Morita-categorical reasoning to establish conditions~\ref{condition:components} and~\ref{condition:charges}; the rest of the proof did not require higher Morita categories. We recall this rest of the proof both for completeness and because it contains many useful details about the $2$-categories in question.

\begin{theorem}[\cite{2011.11165}] \label{thm:SandT}
  Up to braided monoidal equivalence, there are at most two braided fusion $2$-categories satisfying the three conditions \ref{condition:lines}--\ref{condition:charges} above.
  
  These two braided fusion $2$-categories are built as in Example~\ref{eg:linearization} as twisted linearizations 
  \begin{align*}&\cS:= 2\Vec^\sigma[\cG]  &
  &\cT:= 2\Vec^{\tau}[\cG]
  \\ 
  \intertext{%
  of the $1$-groupoid $\mathcal{G} = K(\mathbb{Z}_2, 1) \times K(\mathbb{Z}_2, 0)$ with braided monoidal structure determined by the delooping $\mathrm{B}^2 \mathcal{G} = K(\bZ_2, 3) \times K(\bZ_2, 2)$ and the following twists in $\H^5(\mathrm{B}^2\cG; \bk^\times)$:}
  & \sigma := (-1)^{\Sq^2 t_3 + t_3 t_2} 
 & &\tau:= (-1)^{\Sq^2 t_3 + t_3 t_2 + \Sq^2 \Sq^1 t_2}
  \end{align*}
Here $t_n$ is the generator (over the Steenrod algebra) of $\H^\bullet(K(\bZ_2,n);\bZ_2)$, and we have used the K\"unneth formula and the map $t \mapsto (-1)^t$ on coefficients to map
\begin{multline*}
\H^5(K(\bZ_2, 3); \bZ_2) \oplus \bigl( \H^3(K(\bZ_2, 3); \bZ_2) \otimes \H^2(K(\bZ_2,2); \bZ_2)\bigr) \oplus \H^5(K(\bZ_2, 2); \bZ_2)
\\ \overset{\text{K\"unneth}}{\cong} \H^5(\mathrm{B}^2\cG; \bZ_2) \to[(-1)^-] \H^5(\mathrm{B}^2 \cG; \bk^\times)
\end{multline*}  
\end{theorem}
\begin{proof}
Let $\cC$ be a braided fusion $2$-category satisfying the three conditions \ref{condition:lines}--\ref{condition:charges}. (We will not assume that $\cC$ is equivalent to some $\cZ(\Sigma\cB)$).
We will first show that $\cC$ necessarily arises from linearizing a higher groupoid. Once having done so, the conditions \ref{condition:lines}--\ref{condition:charges} translate into conditions on that groupoid, and we will complete the proof by solving the corresponding purely homotopy-theoretic classification.

 By \cite[Theorem B]{2010.07950}, which applies to any fusion 2-category satisfying condition \ref{condition:lines}, the simple objects of $\cC$ are all invertible and each component contains two simple objects. 
The identity component, equivalent to $\Sigma\sVec$, is described in Example~\ref{eg:sVeclin}. It has two simple objects $I$ and $C$. Arbitrarily choose a magnetic simple object $M$ (i.e.\ a simple object in the non-identity component); the other simple magnetic object is then $C \boxtimes M$.

These choices provide a way to write $\cC$ as a linearization as in Example~\ref{eg:linearization}: there is a non-monoidal equivalence $\Sigma\sVec \simeq 2\Vec[\rB\bZ_2]$ (corresponding to the non-braided equivalence $\sVec \cong \Vec[\bZ_2]$), and hence a non-monoidal equivalence
$$ \cC \simeq 2\Vec[K(\bZ_2,1) \times K(\bZ_2,0)],$$
where the $\bZ_2$ in degree 1 is the group $\{I,e\} = \pi_0 \sVec$, and the $\bZ_2$ in degree 0 is the set $\{I,M\} = \pi_0 \cC$.
  For $\cC$ to admit a presentation as the (twisted) linearization of a groupal braided monoidal groupoid, it suffices for $M$ to satisfy the fusion rule $M^2 \cong I$ (see the last two paragraphs of Example~\ref{eg:linearization}). Since $M$ is invertible and $M^2$ is magnetically neutral, the only other possibility is $M^2 \cong C$.

 To see that $M^2 \cong I$, we will use some facts about not just the monoidal structure but also the braiding on $\cC$. Note first 
that the braided monoidal structure on (the full sub-2-category of $\cC$ spanned by) the identity component is forced by condition~\ref{condition:lines}: it is equivalent to $\Sigma\sVec \cong \cat{sAlg}$. The most important feature worth highlighting from Example~\ref{eg:sVeclin} is the braiding $\br_{C,C} \cong e \boxtimes\id_{C\boxtimes C} $ from equation \eqref{eqn.halfbraidc}. On the other hand, since $\br_{M,M}$ is isomorphic either to $\id_{M \boxtimes M}$ or $e \boxtimes \id_{M \boxtimes M}$, and since $e^4 \cong I$, we find that $\br_{M^2,M^2}$ is isomorphic to $\id_{M^2 \boxtimes M^2}$. (One must insert associators to compute $\id_{M^2 \boxtimes M^2}$ from $\br_{M,M}$, but they cancel in the final computation.) This rules out the possibility that $M^2 \cong C$.

In order to complete the proof, it suffices to classify the possible deloopings $\rB^2\cG$ of $\cG \cong K(\bZ_2,1) \times K(\bZ_2,0)$, and the possible Postnikov classes in $\H^5(\rB^2\cG;\bk^\times)$, which are consistent with the conditions~\ref{condition:lines}--\ref{condition:charges}. We will first argue that condition~\ref{condition:charges} is consistent only with $\rB^2\cG$ being the product $K(\bZ_2,3) \times K(\bZ_2,2)$. We will then study the cohomology of this product, and see that the classes $\sigma$ and $\tau$ are the only two possible classes consistent with our conditions.

Suppose, in search of a contradiction, that $\rB^2 \cG$ were not the product $K(\bZ_2,3) \times K(\bZ_2,2)$, but rather the unique nontrivial extension $X:= K(\bZ_2,3) \cdot K(\bZ_2,2)$ from Lemma~\ref{lem:aux}, classified by $\Sq^2 t_2 = t_2^2 \in \H^4(K(\bZ_2,2); \bZ_2) \cong \bZ_2$. Let $\rB^3 \cG = \rB X$ denote the unique delooping of $\rB^2\cG = X$, given by the extension  $K(\bZ_2,4) \cdot K(\bZ_2,3)$ classified by $\Sq^2 t_3 \in \H^5(K(\bZ_2,3);\bZ_2)$. It follows from Lemma~\ref{lem:aux}, that the desuspension map $\Omega: \H^6(\rB^3 \cG; \bk^\times) \to \H^5(\rB^2\cG; \bk^\times)$ is an isomorphism. 
In categorical language, this means that, for any twisting $\alpha \in \H^5(\rB^2\cG; \bk^\times)$, the braided fusion 2-category $2\Vec^\alpha[\cG]$ admits a (unique) sylleptic structure. But this violates condition~\ref{condition:charges}. Indeed, for any object $X$ and $1$-morphism $f:I \to I$ in a sylleptic monoidal $2$-category, it follows from unitality and naturality of the syllepsis that the full braid trivializes: $\br_{f, X} \cdot \br_{X,f} = \id_{\id_X \boxtimes f}$ (surfaces and lines unlink in $5$-dimensions). 

Since $\rB^2 \cG = K(\bZ_2,3) \times K(\bZ_2,2)$ is a product, the K\"unneth formula provides an isomorphism $\H^5(\rB^2\cG; \bk^\times) \cong \bZ_2^3$. Indeed, writing $\H_\bullet(-)$ for ordinary homology with integer coefficients and $\otimes$ and $\operatorname{Tor}$  for their versions over $\bZ$, a choice of product decomposition for $\rB^2 \cG$ selects a short exact sequence
$$ \bigoplus_{i+j=5} \H_i(K(\bZ_2,3)) \otimes \H_{j}(K(\bZ_2,2)) \to \H_5(\rB^2 \cG) \to \bigoplus_{i+j=4} \operatorname{Tor} \bigl(\H_i(K(\bZ_2,3)) , \H_{j}(K(\bZ_2,2))\bigr). $$
But the connectivity of $K(\bZ_2,3)$ and $K(\bZ_2,2)$ mean that the first term is $\H_5(K(\bZ_2,3)) \oplus \bigl(\H_3(K(\bZ_2,3)) \otimes \H_2(K(\bZ_2,2))\bigr) \oplus \H_5(K(\bZ_2,2))$ and the third term in this sequence vanishes; and $\H^5(\rB^2\cG; \bk^\times) = \hom(\H_5(\rB^2\cG), \bk^\times)$ because $\bk^\times$ is divisible. This calculation also shows that the map $t \mapsto (-1)^t$ on coefficients is a surjection $\H^5(\rB^2\cG; \bZ_2) \to H^5(\rB^2\cG; \bk^\times)$. As in \eqref{eqn:E22names}, $\H^5(\rB^2\cG; \bk^\times)\cong \bZ_2^3$ is therefore generated by the classes:
$$(-1)^{\Sq^2 t_3}, \qquad (-1)^{t_3 t_2}, \qquad (-1)^{\Sq^2 \Sq^1 t_2}.$$

Any Postnikov class consistent with condition~\ref{condition:lines} must have a nontrivial coefficient on $(-1)^{\Sq^2 t_3}$: The restriction $\H^5(\rB^2\cG; \bk^\times) \to \H^5(K(\bZ_2, 3); \bk^\times)$ encodes the non-trivial braiding $\beta_{e,e} = -\id_{e\otimes e}$ on $\Omega \cC =\sVec$. 
Moreover, the classes consistent with condition~\ref{condition:charges} must have a nontrivial coefficient on $(-1)^{t_3 t_2}$. If this coefficient was trivial, our braided fusion $2$-category would admit a sylleptic structure by an argument analogous to the one we gave above. There are therefore two choices for our $5$-cocycle, depending on whether its $(-1)^{\Sq^2\Sq^1 t_2}$ component is trivial or not. As the braided monoidal groupoid $\cG$ and its $5$-cocycle completely determine the braided monoidal structure on $\cC$ (see the last sentence of Example~\ref{eg:linearization}), there can be at most two such braided monoidal 2-categories. 
\end{proof}

The presentations of $\cS$ and $\cT$ as linearizations of braided groupoids as in Theorem~\ref{thm:SandT} depend on a choice of invertible object $M$ in the magnetic component (corresponding to the non-trivial element in $\pi_2\mathrm{B}^2\cG = \bZ_2$), and a choice of equivalence $M^2\cong I$ corresponding to the product decomposition $\mathrm{B}^2 \cG \cong K(\bZ_2, 3) \times K(\bZ_2, 2)$. 
In particular, Theorem~\ref{thm:SandT} does not yet imply that $\cS$ and $\cT$ are inequivalent as braided fusion $2$-categories, merely that any such equivalence must necessarily be incompatible with one of these choices. In fact, the space $\rB^2 \cG = K(\bZ_2,3) \times K(\bZ_2,2)$ has only one nontrivial automorphism, coming from the map $\Sq^1 : K(\bZ_2,2) \to K(\bZ_2,3)$, and this automorphism fixes the classes $\sigma$ and $\tau$. Hence, any potential equivalence between $\cS$ and $\cT$ must be incompatible with the choice of $M$ and cannot come from an equivalence of braided $1$-groupoids. 
In Corollary~\ref{cor:SneqT}, we will see that there is no such equivalence and that $\cS$ and $\cT$ are indeed inequivalent as braided fusion $2$-categories.

Further examining the proof of Theorem~\ref{thm:SandT}, note that in the classes   \begin{align*}& \sigma := (-1)^{\Sq^2 t_3 + t_3 t_2} 
 & &\tau:= (-1)^{\Sq^2 t_3 + t_3 t_2 + \Sq^2 \Sq^1 t_2},
  \end{align*} the factor $(-1)^{\Sq^2 t_3}$ is responsible for the braiding $\beta_{e,e} = -\id_{e\otimes e}$ in $\Omega \cS = \Omega\cT = \sVec$, while the factor $(-1)^{t_3t_2}$ causes the full braid $\br_{e,M} \cdot \br_{M, e}= - \id_{\id_M \boxtimes e}$ between the chosen magnetic object $M$ and $e$. 
The statement in the proof of Theorem~\ref{thm:SandT} that $\rB^2\cG \cong K(\bZ_2, 3) \times K(\bZ_2, 2)$ is isomorphic to a product of Eilenberg--Mac Lane spaces translates into categorical language as the existence of an isomorphism
\begin{equation}\label{eqn:braidingM}
  \br_{M,M} \cong \id_{M \boxtimes M}.
\end{equation}

  In the next section, we will develop an invariant which detects the last factor $(-1)^{\Sq^2\Sq^1 t_2}$ and hence distinguishes $\sigma$ and $\tau$. 
This last factor is already visible on the subcategory generated by the object $M$ (and is completely independent of its interaction with $e$). Indeed, the restriction of the class $\sigma|_{K(\bZ_2,2)} \in \H^5(K(\bZ_2, 2); \bk^\times) $ trivializes, whereas $\tau|_{K(\bZ_2,2)} = (-1)^{\Sq^2 \Sq^1 t_2} \in \H^5(K(\bZ_2, 2); \bk^\times) $. Abusing notation, we will write this still as ``$\tau$.'' These restrictions mean that $\cS$ and $\cT$ admit the following braided fusion sub-2-categories:
\begin{equation}\label{eq:subcategories} \cS \supset 2\Vec[\bZ_2], \qquad \cT \supset 2\Vec^\tau[\bZ_2].\end{equation}
Note that these inclusions depend on the presentation of $\cS$ and $\cT$ as linearizations, and hence depend both on a choice of invertible object $M$ in the magnetic component, and an equivalence $M^2 \cong I$ (or equivalently, a product decomposition $\rB^2 \cG = K(\bZ_2,2) \times K(\bZ_2,3)$).

\begin{lemma} \label{lemma:super2fibre}
  There are braided monoidal 2-functors $2\Vec[\bZ_2] \to \Sigma\sVec$ and $2\Vec^\tau[\bZ_2] \to \Sigma\sVec$.
\end{lemma}

Note that Eilenberg-MacLane spaces (for abelian groups) have unique deloopings, and the cohomology group $\H^5(K(\bZ_2,2); \bk^\times)$ is \define{stable} in the sense that the destabilization map $\H^{5+N}(K(\bZ_2,2+N); \bk^\times) \to \H^5(K(\bZ_2,2); \bk^\times)$ is an isomorphism for all $N \geq 0$. This implies that both the braided monoidal 2-categories  $2\Vec[\bZ_2]$ and $2\Vec^\tau[\bZ_2]$  admit (unique!)\ extensions to symmetric monoidal 2-categories. In particular, the functors to $\Sigma\sVec$ are symmetric monoidal for these unique symmetric monoidal structures.

\begin{proof}
  All three braided monoidal $2$-categories $2\Vec[\bZ_2]$, $2\Vec^\tau[\bZ_2]$, and $\Sigma\sVec$ arise from linearization:
  by Example~\ref{eg:sVeclin},
   the third is $2\Vec^\omega[\rB\bZ_2]$ with twisting $\omega = (-1)^{\Sq^2 t_3} \in \H^5(K(\bZ_2, 3); \bk^\times)$.
    Thus, to give maps $2\Vec[\bZ_2] \to \Sigma\sVec$ and $2\Vec^\tau[\bZ_2] \to \Sigma\sVec$, it suffices to give maps $f : K(\bZ_2,2) \to K(\bZ_2,3)$ such that the restriction $f^*\omega$ either trivializes or becomes cohomologous to $\tau$, respectively. The trivial map suffices for the first case, and for the second case we can take $f = \Sq^1$.
\end{proof}

\begin{remark}\label{rem:selfdualeg}
  In 2-categorical language, the braided monoidal $2$-functors $\text{trivial} : 2\Vec[\bZ_2] \to \Sigma\sVec$ and $F : 2\Vec^\tau[\bZ_2] \to \Sigma\sVec$ from 
  Lemma~\ref{lemma:super2fibre} both map $M$ to $I$. Their difference  is in their monoidality data.
   The nontrivial map $\Sq^1 : K(\bZ_2,2) \to K(\bZ_2,3)$ used for $F:2\Vec^\tau[\bZ_2] \to \Sigma \sVec$ corresponds to choosing the nontrivial monoidality isomorphism
  $$ I = I \boxtimes I = F(I) = F(M) \boxtimes F(M) \overset e \to F(M \boxtimes M) = I.$$
 In other words, $F$ maps the self-duality datum $M^\vee \cong M$, which came from the choice of isomorphism $r : M \boxtimes M \to I$ needed to select 
  $2\Vec^\tau[\bZ_2] \subset \cT$,
 to the nontrivial self-duality datum $e : I^\vee \cong I$.
\end{remark}

\subsection{\texorpdfstring{$\eta$}{eta}-traces and dimensions}\label{subsec:eta}

Our next goal is to present an invariant that can distinguish $\cS$ and $\cT$, and which is more computable than the cohomological distinction described in Theorem~\ref{thm:SandT}. We will do this in \S\ref{subsec:inv} by ``evaluating objects on Klein bottles with nonbounding Pin structure.'' In order to give that definition, we will need to be able to ``evaluate objects on circles with nonbounding Spin structure.'' This section reviews how to do this, and sets some notation for duality that we will use in the remainder of the proof.

\begin{definition}\label{defn:eta1cat}
Let $\cE$ be a braided monoidal 1-category. Each dualizable object $x \in \cE$ has an \define{$\eta$-dimension} defined as:
\begin{equation}\label{eq:eta1cat} \eta(x) := \ev_x \cdot \beta_{x^*, x} \cdot \coev_x =    
 \begin{tz}[std, scale=0.8]
   \draw[string, arrow data ={0.99}{<}] (0,0) to [out=up, in=down] (1,2); 
  \draw[braid, arrow data={0.7}{>}] (0,0) to [out=down, in=down, looseness=2](1,0) to  [out=up, in=down]  (0,2)  to [out=up, in=up, looseness=2]  (1,2);
 \end{tz} 
 \in \Omega\cE
 \end{equation}
 This element $\eta(x) \in \Omega \cE$ is independent of the choice of right duality data $\ev_x: x\otimes x^* \to I, \coev_x: I \to x^* \otimes x$.
 Indeed, $\eta(x)$ depends only on the isomorphism class of $x$. Furthermore, in any braided monoidal $1$-category $\cE$ it holds that  \begin{equation} \label{eq:etax*} \eta(x) = \eta({}^* x).\end{equation}

More generally, for a 1-morphism $\phi : x \to x$, we will define the \define{$\eta$-trace} of $\phi$ to be:
\begin{equation}\label{eqn:etatr}
  \tr(\phi) := \ev_x \cdot (\phi \otimes \id_x) \cdot \beta_{x^*, x} \cdot \coev_x = 
    \begin{tz}[std, scale=0.8]
    \draw[string, arrow data ={0.99}{<}] (0,0) to [out=up, in=down] (1,2); 
  \draw[braid, arrow data={0.31}{>}] (0,0) to [out=down, in=down, looseness=2](1,0) to  [out=up, in=down]   node[pos=0.91, dot] (B){} (0,2)  to [out=up, in=up, looseness=2]  (1,2);
    \node[label,left] at (B) {$\phi~$};
 \end{tz}
 \in \Omega\cE
\end{equation}
\end{definition}

\begin{remark} Suppose that $\cE$ is a braided monoidal groupal $1$-groupoid, corresponding to a connected and simply connected pointed $3$-type $X=\mathrm{B}^2\cE$, and the object $x\in \pi_2(X) = \pi_0 \cE$ is selected by a map $S^2\to X$. Then the object $\eta(x) \in \pi_3(X) = \pi_1\cE$ is selected by the composite $S^3\to[\eta] S^2 \to X$, where $\eta\in \pi_3(S^2)$ is the traditional notation for the Hopf map, the generator of $\pi_3(S^2) \cong \bZ$, hence the name. 

When $\cE$ is a symmetric monoidal 1-category, $\eta(x)$ is nothing but the value of the 1-dimensional framed TFT built from the dualizable object $x$ when evaluated on a circle with Lie group framing. The equality \eqref{eq:etax*} comes from the orientation-reversing diffeomorphism of the circle. 
\end{remark}

\begin{remark}\label{rem:nonmonoidal}
When $\cE$ is not symmetric, $\eta(-)$ is not monoidal. Rather, its failure to be monoidal is measured by the full braid. In particular, if $x$ and $y$ are invertible, then
$$ \eta(x \otimes y) = \langle \beta_{y,x} \beta_{x,y} \rangle \, \eta(x) \, \eta(y)  $$
(As in Remark~\ref{remark.2catDirectSums}, the notation means that $\beta_{y,x} \beta_{x,y} = \langle \beta_{y,x} \beta_{x,y} \rangle \id_{x\otimes y}$ where $\langle \beta_{y, x} \beta_{x,y}\rangle \in \Omega \cE$.)
\end{remark}

When $\cC$ is a braided monoidal 2-category, one must be slightly more careful with the choice of duality data when defining $\eta(-)$. Let us review some definitions about duality in higher categories:

\begin{definition}\label{defn:duals}
  A \define{right adjoint} of a $1$-morphism $f : X\to Y$ in a $2$-category is a $1$-morphism $f^*: Y \to X$ together with $2$-morphisms $\ev_f: f \circ f^* \to \id_b$ and $\coev_f: \id_a \to f^* \circ f$ fulfilling the cusp equations 
\[\id_f = (\ev_f \circ \id_{f}) \cdot (\id_f \circ \coev_f), \qquad \id_{f^*} = (\id_{f^*} \circ \ev_f) \cdot (\coev_f \circ \id_{f^*}).\]
There is a similar notion of left adjoint $^*f$. It is worth emphasizing that, if it exists, the right adjoint to $f$ is unique up to unique isomorphism.
An \define{adjoint equivalence} is a $1$-morphism equipped with a right adjoint for which $\ev_f$ and $\coev_f$ are invertible.

A \emph{right dual} of an object $X$ in a monoidal $2$-category is an object $X^\vee$ together with $1$-morphisms $\ev_X: X\boxtimes X^\vee \to I$ and $\coev_X: I \to X^\vee \boxtimes X$ and a cusp $2$-isomorphism 
\[\cusp_X: \id_X \Isom (\ev_X \boxtimes \id_{X}) \circ (\id_X \boxtimes \coev_X)
\]
such that also the composite $(\id_{X^\vee} \boxtimes \ev_X) \circ (\coev_X \boxtimes \id_{X^\vee})$ is isomorphic to $\id_{X^\vee}$. (To emphasize: we require only the existence of this latter isomorphism and not its data. There is a different equivalent notion in which one also demands the data of an isomorphism $\cusp^X : (\id_{X^\vee} \boxtimes \ev_X) \circ (\coev_X \boxtimes \id_{X^\vee}) \Isom \id_{X^\vee}$ but then asks for a compatibility relation between $\cusp_X$ and $\cusp^X$.) There is a similar notion of left dual. An object $X$ is \define{fully-} (aka \define{2-}) \define{dualizable} when it admits both left and right dual $X^\vee$ and all four evaluation and coevaluation 1-morphisms admit both left and right adjoints (in which case their adjoints end up also being adjunctible).

A \define{(right) mate} of a $1$-morphism $f:X\to Y$ between objects equipped with right duality data $(X^\vee, \ev_X, \coev_X, \cusp_X)$ and $(Y^\vee, \ev_Y, \coev_Y, \cusp_Y)$ is a $1$-morphism $f^\vee: Y^\vee \to X^\vee$ equipped with $2$-isomorphisms 
 \begin{gather*}
  \rot_f :  (\id_{X^\vee} \otimes f) \circ \coev_X \Isom (f^\vee \otimes \id_Y) \circ \coev_Y, \\
 \rot^f : \ev_Y\circ (f \otimes \id_{Y^\vee}) \Isom  \ev_X \circ (\id_X \otimes f^\vee),  \end{gather*}
which are compatible with $\cusp_X$ and $\cusp_Y$ in the sense that the following two isomorphisms $f \cong (\ev_X \boxtimes \id_Y)\circ f^\vee \circ (\id_X \boxtimes \coev_Y)$ are equal (suppressing coherence data for the monoidal 2-category):
\begin{gather*}
  f \cong f\circ\id_X \overset{\cusp_X}\Longrightarrow (\ev_X \boxtimes f) \circ (\id_X \boxtimes \coev_X) \overset{\rot_f}\Longrightarrow (\ev_X \boxtimes \id_Y)\circ f^\vee \circ (\id_X \boxtimes \coev_Y), \\
  f \cong \id_Y \circ f \overset{\cusp_Y}\Longrightarrow (\ev_Y \boxtimes \id_Y) \circ (f \boxtimes \coev_Y) \overset{\rot^f}\Longrightarrow (\ev_X \boxtimes \id_Y)\circ f^\vee \circ (\id_X \boxtimes \coev_Y).
\end{gather*}
There is a similar notion of left mate with respect to left duality data.
It is worth emphasizing that for a 1-morphism $f : X \to Y$ and for any choice of duality data, the mate of $f$ is unique up to unique isomorphism.
\end{definition}

\begin{definition}\label{defn:eta}
Let $X \in \cC$ and choose right duality data $(X^\vee, \ev_X, \coev_X, \cusp_X)$. The \define{$\eta$-dimension} of this data is the composition:
\[\eta(X,X^\vee,\ev_X, \coev_X, \cusp_X) := \ev_X \circ \br_{X^\vee, X} \circ \coev_X =    
 \begin{tz}[std, scale=0.8]
   \draw[string, arrow data ={0.99}{<}] (0,0) to [out=up, in=down] (1,2); 
  \draw[braid, arrow data={0.7}{>}] (0,0) to [out=down, in=down, looseness=2](1,0) to  [out=up, in=down]  (0,2)  to [out=up, in=up, looseness=2]  (1,2);
 \end{tz} \]
\end{definition}

It is easy to check that the isomorphism class of $\eta(X)$ depends only on the equivalence class of $X$. We will need, however, to promote $\eta(-)$ to a functor. To do this, we introduce a 2-category $\cC^\coh$ of \define{coherent dual pairs} in $\cC$ (compare \cite{1411.6691}):
\begin{itemize}
  \item An object of $\cC^\coh$ is an object $X \in \cC$ together with a choice of right duality data $(X^\vee, \ev_X, \coev_X, \cusp_X)$.
  \item A 1-morphism $(X,X^\vee, \dots) \to (Y,Y^\vee, \dots)$ is a 1-morphism $f : X \to Y$ together with a choice of right adjoint $(f^*, \ev_f, \coev_f)$, and also choices of their mates $(f^\vee,\rot_f,\rot^f)$ and $(f^{*\vee},\rot_{f^*},\rot^{f^*})$.
  \item A 2-morphism $(f,f^\vee,f^*,\dots) \To (g,g^\vee,g^*,\dots)$ is a 2-morphism $\phi : f \To g$. This choice selects 2-morphisms $\phi^* : g^* \to f^*$, $\phi^\vee : f^\vee \Rightarrow g^\vee$, and $\phi^{*\vee} : g^{*\vee} \to f^{*\vee}$ canonically.
\end{itemize}
Provided all objects and 1-morphisms in $\cC$ have duals, the forgetful functor $\cC^\coh \rightarrow \cC$, $(X,X^\vee,\dots) \mapsto X$ is an equivalence. Indeed, the definitions are arranged so that a 1-isomorphism in $\cC^\coh$ over $\id_X$ is precisely an \define{equivalence of duality data}: an adjoint equivalence $f: X^\vee \isom X'^\vee$ together with isomorphisms $\ev_X \cong \ev_X' \circ (\id_X \boxtimes f)$ and $(f \boxtimes \id_X) \circ \coev_X \cong \coev_X'$ satisfying the natural compatibility equation with $\cusp_X$ and $\cusp'_X$. Moreover, for a fixed object $X$, the 2-groupoid of right duals of $X$, equivalences of duality data between these, and compatible isomorphisms between the adjoint equivalences between these, is contractible. This contractibility is the reason that we need to give the isomorphism $\cusp_X$ as part of the data, but not a separate choice of isomorphism $\cusp^X$.

\begin{definition}\label{defn:hocat}
  The \define{homotopy 1-category} $\rh_1\cC$ of a (braided monoidal) 2-category $\cC$ is the (braided monoidal) 1-category whose objects are the objects of $\cC$ and whose morphisms are the isomorphism classes of $\cC$. We will write $[f] : X \to Y$ for the 1-morphism in $\rh_1\cC$ represented by a 1-morphism $f : X \to Y \in \cC$. (We will write the objects, and their tensor products, just as in $\cC$.)
\end{definition}

Explicitly, an object of $\rh_1\cC^\coh$ is an object $X \in \cC$ together with a choice of right duality data $(X^\vee,\ev_X,\coev_X,\cusp_X)$. A 1-morphism $(X,X^\vee,\dots) \to (Y,Y^\vee,\dots)$ in $\rh_1 \cC^\coh$ is simply an isomorphism class $[f]$ of 1-morphisms $f : X \to Y$ in $\cC$: the choices of mates and adjoints are uniquely determined up to unique isomorphism. This means in particular that to choose a splitting of the equivalence of $1$-categories $\rh_1\cC^\coh \to \rh_1 \cC$, one needs only to choose, for every object $X \in \cC$, right duality data $(X^\vee,\ev_X,\coev_X,\cusp_X)$.

 \begin{lemma}\label{lemma:etafunctor}
  The assignment $\eta(-)$ extends to a (nonmonoidal) functor $\mathrm{h}_1\cC^{\mathrm{coh}} \to \Omega \cC$, and hence, after choosing duality data for every object of $\cC$, into a functor $\mathrm{h}_1\cC \to \Omega \cC$. \end{lemma}
  
  \begin{proof}
  Let $f : X \to Y$ be a representative of a 1-morphism $[f] : (X,X^\vee,\dots) \to (Y,Y^\vee,\dots)$ in $\rh_1 \cC^\coh$. Choose a right adjoint $(f^*,\ev_f,\coev_v)$ and a mate $(f^\vee, \rot_f,\rot^f)$, and define the following $2$-morphism $\eta(f) : \eta(X) \Rightarrow \eta(Y)$: 
 \begin{equation}\label{eq:etaf}
 \def\l{-0.5}
 \def\r{1.5}
 \def\u{2.65}
 \def\d{-0.65}
 \def\scl{0.98}
    \begin{tz}[std, scale=\scl]
    \clip (\l,\d) rectangle (\r, \u);
      \draw[string%, arrow data ={0.99}{<}
      ] (0,0) to [out=up, in=down] (1,2); 
  \draw[braid=main%, arrow data={0.7}{>}
  ] (0,0) to [out=down, in=down, looseness=2](1,0) to  [out=up, in=down]  (0,2)  to [out=up, in=up, looseness=2]  (1,2);
 \end{tz}
 \overset{\coev_f}{ \Rightarrow} 
  \begin{tz}[std, scale=-\scl]
    \clip (\l,\d) rectangle (\r, \u);
    \draw[string] (0,0) to [out=up, in=down] (1,2); 
  \draw[braid] (0,0) to [out=down, in=down, looseness=2](1,0) to  [out=up, in=down]  node[dot, pos=1.09](A) {} node[pos=0.91, dot] (B){} (0,2)  to [out=up, in=up, looseness=2]  (1,2);
  \node[label, right] at (A) {$f\hphantom{^*}$};
    \node[label,right] at (B) {$f^*$};
 \end{tz}
  \overset{\rot_f}{\cong} 
   \begin{tz}[std,scale=-\scl]
    \clip (\l,\d) rectangle (\r, \u);
    \draw[string] (0,0) to [out=up, in=down] node[pos=1.05, dot] (B) {}(1,2); 
 \draw[braid] (0,0) to [out=down, in=down, looseness=2] (1,0) to  [out=up, in=down]  node[dot, pos=0.9](A) {} (0,2)  to [out=up, in=up, looseness=2]  (1,2);
  \node[label, right] at (A) {$f^*~$};
    \node[label, left] at (B) {$f^\vee$};
    \node[dot] at (B) {}; %%% this repeated dot fixes something funny about the braiding
 \end{tz}
 \cong 
  \begin{tz}[std,scale=\scl]
    \clip (\l,\d) rectangle (\r, \u);
    \draw[string] (0,0) to [out=up, in=down] node[pos=0, dot] (B) {}(1,2); 
 \draw[braid] (0,0) to [out=down, in=down, looseness=2] (1,0) to  [out=up, in=down]  node[dot, pos=1.05](A) {} (0,2)  to [out=up, in=up, looseness=2]  (1,2);
  \node[label, left] at (A) {$f^*$};
    \node[label, left] at (B) {$f^\vee$};
    \node[dot] at (B) {}; %%% this repeated dot fixes something funny about the braiding
 \end{tz}
  \cong 
  \begin{tz}[std,scale=\scl]
    \clip (\l,\d) rectangle (\r, \u);
    \draw[string] (0,0) to [out=up, in=down] node[pos=0.85, dot] (B) {}(1,2); 
 \draw[braid] (0,0) to [out=down, in=down, looseness=2] (1,0) to  [out=up, in=down]  node[dot, pos=1.05](A) {} (0,2)  to [out=up, in=up, looseness=2]  (1,2);
  \node[label, left] at (A) {$f^*$};
    \node[label, right] at (B) {$f^\vee$};
    \node[dot] at (B) {}; %%% this repeated dot fixes something funny about the braiding
 \end{tz}
  \cong  
  \begin{tz}[std,scale=\scl]
    \clip (\l,\d) rectangle (\r, \u);
    \draw[string] (0,0) to [out=up, in=down] node[pos=1.05, dot] (B) {}(1,2); 
 \draw[braid] (0,0) to [out=down, in=down, looseness=2] (1,0) to  [out=up, in=down]  node[dot, pos=0.9](A) {} (0,2)  to [out=up, in=up, looseness=2]  (1,2);
  \node[label, left] at (A) {$f^*$};
    \node[label, right] at (B) {$f^\vee$};
    \node[dot] at (B) {}; %%% this repeated dot fixes something funny about the braiding
 \end{tz}
 \overset{\rot^f}{ \cong} 
  \begin{tz}[std, scale=\scl]
    \clip (\l,\d) rectangle (\r, \u);
    \draw[string] (0,0) to [out=up, in=down] (1,2); 
  \draw[braid] (0,0) to [out=down, in=down, looseness=2](1,0) to  [out=up, in=down]  node[dot, pos=1.09](A) {} node[pos=0.91, dot] (B){} (0,2)  to [out=up, in=up, looseness=2]  (1,2);
  \node[label, left] at (A) {$f\hphantom{^*}$};
    \node[label,left] at (B) {$f^*$};
 \end{tz}
 \overset{\ev_f}{\Rightarrow} 
   \begin{tz}[std, scale=\scl]
    \clip (\l,\d) rectangle (\r, \u);
     \draw[string] (0,0) to [out=up, in=down] (1,2); 
  \draw[braid] (0,0) to [out=down, in=down, looseness=2](1,0) to  [out=up, in=down]  (0,2)  to [out=up, in=up, looseness=2]  (1,2);
 \end{tz}
 \end{equation}
 (We have omitted orientation marks for better readability.)
  It is easy to verify that $\eta(f)$ does not depend on the choice of adjoint and mate for $f$ and in fact only depends on the isomorphism class of $f$ and that $\eta(-)$ assembles into a $1$-functor. 
\end{proof}

 Note that, in the definition of $\eta(f)$ in equation~\eqref{eq:etaf}, the figure-eight is fixed, and simply serves as a track for the defects points to move along. Later we will move the figure-eight itself; see e.g.\ equation~\eqref{eqn:fl1} and Lemma~\ref{lemma:fl}.

\begin{remark}
For morphisms $a:A\to A'$ and $b:B\to B'$ in a braided monoidal $2$-category, denote by \[\widetilde{\beta}_{a,b}: \br_{A', B'} \circ (a\boxtimes b) \To (b \boxtimes a) \circ \br_{A, B} \]
the naturality isomorphism of the braiding (i.e.\ the composite of the unlabelled isomorphisms in~\eqref{eq:etaf}). Writing \begin{equation}\label{eq:tildecoev}
\widetilde{\coev}_f: \coev_X \To (f^\vee \boxtimes f^*) \circ \coev_Y \hspace{1cm} \widetilde{\ev}_f: \ev_X \circ (f^* \boxtimes f^\vee) \To \ev_Y
\end{equation}
 for the obvious composites of $\coev_f$ followed by $\rot_f$, and $\rot^f$ followed by $\ev_f$, respectively, we may express~\eqref{eq:etaf} more succinctly as the composite 
 \begin{equation}
 \label{eq:etafeasier} \eta(f)= \left(\widetilde{\ev}_f \circ \br_{Y^\vee, Y} \circ \coev_Y \right) \cdot  \left( \ev_X \circ \widetilde{\beta}_{ f^\vee, f^*} \circ \coev_Y \right) \cdot \left( {\ev_X} \circ {\br_{X^\vee, X}} \circ \widetilde{\coev}_{f}\right)
 \end{equation}
  \begin{equation*}
    \begin{tz}[std, scale=1]
      \draw[string%, arrow data ={0.99}{<}
      ] (0,0) to [out=up, in=down] (1,2); 
  \draw[braid=main%, arrow data={0.7}{>}
  ] (0,0) to [out=down, in=down, looseness=2](1,0) to  [out=up, in=down]  (0,2)  to [out=up, in=up, looseness=2]  (1,2);
 \end{tz}
 ~\overset{\widetilde{\coev}_f}{\Rightarrow}~
   \begin{tz}[std,scale=-1]
    \draw[string] (0,0) to [out=up, in=down] node[pos=1.0, dot] (B) {}(1,2); 
 \draw[braid] (0,0) to [out=down, in=down, looseness=2] (1,0) to  [out=up, in=down]  node[dot, pos=1](A) {} (0,2)  to [out=up, in=up, looseness=2]  (1,2);
  \node[label, right] at (A) {$f^*~$};
    \node[label, left] at (B) {$f^\vee$};
    \node[dot] at (B) {}; %%% this repeated dot fixes something funny about the braiding
 \end{tz}
 ~\overset{\widetilde{\beta}_{f^{\vee}, f^*}}{\cong}~
  \begin{tz}[std]
    \draw[string] (0,0) to [out=up, in=down] node[pos=1, dot] (B) {}(1,2); 
 \draw[braid] (0,0) to [out=down, in=down, looseness=2] (1,0) to  [out=up, in=down]  node[dot, pos=1.0](A) {} (0,2)  to [out=up, in=up, looseness=2]  (1,2);
  \node[label, left] at (A) {$f^*~$};
    \node[label, right] at (B) {$f^\vee$};
    \node[dot] at (B) {}; %%% this repeated dot fixes something funny about the braiding
 \end{tz}
 ~\overset{\widetilde{\ev}_f}{\Rightarrow}~
   \begin{tz}[std, yscale=1]
     \draw[string] (0,0) to [out=up, in=down] (1,2); 
  \draw[braid] (0,0) to [out=down, in=down, looseness=2](1,0) to  [out=up, in=down]  (0,2)  to [out=up, in=up, looseness=2]  (1,2);
 \end{tz}
 \end{equation*}
 resulting in a more symmetric formula for $\eta(f)$ closely resembling~\eqref{eq:eta1cat}. 

In particular, given a $1$-morphism $x:I \to I$ in a braided monoidal $2$-category $\cC$, and choosing the trivial duality datum $(I^\vee =I, \ldots)$ for the tensor unit, expression~\eqref{eq:etafeasier} immediately implies that $\eta(x)$ agrees with $\eta(x)$ from Definition~\ref{defn:eta1cat} for the symmetric monoidal $1$-category $\Omega \cC$.
\end{remark}

\begin{remark}\label{rem:otherexpressions}
There are a number of equivalent expressions for  $\eta(f)$. For example, in~\eqref{eq:etaf}, the morphism $f$ ``moves around the circle'' while $f^*$ only moves up. Instead, one may work with $f^*$ and $(f^*)^\vee$ and analogously to~\eqref{eq:tildecoev} define the following composites of $\coev_f$ followed by $\rot^{f^*}$, and $\rot_{f^*}$ followed by $\ev_f$, respectively: 
\[
\widetilde{\widetilde{\coev}}_f: \ev_X \To \ev_Y\circ (f \boxtimes (f^*)^\vee) \hspace{1cm}  \widetilde{\widetilde{\ev}}_f: \coev_X \circ ((f^*)^\vee \boxtimes f) \To \coev_Y.
\]
Observe that $(f^*)^\vee$ is a \emph{left adjoint} of $f^\vee$ (in formulas: $(f^*)^\vee \cong {}^*(f^\vee)$ --- notice the change in variance coming from contravariance of $(-)^\vee$ in $1$-morphism direction), and hence that $( f \boxtimes (f^*)^\vee ) \dashv (f^* \boxtimes f^\vee)$ and $((f^*)^\vee \boxtimes f ) \dashv (f^\vee \boxtimes f^*)$. With respect to these adjunctions, $\widetilde{\coev}_f, \widetilde{\ev}_f$, and $\widetilde{\beta}_{f^\vee, f^*}$ are ``$*$-mates'' (i.e.\ mates of 2-morphisms with respect to adjunction of 1-morphisms)
of $\widetilde{\widetilde{\ev}}_f$, $\widetilde{\widetilde{\coev}}_f$ and $\widetilde{\beta}^{-1}_{(f^*)^\vee, f}$, respectively. (More generally, for any pair of $1$-morphisms $a:A\to A'$ and $b:B\to B'$ with right adjoints $a^*$ and $b^*$, the $2$-morphism $\widetilde{\beta}_{a^*,b^*} : \br_{A', B'} \circ (a^* \boxtimes b^*) \To (b^*\boxtimes a^*) \circ \br_{}$ is a $*$-mate of $\widetilde{\beta}^{-1}_{a,b}$ with respect to the induced adjunctions $(a\boxtimes b) \dashv (a^* \boxtimes b^*)$ and $(b \boxtimes a)\dashv (b^* \boxtimes a^*)$.) As the composite of mates is the mate of composites, it follows that we may write: 
\begin{equation}
\label{eq:otherexpression}\eta(f) = \left(\ev_Y \circ \br_{Y^\vee, Y} \circ \widetilde{\widetilde{\ev}}_f \right) \cdot \left( \ev_Y \circ \widetilde{\beta}^{-1}_{(f^*)^\vee, f} \circ \coev_X\right) \cdot  \left( \widetilde{\widetilde{\coev}}_f \circ \br_{X^\vee, X} \circ \coev_X\right).
\end{equation}
\end{remark}

\begin{remark}\label{rem:nonmonoidality2}
  The nonmonoidality in Remark~\ref{rem:nonmonoidal} categorifies in various ways to nonmonoidality of the functor $\eta : \mathrm{h}_1\cC \to \Omega \cC$. Suppose as an example that $x \in \Omega\cC$ and $Y \in \cC$ are both invertible. Then the full braid $\br_{Y,x}\br_{x,Y} : x \boxtimes \id_Y \Rightarrow x \boxtimes \id_Y$ is some ``scalar'' $\langle \br_{Y,x}\br_{x,Y} \rangle \in \Omega^2 \cC$ times $\id_{x \boxtimes \id_Y}$, and we have
  $$ \eta(x \boxtimes \id_Y : Y \to Y) = \langle  \br_{Y,x}\br_{x,Y} \rangle \, \eta(x) \, \id_{\eta(Y)}.$$
\end{remark}

\begin{example}\label{ex:computationofeta}
  The isomorphism classes of the $\eta$-dimensions of the simple objects in $\cS$ and $\cT$ are given by the following objects in $\Omega\cS = \Omega\cT = \sVec$:
  \begin{equation}\label{eqn:etaM}
   \eta(I) \cong I, \qquad \eta(C) \cong e, \qquad \eta(M) \cong I, \qquad \eta(C \boxtimes M) \cong I.\end{equation}
  The first isomorphism holds in any braided monoidal 2-category. The second isomorphism follows from~\eqref{eqn.halfbraidc}, and the third follows from~\eqref{eqn:braidingM}. 
  
  To prove the last isomorphism, note that Remark~\ref{rem:nonmonoidality2} implies that for $e\boxtimes \id_C: C \to C$, $e\boxtimes \id_M :M \to M$, and $e\boxtimes \id_{C\boxtimes M}: C \boxtimes M \to C\boxtimes M$, we have that 
\begin{equation}
\label{eq:etae}
\eta(e \boxtimes \id_C) = (-1) \id_{\eta(C)}, \qquad \eta(e \boxtimes \id_M) = \id_{\eta(M)}, \qquad \eta(e \boxtimes \id_{C\boxtimes M}) = \id_{\eta(C\boxtimes M)}.
\end{equation} 
Now recall from Example~\ref{eg:sVeclin} the (abusive) notation $v:I \to C$ and $v:C \to I$ for the non-invertible simple $1$-morphisms in $\Sigma \sVec$. They satisfy the fusion rules $v^2 \cong \id \oplus (e \boxtimes \id)$ (with $\id$ being either $\id_C$ or $\id_I$ depending on the order of composition). By tensoring with $\id_M$, we find noninvertible 1-morphisms $v \boxtimes \id_M : M \leftrightarrow C \boxtimes M$ for which $(v \boxtimes \id_M)^2 \cong (I \oplus e) \boxtimes \id$. Thus, by functoriality,
\[\eta(v \boxtimes \id_M)^2  = \eta((v\boxtimes \id_M)^2) = \eta((I \oplus e)\boxtimes \id_X) = 2\cdot \id_{\eta(X)}.
\]
where $X$ is either $M$ or $C\boxtimes M$ depending on the order of composition. In particular, $\eta(v \boxtimes \id_M) : \eta(M) \to \eta(C \boxtimes M)$ is an isomorphism, giving the last isomorphism in \eqref{eqn:etaM}.
\end{example}

We will also use a coherent 2-categorical version of the well-known fact that, in a braided monoidal $1$-category, it holds that $\eta(^*x) =\eta(x)$. This equation has the following string-diagrammatic proof:
\begin{equation}\label{eqn:fl1}
 \begin{tz}[std, scale=0.8]
   \draw[string] (0,0) to [out=up, in=down] (1,2); 
  \draw[braid] (0,0) to [out=down, in=down, looseness=2](1,0) to  [out=up, in=down]  (0,2)  to [out=up, in=up, looseness=2]  (1,2);
  \path[arrow data={0}{>}] (1,2) to (1,3);
    \path[arrow data={0}{<}] (0,2) to (0,3);
  \node[label, left] at (0,2) {${}^* x~$};
    \node[label, right] at (1,2) {$~x$};
 \end{tz} 
~~ = ~~
  \begin{tz}[std, scale=0.8]
   \draw[string] (0,0) to [out=up, in=down] (1,2) to [out=up, in=up, looseness=2] (2,2) to [out=down, in=down, looseness=2] (3,2) to [out=up, in=up, looseness=2] (0,2); 
  \draw[braid, arrow data={0.67}{>}] (0,2) to [out=down, in=up] (1,0) to [out=down, in=down, looseness=2] (-2,0) to[out=up, in=up, looseness=2] (-1,0) to [out=down, in=down, looseness=2] (0,0);% to [out=down, in=down, looseness=2](1,0) to  [out=up, in=down]  (0,2);%  to [out=up, in=up, looseness=2]  (1,2);
      \path[arrow data={0}{<}] (0,2) to (0,3);
      \path[arrow data={0}{<}] (2,2) to (2,3);
%  \node[label,below  left] at (2,2) {$X^\vee\!\!$};
   % \node[label, left] at (0,2) {${}^\vee X~$};
 \end{tz} 
~~= ~~
   \begin{tz}[std, scale=0.8]
   \draw[string] (0,0) to [out=up, in=down] (1,2) to [out=up, in=up, looseness=2] (2,2) to (2,-1) to  [out=down, in=down, looseness=2] (3,-1) to (3,2) to [out=up, in=up, looseness=2] (0,2); 
  \draw[braid] (0,2) to [out=down, in=up] (1,0) to [out=down, in=down, looseness=2] (-2,0) to (-2,3) to [out=up, in=up, looseness=2] (-1,3) to (-1,0) to [out=down, in=down, looseness=2] (0,0);% to [out=down, in=down, looseness=2](1,0) to  [out=up, in=down]  (0,2);%  to [out=up, in=up, looseness=2]  (1,2);
\path[arrow data={0}{>}] (-2,3) to (-2,4); 
   \path[arrow data={0}{<}] (0,2) to (0,3);
   \path[arrow data ={0}{>}] (2,-1) to (2,-2);
   \end{tz} 
   ~~= ~~
    \begin{tz}[std, scale=0.8]
   \draw[string] (0,0) to [out=up, in=down] (1,2); 
  \draw[braid] (0,0) to [out=down, in=down, looseness=2](1,0) to  [out=up, in=down]  (0,2)  to [out=up, in=up, looseness=2]  (1,2);
  \path[arrow data={0}{<}] (1,2) to (1,3);
    \path[arrow data={0}{>}] (0,2) to (0,3);
  \node[label, left] at (0,2) {$x~$};
      \node[label, right] at (1,2) {$~x^*$};
 \end{tz} 
\end{equation}

Consider a 2-category $^\coh \cC ^\coh$ whose objects are an object $X \in \cC$ together with both a choice of right dual $(X^\vee, \dots)$ and a separate choice of left dual $(^\vee X, \dots)$, and whose 1-morphisms similarly come with choices of both right and left mates and adjoints. 
As is the case with $\rh_1 \cC^\coh$, to give a 1-morphism $(X,{^\vee X},X^\vee, \dots) \to (Y,{^\vee Y},Y^\vee, \dots)$  in the homotopy category  is simply to give an isomorphism class $[f]$ of 1-morphisms $f : X \to Y$. Indeed, $\rh_1(^\coh \cC ^\coh)$ arises as a (strict!)\ fibre product
$$ \rh_1(^\coh \cC ^\coh) = \rh_1\cC^\coh \underset{\rh_1\cC}\times \rh_1\cC^\coh $$
where one of the maps $\rh_1\cC^\coh \to \rh_1\cC$ is the forgetful functor $(X,X^\vee,\dots) \mapsto X$, and the other is
$$ (Y,Y^\vee,\dots) \mapsto Y^\vee, \qquad [g] \mapsto [g^{*\vee}].$$
Note that both $(-)^*$ and $(-)^\vee$ are contravariant on 1-morphisms, so that $[g] \mapsto [g^{*\vee}]$ is covariant.
A splitting of this functor is a choice of left dual $(^\vee X,\dots)$ for each object $X$.

In particular, the two canonical functors $h_1{}^\coh \cC ^\coh \to h_1\cC^{\coh}$ are 
\begin{align*}
 (X, {}^\vee X, X^\vee, \ldots) \mapsto (X, X^\vee, \ldots),& \qquad [f] \mapsto [f] ,
\\
\intertext{and}
(X, {}^\vee X, X^\vee, \ldots) \mapsto ({}^\vee X, X, \ldots),& \qquad [f] \mapsto [{}^* {}^\vee f]. \end{align*}

 \begin{lemma} \label{lemma:fl}
 The proof \eqref{eqn:fl1} categorifies to a natural isomorphism 
 $$\fl_{(X,{^\vee X},X^\vee,\dots)}: \eta({}^\vee X ,X,\dots) \Isom \eta(X,X^\vee,\dots)$$
 of functors $\rh_1 (^\coh\cC^\coh) \rightrightarrows \rh_1\cC^{\coh} \to \Omega\cC$.
 After choosing for each object in $\cC$ both a left and right dual, we will write this as a natural isomorphism
 $$\fl_X : \eta(^{\vee}X) \Isom \eta(X).$$
 \end{lemma}
  
  \begin{proof}
The choices of left and right duality data for $X$ needed to lift $X$ and $^\vee X$ to $\cC^\coh$ also select right duality data for the objects ${^\vee X} \boxtimes X$ and $X\boxtimes {^\vee X}$: for example, after suppressing associators and unitors, we may choose $\ev_{{^\vee X} \boxtimes X} := \ev_{^\vee X} \circ (\id_{^\vee X} \boxtimes \ev_X \boxtimes \id_X)$ as the evaluation for a duality ${}^\vee X \boxtimes X \dashv X^\vee \boxtimes X$. With respect to these chosen duality data, $\coev_X$ is a right mate of $\ev_{{^\vee X}}$, $\ev_X$ is a right mate of $\coev_{{^\vee X}}$, and $\br_{X^\vee, X}$ is a right mate of $\br_{X,{^\vee X}}$.
    
On objects, we define $\fl_X$ to be the isomorphism
\begin{multline} \label{eqn:fl}
\eta({^\vee X}) \cong \bigl(\eta({^\vee X})\bigr)^\vee = 
 \bigl(\ev_{{^\vee X}} \circ \br_{X, {^\vee X}} \circ \coev_{{^\vee X}}\bigr)^\vee \\
 \cong (\coev_{{^\vee X}})^\vee \circ (\br_{X, {^\vee X}})^\vee \circ (\ev_{{^\vee X}})^\vee 
%\\
\cong \ev_X \circ \br_{X^\vee, X} \circ \coev_X = \eta(X)
\end{multline}
  The first isomorphism is the fact that $\eta({^\vee X})$, being an endomorphism of the unit object, is trivially a mate for itself. 
  The second isomorphism is the canonical isomorphism between mates of composites and composites of mates. The last isomorphism uses the matings listed in the previous paragraph.
  
  To prove naturality, let $f:X \to Y$ be a $1$-morphism, and choose a right adjoint $f^*$ and a right mate $f^\vee$ as needed to define $\eta(f)$. Further, choose a left mate ${}^\vee f$ and a left adjoint ${}^*{}^\vee f$. To compute $\eta({}^*{}^\vee f)$, we choose ${}^\vee f$ as a right adjoint of ${}^*{}^\vee f$ and we choose $f^*$ as a right mate of ${}^*{}^\vee f$ (in formulas: $({}^*{}^\vee f)^\vee \cong f^*$, cf.\ Remark~\ref{rem:otherexpressions}).
  With these choices, the expressions~\eqref{eq:tildecoev} unpack to
\[
\widetilde{\coev}_{{}^*{}^\vee f}: \coev_{{}^\vee X} \To (f^* \boxtimes {}^\vee f) \circ \coev_{{}^\vee Y} \hspace{1cm} \widetilde{\ev}_{{}^*{}^\vee f}: \ev_{{}^\vee X} \circ ({}^\vee f \otimes f^*) \To \ev_{{}^\vee Y}
\]
and we find the following commutative diagrams:
      \[\begin{tikzcd}
  \left( \coev_{{}^\vee X}\right)^\vee 
  \arrow[d, Rightarrow, "\cong"'] \arrow[rr,Rightarrow,  "\left(\widetilde{\coev}_{{}^*{}^\vee f}\right)^\vee"]&& \left( (f^*\boxtimes {}^\vee f) \circ \coev_{{}^\vee Y} \right)^\vee  \arrow[d, Rightarrow, "\cong"]\\
  \ev_X  \arrow[rr, Rightarrow, "\widetilde{\widetilde{\coev}}_{f}"]&& \ev_Y \circ \left( f \boxtimes (f^*)^\vee\right)
  \end{tikzcd}
  \hspace{1.cm}
  \begin{tikzcd} 
  \left(\ev_{{}^\vee X} \circ ({}^\vee f \otimes f^*)\right)^\vee 
  \arrow[d,Rightarrow,  "\cong"'] \arrow[rr,Rightarrow, "\left(\widetilde{\ev}_{{}^*{}^\vee f}\right)^\vee"]&& \left( \ev_{{}^\vee Y} \right)^\vee 
  \arrow[d, Rightarrow, "\cong"] \\
  ((f^*)^\vee \boxtimes f) \circ \coev_X  \arrow[rr, Rightarrow,"\widetilde{\widetilde{\ev}}_{f}"]&& \coev_Y
  \end{tikzcd}
  \]
      \[\begin{tikzcd}
  \left( \br_{Y, {}^\vee X} \circ (f^* \boxtimes {}^\vee f) \right)^\vee 
  \arrow[d, Rightarrow, "\cong"'] \arrow[rr,Rightarrow,  "\left(\widetilde{\beta}_{f^*, {}^\vee f}\right)^\vee"]&& \left( ({}^\vee f \boxtimes f^*) \circ \br_{X, {}^\vee Y}\right)^\vee   \arrow[d, Rightarrow, "\cong"]\\
  (f \boxtimes (f^*)^\vee) \circ \br_{Y^\vee, X}  \arrow[rr, Rightarrow, "\widetilde{\beta}^{-1}_{(f^*)^\vee, f}"]&& \br_{X^\vee, Y} \circ  \left( (f^*)^\vee \boxtimes f\right)
  \end{tikzcd}
  \]
  Here, the vertical morphisms are the canonical isomorphisms between mates. Hence, it follows from the equality between~\eqref{eq:tildecoev} and~\eqref{eq:otherexpression} and the definition~\eqref{eqn:fl} of $\fl_X$ in terms of canonical isomorphisms between mates that $\eta(f) \cdot \fl_X = \fl_{{}^\vee Y} \cdot \eta({}^* {}^\vee f)$.
    \end{proof}

\begin{remark}\label{rem:complicatedfl}
We will in fact use a slightly more flexible formula for $\fl$ than the one in~(\ref{eqn:fl}). Given $X \in \cC$, choose right and left duality data, needed to define $\eta({^\vee X})$ and $\eta(X)$. Furthermore, choose:
\begin{enumerate}
  \item arbitrary right duality data $(({^\vee X} \boxtimes X)^\vee, \ldots)$ and $((X\boxtimes {^\vee X})^\vee, \ldots)$ for ${^\vee X} \boxtimes X$ and $X\boxtimes {^\vee X} $. \label{choice:productarbitrary}
  \item equivalences of duality data $\omega_L: ({}^\vee X \boxtimes X)^\vee \to X^\vee \boxtimes X$ and $\omega_R: (X\boxtimes {}^\vee X)^\vee \to X \boxtimes X^\vee$ between the data chosen in \ref{choice:productarbitrary} and the canonical tensor product duality data used in the proof of Lemma~\ref{lemma:fl}. \label{choice:productarbitraryundo}
  \item mates for $\coev_{{}^\vee X}$, $\ev_{{}^\vee X}$, and $\br_{X, {^\vee X}}$ with respect to the chosen duality data \ref{choice:productarbitrary}. \label{choice:mates}
\end{enumerate}
Then we can just as well define $\fl$ as the isomorphism
\begin{multline} \label{eqn:flfull}
\eta({^\vee X}) 
 \cong (\coev_{{^\vee X}})^\vee \circ (\br_{X, {^\vee X}})^\vee \circ (\ev_{{^\vee X}})^\vee 
 \\ 
\cong (\coev_{{^\vee X}})^\vee \circ \omega_L^{-1} \circ \omega_L \circ  (\br_{X, {^\vee X}})^\vee \circ \omega_R^{-1} \circ \omega_R \circ (\ev_{{^\vee X}})^\vee
\\
\cong \ev_X \circ \br_{X^\vee, X} \circ \coev_X = \eta(X)
\end{multline}
It is then straightforward to check that $\fl$ as defined in (\ref{eqn:flfull}) is independent of the choices made in \ref{choice:productarbitrary}--\ref{choice:mates}. The reason we will want this flexibility is that we will have examples where the ``natural'' duals and mates are not the canonical ones given by tensor products.
\end{remark}

\subsection{The Klein invariant \texorpdfstring{$\Inv$}{kappa}}\label{subsec:inv}

With the functor $\eta(-)$ and the natural transformation $\fl: \eta({}^*{}^\vee(-)) \To \eta(-)$ in hand, we may now define our invariant $\Inv$. 

\begin{definition}\label{defn:klein}
Suppose that $X \in \cC$ is a fully-dualizable object equipped with a self-duality $r : X \isom {^\vee X}$. Then, $\fl_X \cdot \eta(r) : \eta(X) \to \eta({}^\vee X) \to \eta(X)$ is an automorphism.
The \define{Klein invariant} $\Inv(\cC, X, r)$ is its trace as defined in equation (\ref{eqn:etatr}):
$$ \kappa(\cC,X,r) := \tr\left(\vphantom{\frac{a}{b}}\fl_X \cdot \eta(r) : \eta(X) \to \eta(X)\right) \in \Omega^2 \cC $$
A priori, our Klein ``invariant'' looks like it depends on choices of duality data for $X, r, \dots$. But functoriality of $\eta$ and naturality of $\fl$ shows that it does not: if $Y $ is any other choice of left dual of $X$, $\fl_{(X, {}^\vee X, \ldots) }: \eta({}^\vee X) \to \eta(X)$ and $\fl_{(X, Y, \ldots )}: \eta(Y) \to \eta(X)$ are the respective coefficients of the flip natural isomorphism, and $h : {^\vee X} \to Y$ is an equivalence of duality data, then  $ \tr \bigl(\fl_{(X, {}^\vee X, \ldots)} \cdot \eta(r) : \eta(X) \to \eta(X)\bigr)= \tr \bigl(\fl_{(X,Y, \ldots)} \cdot \eta(g \circ r) : \eta(X) \to \eta(X)\bigr) $. 
We will occasionally write just $\kappa(X,r)$ when the value of $\cC$ is understood.
\end{definition}

The name ``Klein'' for our invariant $\Inv$ comes from its topological interpretation: the trace of the flip map --- the orientation-reversing diffeomorphism of the circle --- is topologically a Klein bottle.

\begin{remark}
  One could invent the invariant $\Inv$, and in particular discover the importance of the Klein bottle to our problem, via the following topological analysis. (Our actual route to discovery was more circuitous.)

  Our goal is to distinguish the two braided fusion 2-categories $\cS$ and $\cT$. After choosing magnetic objects $M$ and self-duality data, it would suffice to distinguish the symmetric monoidal sub-2-categories $2\Vec[\bZ_2] \subset \cS$ and $2\Vec^\tau[\bZ_2] \subset \cT$ generated by $\cM$, where $\tau = (-1)^{\Sq^2 \Sq^1 t_\bullet} \in \H^{\bullet + 3}(K(\bZ_2,\bullet); \bk^\times)$. We attempt to distinguish these symmetric monoidal $2$-categories $\cC$ by evaluating (various invertible $\cC$-valued $2$-dimensional TFTs induced by) the non-trivial invertible object $M$ on certain closed $2$-manifolds. 
  In fact, as $M$ is invertible, any such field theory factors through the Picard spectrum (i.e. the symmetric monoidal groupal sub-$2$-groupoid of invertible objects, $1$- and $2$-morphisms) of $\cC$, and hence it suffices to distinguish these Picard spectra. By definition, $\mathrm{Pic}(2\Vec[\bZ_2])$ is the product $\Sigma^2 H \bk^\times \times H\bZ_2$ of Eilenberg-Maclane spectra, while $\mathrm{Pic}(2\Vec^\tau[\bZ_2])$ is the non-trivial extension of $H\bZ_2$ by $\Sigma^2 H \bk^\times$ controlled by the Postnikov class $(-1)^{\Sq^2\Sq^1}$. Simplifying further by only considering scalars $\{ \pm 1\} \subseteq \bk^\times$, let $W$ be the spectrum with homotopy groups $\pi_0 W= \pi_2 W = \bZ_2$ connected by the Postnikov class $\Sq^2 \Sq^1$. Hence, our goal is to write down a field-theoretic invariant that sees that this Postnikov class is nontrivial. 
  
  Following the classification of invertible field theories~\cite{SchommerPriesInv}, let us try to do this by mapping  into $W$ from various (for simplicity stable) bordism spectra $\mathrm{MT}G$ and then evaluating the induced map $\pi_2 \mathrm{MT}G \to \pi_2 W$.

 As a first attempt, one could consider mapping into $W$ from the framed bordism aka sphere spectrum $\bS = \mathrm{MT*}$. A map $\bS \to W$ is classified by an object $M$ in $\pi_0 W = \bZ_2$. Unfortunately, whichever object one picks, the induced map on $\pi_2 \bS \to \pi_2 W$ will always be trivial as any map $\bS^2 \to \bS$ factors through $\bS^1$ and $\pi_1 W =0$. Therefore, the Postnikov class cannot be detected by evaluating the field theory induced by our object $M$ on framed $2$-manifolds.
   
  Let us look at the other extreme: the invertible (stable) unoriented bordism spectrum $\mathrm{MTO}$. Fix the non-trivial map $\pi_0 \mathrm{MTO} \isom \pi_0 W = \bZ_2$ (determined by picking the non-trivial object in $\pi_0 W$ as labelling our bordisms) and ask whether this lifts to a map $\mathrm{MTO} \to W$. By the definition of $W$, such lifts are equivalent to trivializations of $\Sq^2 \Sq^1u$, where $u$ is the nontrivial element in reduced cohomology
   $\widetilde{\H}{}^0(\mathrm{MTO};\bZ_2) = \bZ_2$. But $\mathrm{MTO} = \mathrm{MO}$ is a $\bZ_2$-oriented Thom spectrum, so reduced cohomology of $\mathrm{MTO}$ agrees with unreduced cohomology of $\mathrm{BO}$ and $\Sq^2\Sq^1$ sends the Thom class $u$ to the non-trivial class $w_1w_2 + w_3 \in \H^3(\mathrm{B}O, \bZ_2)$. We conclude that $\mathrm{MTO}$ simply does not map nontrivially to $W$.

Therefore, we need to use a spectrum $\mathrm{MT}G$ somewhere between $\bS$ and $\mathrm{MTO}$. In particular, we should ask 
  that the composition $\mathrm{MT}G \to \mathrm{MTO} \to[w_1 w_2 + w_3] \Sigma^3 H \bZ_2$ vanishes, while still asking  for the induced map $\pi_2\mathrm{MT}G \to \pi_2W$ to be nontrivial. A candidate for such a bordism spectrum is $\mathrm{MTPin_+}$. This follows since $w_1 w_2 + w_3 = \Sq^1 w_2$, and $\mathrm{Pin}_+$-manifolds are precisely those for which (tangential) $w_2$ vanishes. 
  
  In fact, the low-dimensional tangential $\mathrm{Pin}_+$ bordism groups are (see~\cite{MR1171915}) \[ \Omega^{\mathrm{Pin}_+}_0 = \bZ_2 \hspace{1cm} \Omega^{\mathrm{Pin}_+}_1= 0 \hspace{1cm} \Omega^{\mathrm{Pin}_+}_2= \bZ_2,\]  and since $\mathrm{MTPin_+} \cong \mathrm{MPin_-}$ is a Thom spectrum, its low-dimensional reduced cohomology groups $ \widetilde{\H}{}^\bullet(\mathrm{MTPin_+}, \bZ_2) \cong \H^{\bullet}(\mathrm{BPin_-}, \bZ_2)$ are
  \[ \H^0 = \bZ_2 \hspace{1cm} \H^1 = \bZ_2 \hspace{1cm} \H^2 = \bZ_2
\] These cohomology groups imply that the first Postnikov invariant of $\mathrm{MTPin}_+$ has to be non-trivial and hence is $\Sq^2 \Sq^1$. In other words, our spectrum $W$ is precisely equivalent to the 2-type of $\mathrm{MTPin}_+$. It is known that $\pi_2 \mathrm{MTPin}_+$ is generated by the Klein bottle with non-bounding $\mathrm{Pin}_+$ structure (see~\cite{MR1171915}). Thus, this Klein bottle is the \emph{universal invariant} to see our nontrivial extension.

This Klein bottle invariant can be extended from symmetric monoidal groupal $2$-groupoids to symmetric monoidal $2$-categories, by interpreting the above computation in terms of the Cobordism Hypothesis:  this is achieved by requiring the point to map to an (appropriately structured) dualizable object, rather than an invertible object. By the Cobordism Hypothesis, to compute the Klein invariant on a fully dualizable object $X$ in a symmetric monoidal $2$-category, one needs to equip this object with a $\mathrm{Pin}_+$-structure, i.e.\ a homotopy fixed point structure for $\rB\mathrm{Pin}_+(2) \to \rB \rO(2)$. In a $(2,2)$-category, this amounts to a self-duality $X\to X^\vee$,  which is ``self-dual'' in the sense that there is an isomorphism between $r$ and its mate which satisfies a further symmetry property, and a trivialization of the square of the Serre equivalence  $S_X: X\to X^{\vee\vee}$ (Remark~\ref{rem:Serre}) fulfilling further compatibility conditions with $r$. 

In fact, computing the Klein invariant on $X$ requires far less structure: The nonbounding $\mathrm{Pin}_+$ structure on the Klein bottle is induced by a \emph{projective framing} --- a trivialization of the projectivization of the tangent bundle of $K$ --- corresponding to the tangential structure $\rB\bZ_2\to \rB \rO(2)$. Moreover, this structure can be further lifted to a $\rB\bZ \to \rB \rO(2)$ structure, corresponding to the fact that the Klein bottle admits an integral lift of its $w_1$. By the Cobordism Hypothesis, a two-dimensional field theory with such a structure $\rB\bZ \to \rB \rO(2)$ is classified by an object $X$ equipped with a self-duality $X\to {}^\vee X$ fulfilling no further coherence condition. This explains why we only needed such a self-duality in Definition~\ref{defn:klein}. 

By embedding our Klein bottle in $\bR^4$ (in a way compatible with its tangential structure), we may further  replace ``symmetric'' with ``braided'' by using the \define{Embedded Cobordism} aka \define{Tangle hypothesis} from \cite[\S4.4]{lurie}.  (That the Klein bottle embeds in $\bR^4$ explains why our invariant will require only a braided monoidal structure on our 2-category and not a sylleptic or symmetric structure. That it does not embed in $\bR^3$ explains why it will not be defined in plain monoidal 2-categories.)
\end{remark}

Another perspective on the Klein invariant $\Inv(M, r)$ follows from an observation of Noah Snyder, who pointed out that $\Inv(M, r)$ appears as the attaching map of the $5$-cell in a minimal cell decomposition of the $4$-type of $\mathrm{BSO}(3)$, constructed in upcoming work of Douglas--Schommer-Pries--Snyder~\cite{DSPS-SO3}. Since the $4$-types of $\mathrm{BSO}(3)$ and $K(\bZ_2, 2)$ only differ by $\pi_4(\mathrm{BSO}(3))\cong \bZ$, this immediately leads to a cell decomposition of the $4$-type of $K(\bZ_2,2)$. Combined with~\cite[Theorem 8.18]{DN} which characterizes $\bZ_2$-extensions of braided fusion $1$-categories in terms of braided monoidal $2$-functors $\bZ_2 \to \BrModr{\cB}$, or equivalently homotopy classes of maps from $K(\bZ_2, 2)$ into the $4$-type corresponding to the braided monoidal groupal $2$-groupoid of invertible objects, $1$- and $2$-morphisms in $\BrModr{\cB}$,   this justifies the appearance of $\Inv(M, r)$ in our proof. Since the following Proposition~\ref{prop:CW} and Corollary~\ref{cor:CW} rely on~\cite{DSPS-SO3} which is not yet publicly available, we will not use them in the rest of the paper.
But see Remark~\ref{rem:Z2ext} for a version of our proof that uses Proposition~\ref{prop:CW}, Corollary~\ref{cor:CW}, and~\cite[Theorem 8.18]{DN} instead of Theorem~\ref{thm.BCs} and the content of~\S\ref{subsec:SandT} and~\S\ref{subsec:KleinST}. 

\begin{prop}\label{prop:CW} Let $\cC$ be a braided monoidal $2$-category. The $2$-groupoid of braided monoidal $2$-functors $\bZ_2 \to\cC$ is equivalent to the $2$-groupoid of triples $(M, r, \phi)$ of 
\begin{itemize}
\item an invertible object $M \in \cC$;
\item a $1$-equivalence $r: M^2 \to I$;
\item a $2$-isomorphism $\phi: \eta(M) \To \id_I$ (or equivalently, a $2$-isomorphism $\br_{M,M} \cong \id_{M \boxtimes M}$);
\end{itemize}
subject to the following conditions:
\begin{itemize}
\item the Klein invariant $\Inv(M, r): \id_I \To \id_I $ is the identity $\id_{\id_I}$;
\item a further obstructing $2$-isomorphism $\lambda(M, r, \phi): \id_I \To \id_I$ is the identity $\id_{\id_I}$. 
\end{itemize}
If $\alpha: \id_I \To \id_I$ is any $2$-isomorphism, then $\lambda(M, r, \alpha\cdot \phi) = \alpha^4\cdot \lambda(M, r, \phi)$. (We will not give an explicit expression for $\lambda$.)
\end{prop}
\begin{proof}
Let $\mathrm{B}^2\cC^\times$ denote the connected and simply connected $4$-type obtained from the braided monoidal groupal $2$-groupoid of invertible objects, $1$- and $2$-morphisms in $\cC$. 
In upcoming work~\cite{DSPS-SO3}, Douglas--Schommer-Pries--Snyder construct a minimal CW decomposition of the $4$-type of $\rB\mathrm{SO}(3)$. It follows from their description that the $2$-groupoid of maps $\mathrm{BSO}(3) \to \mathrm{B}^2\cC^\times$ is equivalent to the $2$-groupoid of triples $(M,r, \phi)$ as above fulfilling $\Inv(M,r)=\id_{\id_I}$ but \emph{without} the condition on $\lambda$. Since $\pi_2(\mathrm{BSO}(3))\cong \bZ_2$ and $\pi_k(\mathrm{BSO}(3)) = 0$ for $k=1,3$, the $4$-type of $K(\bZ_2, 2)$ can be built by attaching a $5$-cell along a generator of $\pi_4(\mathrm{BSO}(3)) \cong \bZ$. Hence, maps $(M, r, \phi): \mathrm{BSO}(3) \to \rB^2\cC^\times$ lift along $\mathrm{BSO}(3) \to K(\bZ_2, 2)$ if and only if the image of this generator in  $\pi_4(\rB^2\cC^\times) = \Omega^2\cC^\times$ is trivial. Denoting the image of such a generator by $\lambda(M, r, \phi)^{\pm 1} \in \Omega^2\cC^\times$, the first part of the proposition follows. (We will fix the sign ambiguity $\lambda^{\pm 1}$ in choosing a generator of $\pi_4(\mathrm{BSO}(3)) \cong \bZ$ by requiring positive scaling dependence in $\phi$.)

Even though we do not give an explicit expression for $\lambda(M, r, \phi)$ in terms of $(M, r, \phi)$, its scaling dependence in $\phi$ can be understood as follows. Since $\mathrm{BSO}(3)$ is connected and simply connected, there is an integer $n$ such that $\lambda(M, r, \alpha \phi) = \alpha^n \lambda(M, r, \phi)$ and this integer is independent of the target $\cC$. Indeed, connectivity arguments show that $\pi_4(\mathrm{BSO}(3) \vee S^4) \cong \pi_4(\mathrm{BSO}(3)) \oplus \bZ$ and hence that this relation already holds in the $4$-type of $\mathrm{BSO}(3) \vee S^4$ which is freely generated by the cells $M, r, \phi$ fulfilling $\Inv(M,r) =1 $ as above, together with an additional cell $\alpha: \id_I \To \id_I$. 
Hence, the integer $n$ can be computed in integral cohomology. It is well known that $\H^4(\mathrm{BSO}(3); \bZ) \cong \bZ$ is generated by the first Pontryagin class $\mathrm{BSO}(3) \to \mathrm{BSO} \to[p_1] K(\bZ,4)$. On the other hand, it follows from the above cell decomposition that this generator $p_1$ maps the $4$-cell $\phi$ to $1\in \bZ$.   
Hence, the scaling parameter $n$  can be computed as the image of the generator under
\[\bZ \cong \pi_4(\mathrm{BSO}(3)) \to[(p_1)_*] \pi_4(K(\bZ, 4)) = \bZ.
\]
It follows from~\cite[Theorem 4.1]{tamura} that this map is given by multiplication by $\pm 4$. \end{proof}

The following is an immediate corollary of this description.
\begin{cor}\label{cor:CW}
Let $\cC$ be a braided monoidal $2$-category and assume that the group of invertible scalars $A:=\Omega^2\cC^\times$ has fourth roots (i.e. in additive notation $A/4A \cong 0$). Then any choice $(M, r)$ of an invertible object $M\in \cC$
for which $\eta(M)$ is trivializable (or equivalently, for which $\br_{M,M} \cong \id_{M \boxtimes M}$)
 and a $1$-equivalence $r: M^2\to I$ for which $\Inv(M,r)$ is trivial
   lifts to a braided monoidal $2$-functor $\bZ_2 \to \cC$. Up to isomorphism, such lifts correspond to admissible choices of $2$-isomorphism $\phi: \eta(M) \To \id_I$ (or equivalently, of $2$-isomorphisms $\br_{M, M} \To \id_{M \boxtimes M}$) which form a torsor over $A_4:= \{\alpha \in \Omega^2 \cC^\times~|~ \alpha^4 = \id_{\id_I}\}$. \qed
\end{cor}

\subsection{The Klein invariant of magnetic objects in \texorpdfstring{$\cS$}{S} and \texorpdfstring{$\cT$}{T}}\label{subsec:KleinST}

We now compute the values of our invariant $\Inv$ in the categories $\cS$ and $\cT$ from Theorem~\ref{thm:SandT}.

\begin{example}\label{eg:invI}
  Unit objects are always canonically self-dual: we may simply identify $^\vee\one = \one$, and write the canonical self-duality as ``$\id$.'' With respect to this canonical self-duality, $\fl_I$ becomes the identity map. In particular, for any braided fusion 2-category $\cC$,
  $$ \kappa(\cC, I, \id) = 1 \in \bk = \Omega^2 \cC.$$
  
  For $\cC = \Sigma\sVec$, the unit object $\one$ has a second self-duality, namely $e : \one \to {^\vee \one}$, and:
  $$ \kappa(\Sigma\sVec, I, e) = \eta(e) = -1.$$
\end{example}

\begin{example}\label{eg:inv} 
  Take the braided fusion 2-category $\cS$, and choose a magnetic simple object $M$ and a self-duality $r : M \to {^\vee M}$, and hence a copy of $\Vec[\bZ_2] \subset \cS$. Lemma~\ref{lemma:super2fibre} provides a map $\Vec[\bZ_2] \to \Sigma\sVec$ sending $(M,r) \mapsto (I,\id)$. We therefore find
  \begin{equation} 
  \label{eqn:InvV} 
  \kappa(\cS,M,r) = \kappa(\Vec[\bZ_2], M,r)= \kappa(\Sigma\sVec, I, \id) = 1.
  \end{equation}
  In the same fashion, if we start with $\cT$ and choose $(M,r)$, then our functor $\Vec^{\tau}[\bZ_2] \to \Sigma \sVec$ sends $(M, r) \mapsto (I, e)$ and hence  \begin{equation} 
  \label{eqn:InvW} 
  \kappa(\cT,M,r) = \kappa(\Vec^\tau[\bZ_2], M,r)= \kappa(\Sigma\sVec, I, e) = -1.
  \end{equation}
  
How does this depend on the choices of $M$ and $r$? It certainly cannot depend on the choice of self-duality. Indeed, given a self-duality $r :M \to {}^\vee M$, the only other choice (up to isomorphism) is $e \boxtimes r$.  
Recall from Example~\ref{ex:computationofeta} that $\eta(e \boxtimes \id_M) = \id_{\eta(M)}$ and hence that
 \begin{equation} \label{eqn:InvM=InvMe}
  \kappa(\cS \text{ or } \cT, M, e \boxtimes r) = \tr(\fl \cdot \eta(e \boxtimes r)) = \tr(\fl \cdot \eta(r) \cdot \eta(e \boxtimes \id_M)) = \kappa(\cS\text{ or }\cT, M, r).
  \end{equation}
  We are therefore justified in writing plain ``$\kappa(\cS,M)$'' and ``$\kappa(\cT,M)$'' without declaring the self-duality.
  
  In addition to $M$, the other magnetic simple object is $C \boxtimes M$. Equation~(\ref{eqn:InvM=InvMe}) applies equally well to $C \boxtimes M$ in place of $M$, and so we are justified in writing ``$\kappa(\cS,C \boxtimes M)$'' and ``$\kappa(\cT,C \boxtimes M)$'' without declaring the self-duality. 
\end{example}

  \begin{lemma}\label{lem:invariantsSandT} In both $\cS$ and $\cT$, the invariants of both simple magnetic objects agree: 
\[
  \kappa(\cS, M) = \kappa(\cS, C \boxtimes M), \qquad \kappa(\cT, M) = \kappa(\cT, C \boxtimes M)
\]
  \end{lemma}
  \begin{proof}
  Let $v \boxtimes \id_M : M \to C \boxtimes M$ denote the unique simple $1$-morphism. 
 Recall from Example~\ref{ex:computationofeta} that $\eta(v \boxtimes \id_M): \eta(M) \to \eta(C \boxtimes M)$ is a $2$-isomorphism. Choose $1$-isomorphisms $r:M \to  {}^\vee M $ and $s:C\boxtimes M \to {}^\vee(C\boxtimes M) $. The $1$-morphisms $s \circ (v\boxtimes \id_M)$ and $ \left({}^*{}^\vee(v\boxtimes \id_M)\right) \circ r$ are isomorphic since they are both simple (being the composite of an invertible $1$-morphism and a simple $1$-morphism) and $\Hom(M, {}^\vee(C\boxtimes M)) \cong \Vec$. 
 Hence, it follows from functoriality of 
 $\eta(-)$ and naturality of $\fl$
  that the following diagram commutes:
 \[
 \begin{tikzcd}
 \eta(M) \arrow[d, Rightarrow,  "\eta(v\boxtimes \id_M)"'] \arrow[r, Rightarrow, "\eta(r)"]& \eta({}^\vee M) \arrow[r, Rightarrow, "\fl_M"] \arrow[d, Rightarrow,  "\eta({}^*{}^\vee(v\boxtimes \id_M))"']& \eta(M) \arrow[d, Rightarrow, "\eta(v\boxtimes \id_M)"]\\
 \eta(C\boxtimes M) \arrow[r, Rightarrow, "\eta(s)"'] & \eta({}^\vee (C\boxtimes M))  \arrow[r, Rightarrow, "\fl_{C\boxtimes M}"']& \eta(C \boxtimes M) 
 \end{tikzcd}
 \]
 As both $\Inv(\cS\text{ or } \cT, M)$ and $\Inv(\cS\text{ or }\cT, C\boxtimes M)$ are independent of the chosen self-duality, the assertion follows.
   \end{proof}

\begin{corollary} \label{cor:SneqT}
  $\cS$ and $\cT$ are inequivalent as braided fusion 2-categories. Each one admits an automorphism exchanging $M$ with $C \boxtimes M$.
\end{corollary}
\begin{proof}
 An equivalence $\cS \simeq \cT$ would have to take $M \in \cS$ to a magnetic simple object in~$\cT$, but this would contradict the calculations in Example~\ref{eg:inv}. To see that $\cS$ and $\cT$ admit automorphisms exchanging $M$ and $C\boxtimes M$, suppose that one were to run the proof of Theorem~\ref{thm:SandT} starting with $\cS$ or $\cT$ but selecting $C \boxtimes M$ in place of $M$. One would find a linearization description making the category equivalent to either $\cS$ or $\cT$, but since $\cS \simeq \cT$ is impossible, one would have to construct equivalences $\cS \simeq \cS$ and $\cT \simeq \cT$ exchanging the two magnetic simple objects.
\end{proof}
Note that Corollary~\ref{cor:SneqT}, and in particular the existence of an automorphism exchanging $M$ and $C\boxtimes M$, implies Lemma~\ref{lem:invariantsSandT}. 
A different proof of Corollary~\ref{cor:SneqT} (and hence also of Lemma~\ref{lem:invariantsSandT}) is given in \cite{2011.11165}. Using higher Morita categories, that paper furthermore classifies the automorphisms of $\cS$ and $\cT$: each has a $\bZ_{16}$ worth of automorphisms (matching the $\bZ_{16}$ studied in \cite{foldway}), and the odd elements in $\bZ_{16}$ (corresponding to the Ising categories) are the ones exchanging $M$ with $C \boxtimes M$.

Of course, by Example~\ref{ex:computationofeta} both $\eta(M)$ and $\eta(C \boxtimes M)$ are trivial, in which case our invariant reduces to the following simpler variant: 

\begin{definition}\label{defn:underlineversions}
Given an object $X$ in a braided fusion $2$-category $\cC$ equipped with a self-duality $r: X\to {}^\vee X$, we define $\underline{\eta}(X)$ to be the vector space $ \Hom_{\Omega\cC}(\eta(X), I)$ and define $\underline{\Inv}(X, r) \in \bk$ to be the linear trace of the automorphism $\underline{\eta}(r)\cdot \underline{\fl}_X: \underline{\eta}(X) \to \underline{\eta}(X)$, where $\underline{\eta}(r) = (-) \cdot \eta(r) : \underline{\eta}({}^\vee X) \to \underline{\eta}(X) $ and $\underline{\fl}_X = (-) \cdot \fl_X: \underline{\eta}(X) \to \underline{\eta}({}^\vee X) $.\end{definition}

\begin{proposition}\label{prop:fundamental}
  Suppose $X \in \cS$ is an object which is ``purely magnetic'' in the sense that the full braiding of $e$ with $X$ is $-\id$ and suppose that $X$ is equipped with a self-duality $r : X \isom {^\vee X}$. Then, $\Inv(\cS,X,r) = \underline{\Inv}(\cS, X, r) \in \bZ_{\geq 0} \subset \bk$. (Recall that our field $\bk$ is of characteristic zero.)
  If $X \in \cT$ is purely magnetic and self dual, then $\Inv(\cT, X, r) = \underline{\kappa}(\cT, X, r) \in \bZ_{\leq 0} \subset\bk$.\end{proposition}
\begin{proof}
  An object $X$ is purely magnetic if and only if it decomposes as a direct sum of copies of $M$ and $C \boxtimes M$.  
  Fix such a decomposition. This choice selects a direct sum decomposition of $\eta(X)$ as a sum of copies of $\eta(M)$ and $\eta(C \boxtimes M)$.
  It follows from Example~\ref{ex:computationofeta} that $\eta(X)$ is a multiple of $\id_I$ and therefore $\Inv(X, r) = \underline{\Inv}(X, r)$. 
  
Fix reference self-dualities $M \isom {^\vee M}$ and $C \boxtimes M \isom {^\vee(C \boxtimes M)}$. These choices select a reference self-duality $r_{\text{ref}}$ for $X$. 
For this reference self-duality, the matrix $\underline{\eta}(r_{\text{ref}}) \cdot\underline{\fl} $ whose trace gives $\underline{\Inv}(X, r_{\text{ref}})$ is diagonal with respect to our fixed direct sum decomposition of $\underline{\eta}(X)$ into $\underline{\eta}(M) \cong \bk$ and $\underline{\eta}(C\boxtimes M) \cong \bk$. 
Indeed, \eqref{eqn:InvV}  implies that in the $\cS$ case, $\underline{\eta}(r_{\text{ref}}) \cdot\underline{\fl} $ is the identity matrix, and  \eqref{eqn:InvW}  implies that in the $\cT$ case it is minus the identity matrix.

  Using our chosen self-duality $r: X \isom {^\vee X}$, note that $r_{\text{ref}}^{-1} \circ r$ is an automorphism of $X$. Following the last paragraph of Remark~\ref{remark.2catDirectSums}, with respect to our fixed decomposition of $X$ into simples, this automorphism is the composite of a permutation of summands followed by a diagonal matrix of automorphisms. As any automorphism of a simple object in $\cS$ or $\cT$ is either the identity or $e \boxtimes \id$ and since $\eta(e \boxtimes_{\id_M}): \eta(M) \to \eta(M)$ and $\eta(e \boxtimes \id_{C\boxtimes M}): \eta(C\boxtimes M) \to \eta(C \boxtimes M)$ are trivial (see Example~\ref{ex:computationofeta}), it follows that $\fl \cdot \eta(r)$ is given by a permutation of summands followed by $\fl \cdot \eta(r_{\text{ref}})$. In other words, for $\cS$ it is a permutation matrix and for $\cT$ it is minus a permutation matrix.

  But the trace of a permutation matrix is always nonnegative (as it simply counts the number of fixed points of the permutation).
\end{proof}

Note that the case $\Inv(X, r) =0 $ can occur both in $\cS$ and $\cT$, and it occurs precisely if no indecomposable summand of $X$ is fixed under the self-duality $r:X \to {^\vee X}$. Therefore, only strictly positive or strictly negative Klein invariant on a purely magnetic object can be used to distinguish $\cS$ and $\cT$.

\subsection{Twisted traces and \texorpdfstring{$\eta$}{eta} of a half-braided algebra}\label{subsec:twistedtrace}

Our next goal is to study our invariant $\Inv$ in the braided monoidal 2-category $\cZ(\Sigma\cB)$. We will do so by using Theorem~\ref{thm:HBADrinfeld}, which translates the problem to a computation for an arbitrary half-braided algebra $(A, \gamma)$. The advantage of this approach is that half-braided algebras allow quite explicit formulas for the values of $\eta(-)$ and $\fl$.

\begin{definition}\label{def:hbtwist}
Given a half-braided object $(b,\gamma) \in \cZ(\cB)$, we define the \define{half-braided twist} as the following isomorphism $\theta_{b, \gamma}: b\to b$ in $\cB$:
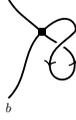
\[\begin{tz}[std,scale=0.8]
\draw[braid, arrow data={0.7}{<}] (-1, 2) to [out=-60,in=135] (0,1) to [out=-45, in=up]  (1,0) to [out=down, in=down, looseness=2 ] ( 0.25,0);
\draw[braid =main, arrow data={0.1}{<}] (0.25,0) to [out=up, in=-135] (1,1);
\draw[string] (-1, -1) to [in=-135] (0,1) to [out=45, in=45, out looseness=3,in looseness=0.5] (1,1);
\node[smalldot] at(0,1){};
\node[label, below] at (-1,-1){$b$};
\end{tz}
\]
\end{definition}

\begin{definition}\label{def:twistedcocenter}
A \emph{twisted trace} on a half-braided algebra $(A,\gamma)$, with multiplication $\mu_A : A \otimes A \to A$ and unit $\nu_A : I \to A$,
 is a morphism $\epsilon: A\to I$ in $\cB$ such that  the following holds:
\begin{equation}\label{eq:twistedtrace}
\epsilon \cdot \mu_A \cdot (\theta_{A, \gamma} \otimes \id_A) = \epsilon \cdot \mu_A \cdot \beta^{-1}_{A,A}
\end{equation}
\end{definition}

\begin{example}
  Equation~\eqref{eq:twistedtrace} is the assertion that the trace $\epsilon$ is cyclic, but with cyclicity twisted by $\theta_{A, \gamma}$. If the object $A$ is in the M\"uger centre $\cZ_{2}(\cB)$ with canonical half-braiding $\gamma_b = \beta_{b, A}= \beta_{A,b}^{-1}$, then $\theta_{A, \gamma} = \id_A$ and a twisted trace on $(A, \gamma)$ is just an ordinary cyclic trace.
\end{example}

By definition, a twisted trace is a morphism from the ``universal trace''
\begin{equation}\label{eq:coeq}
\coeq\left( 
\begin{tikzcd}A \otimes A \arrow[rrrr, shift left = 1, "\mu_A\cdot (\theta_{A, \gamma} \otimes \id_A) = \mu_A\cdot \beta_{A,A}^{-1} \cdot \gamma_A"]\arrow[rrrr, shift right=1, "\mu_A \cdot \beta_{A,A}^{-1}"'] &&&& A
\end{tikzcd}\right)
\end{equation}
into the tensor unit $I$. The equality $ \mu_A\cdot (\theta_{A, \gamma} \otimes \id_A)= \mu_A \cdot \beta^{-1}_{A,A} \cdot \gamma_A $ follows directly from the defining property~\eqref{eq:hb} of a half-braided algebra and will be used repeatedly in this section. 

The first goal of this section is to prove that this coequalizer~\eqref{eq:coeq} is canonically isomorphic to $\eta(A,\gamma)$ (and is in particular an object of $\cZ_{2}(\cB)$). The vector space of twisted traces can therefore be identified with $\underline{\eta}(A, \gamma) = \Hom_{\cZ_{2}(\cB)}(\eta(A, \gamma), I)$ from Definition~\ref{defn:underlineversions}.

To this end, we carefully choose convenient duality data for $(A, \gamma)$ as an object of $\sHBA(\cB)\cong \cZ(\Sigma \cB)$. Given an algebra object $A$ in a braided monoidal category $\cB$, we define its \emph{right opposite} $A^\op$  and its \emph{left opposite} ${{}^\op A}$ as the algebra objects with the same underlying object and unit $\nu_A$ as $A$, and multiplication defined as follows, respectively:
\begin{equation}
\label{eq:defmop} \mu_{A^\op} := \mu_A \cdot \beta^{-1}_{A,A}\hspace{1.5cm}
\mu_{{{}^\op A}} := \mu_A \cdot \beta_{A,A}
\end{equation}
If $(A, \gamma)$ is a half-braided algebra, then the following isomorphisms $ x\otimes A\to A \otimes x$, natural in $x\in \cB$,  define half-braidings on the algebras $A^\op$ and ${}^\op A$, respectively: 
\[(\gamma^\op)_x := \beta_{x, A}\cdot \gamma_x^{-1} \cdot  \beta^{-1}_{A,x} \hspace{1cm} (^{\op} \gamma)_x:= \beta^{-1}_{A,x} \cdot \gamma^{-1}_x \cdot \beta_{x, A}
\] 

The half-braided twist $\theta_{A, \gamma} :A \to A$ defines an isomorphism of half-braided algebras $\theta_{A, \gamma}:  (A^\op, \gamma^\op) \to ({}^\op A, {}^\op \gamma) $. This can be verified directly but also follows from Remark~\ref{rem:Serre} below.

\begin{lemma}\label{lem:dualityHBA}

The half-braided bimodules $A_{A\otimes A^{\op}}$ and $_{A^\op \otimes A}A$ with action 
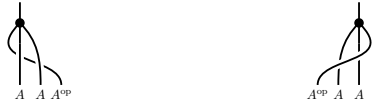
\[
\begin{tz}[std, xscale=-1]
%\clip (-0.5, -0.4) rectangle (1.2, 2);
\draw[string, on layer=front] (0,1) to (0,2);
\draw[string] (-1,0) to [out=up, in=-45, in looseness=2] (0,1.5);
\draw[braid=main] (-0.5, 0) to [out=up, in=-135] (0, 1.5);
\draw[braid=main] (0,0) to (0,1);
\node[dot] at (0,1.5) {};
\node[label, below] at (0,0) {$A$};
\node[label, below] at (-0.5,0) {$A$};
\node[label, below] at (-1,0) {$A^{\op}$};
\end{tz}
\hspace{3cm}
\begin{tz}[std]
%\clip (-0.5, -0.4) rectangle (1.2, 2);
\draw[string] (0,0) to (0,1);
\draw[string, on layer=front] (0,1) to (0,2);
\draw[string] (-0.5, 0) to [out=up, in=-135] (0, 1.5);
\draw[braid=main] (-1,0) to [out=up, in=-45, in looseness=2] (0,1.5);
\node[dot] at (0,1.5) {};
\node[label, below] at (0,0) {$A$};
\node[label, below] at (-0.5,0) {$A$};
\node[label, below] at (-1,0) {$A^{\op}$};
\end{tz}
\]
define evaluation $\ev_A: A\boxtimes A^{\op} \to I$  and coevaluation $\coev_A:I \to A^{\op} \boxtimes A$ $1$-morphisms which together with the cusp $2$-isomorphism $\cusp_A: \id_A \Rightarrow (\ev_A \boxtimes \id_{A}) \circ (\id_A \boxtimes \coev_A)$
 given by the following $A$--$A$ bimodule isomorphism (with inverse $\mu_A\cdot \beta^{-1}_{A,A}$) 
 \begin{equation}\label{eq:zigzagalgebra} \id_A \otimes \nu_A = \nu_A \otimes \id_A: A \to  \coeq\left(~\begin{tz}[std]
%\clip (-0.5, -0.4) rectangle (1.2, 2);
\draw[string] (0,0) to (0,1);
\draw[string] (-1,0) to (-1,2);
\draw[string, on layer=front] (0,1) to (0,2);
\draw[braid=main] (-0.5,0) to [out=up, in=-45, in looseness=2] (0,1.5);
\node[dot] at (0,1.5) {};
\node[label, below] at (0,0) {$A$};
\node[label, below] at (-0.5,0) {$A^{\op}$};
\end{tz}
,
\begin{tz}[std, xscale=-1]
%\clip (-0.5, -0.4) rectangle (1.2, 2);
\draw[string, on layer=front] (0,1) to (0,2);
\draw[string] (-1, 0) to(-1, 2);
\draw[string] (-0.5,0) to [out=up, in=-45, in looseness=2] (0,1.5);
\draw[braid=main] (0,0) to (0,1);
\node[dot] at (0,1.5) {};
\node[label, below] at (0,0) {$A$};
\node[label, below] at (-0.5,0) {$A^{\op}$};
\end{tz}~
\right)
\end{equation}
exhibit the half-braided algebra $(A^{\op}, \gamma^{\op})$ as a right dual to $(A, \gamma)$ in $\sHBA(\cB)$. Analogous choices exhibit $({}^\op A, {}^\op \gamma)$ as a left dual. \end{lemma}
\begin{proof} After using the canonical isomorphism $A\otimes_A A \to A$ where necessary to identify $(\ev_A \boxtimes \id_{A}) \circ (\id_A \boxtimes \coev_A)$ with the coequalizer in~\eqref{eq:zigzagalgebra}, the proof is immediate. \end{proof}

\begin{remark}
Recall that any (not necessarily braided) left $\cB$-module category $({}_{\cB}\cM, *)$ may be turned into a right $\cB$-module category $(\cM_{\cB}, *')$ with $\cB$-action $m *' b := b* m$ and $\cB$-module associator 
\begin{equation} \label{eq:moduleturnaround}(m*'b)*'c = c *(b*m) \cong (c\otimes b) * m \to[\beta_{c,b}*m] (b\otimes c) * m = m*' (b\otimes c).
\end{equation}
If $_{\cB} \cM \cong \Modr[\cB]{A}$ for a (not necessarily half-braided) algebra $A$ in $\cB$, one may show that $\cM_{\cB} \cong \Modl[\cB]{A^\op}$. 
In fact, in~\eqref{eq:moduleturnaround} we could have used $\beta^{-1}_{b,c}$ instead of $\beta_{c,b}$ (or more generally any odd integral power of the braiding) to obtain the category $\Modl[\cB]{{}^\op A}$ from the category $\Modr[\cB]{A}$ (or more generally left modules for an algebra structure on $A$ given by $m_A$ twisted by an odd number of braidings).

It follows from this observation and Lemma~\ref{lem:dualityHBA} (or alternatively, by direct verification), that a right dual of a braided $\cB$-module category $(\cM_{\cB}, *, \sigma)$ is given by the category $\mathrm{Fun}_{\cB}(\cM_{\cB}, \cB_{\cB})$ of $\bk$-linear $\cB$-module functors with an appropriate right $\cB$-module structure and module braiding. Indeed, $\mathrm{Fun}_{\cB}(\cM_{\cB}, \cB_{\cB})$ inherits an evident \emph{left} $\cB$-module structure (mapping an object $b \in \cB$ and a module functor $F$ to the module functor $b\otimes F(-)$) which one turns into a right module structure using~\eqref{eq:moduleturnaround}. The $\cB$-module braiding $F*' b \to F*' b$ is inherited from $\cM$ as follows:
\begin{equation}\label{eq:modulebraidingop}
 b\otimes F(m) \to[\br_{b, F(m)}]F(m) \otimes b \cong F(m*b) \to[F(\sigma_{m, b}^{-1})] F(m *b) \cong F(m) \otimes b \to[\br_{F(m), b}] b\otimes F(m)
\end{equation}
If $\cM_{\cB} = \Modl[\cB]{A}$ with module braiding induced by a half-braiding $\gamma$ on $A$ (as in Lemma \ref{lem:hbvsBmodbraiding}), then the left $\cB$-module category $\mathrm{Fun}_{\cB}(\Modl[\cB]{A}, \cB_{\cB})$ is equivalent to $\Modr[\cB]{A}$ (see \cite[Proposition 7.11.1]{EGNO}), so that $\mathrm{Fun}_{\cB}(\Modl[\cB]{A}, \cB_{\cB})$ with the above right $\cB$-module structure is equivalent to $\Modl[\cB]{A^\op}$ with module braiding induced by $\gamma^\op$. 

Similarly, the left dual of $(\cM_{\cB}, *, \sigma)$ is equivalent to the category $\mathrm{Fun}_{\cB}(\cM_{\cB}, \cB_{\cB})$ with right module structure obtained from the canonical left module structure by replacing $\beta_{c,b}$ by $\beta^{-1}_{b,c}$ in~\eqref{eq:moduleturnaround} and with the same expression~\eqref{eq:modulebraidingop} for the module braiding. 
\end{remark}

For the following proofs, we adopt the following notation. Given a right $B$ module $m_B$ (a left $A$ module $_{A}n$) and an algebra isomorphism $\phi:A \to B$, we write $m_{\phi}$ (or $_{\phi}n$) for the resulting $A$ module ($B$ module). There are canonical isomorphisms $m\otimes_B \phi_* \cong m_\phi$ and $(\phi^{-1})_* \otimes _A n \cong {}_{\phi}n$, where $\phi_* := {_BB_{\phi A}}$ from Example~\ref{eg:phistar} is the canonical bimodule built from $\phi$.

Recall that the braiding $\br^{\sHBA}_{(A, \gamma), (B, \zeta)}: (A,\gamma) \boxtimes (B, \zeta) \to (B, \zeta) \boxtimes (A,\gamma)$ of half-braided algebras is the bimodule represented by the half-braided algebra isomorphism $\zeta_A: A\otimes B \to B \otimes A$. 
Hence, using the duality choices from Lemma~\ref{lem:dualityHBA}, we can simplify 
\[\eta(A, \gamma) := \ev_{(A,\gamma)} \circ \br^{\sHBA}_{({}^\op A, {}^\op \gamma), (A, \gamma)} \circ \coev_{(A,\gamma)} \cong A_{\gamma_A} \otimes_{(A^\op \otimes A)}A.
\]

\begin{prop}\label{prop:etaiscenter}
Let $(A,\gamma)$ be a half-braided algebra. Then, the morphisms 
\[A \to[\nu_A \otimes \id_A] A\otimes A \twoheadrightarrow A_{\gamma_A} \otimes_{(A^\op \otimes A)} A \cong \eta(A,\gamma)
\]
\[A \to[ \id_A \otimes \nu_A] A\otimes A \twoheadrightarrow A_{\gamma_A} \otimes_{(A^\op \otimes A)} A \cong \eta(A,\gamma)
\]
agree and exhibit $\eta(A, \gamma) \in \cZ_{2}(\cB)$ as the coequalizer~\eqref{eq:coeq}.

In particular, the induced morphism $ \underline{\eta}(A,\gamma) \to \Hom_{\cB}(A, I)$ may be identified with the inclusion of the vector space of twisted traces as in Definition~\ref{def:twistedcocenter}.\end{prop}

Even though the object $\eta(A, \gamma) \in \cZ_{2}(\cB)$ only depends on the Morita class of $A$, the morphism $A \twoheadrightarrow \eta(A, \gamma)$ depends on $A$ as an algebra. 

\begin{proof}  
Recall that for algebra objects $B$ and $C$ in a braided monoidal $1$-category $\cB$, a right $B\otimes C$ module $m_{B \otimes C}$ is the same as an object $m$ with commuting right $B$ and $C$ actions, or equivalently an $B^{\op}$--$C$ bimodule. Similarly, a left $B\otimes C$ module is equivalently a $C$--${}^\op B$ bimodule. In particular, given a left and a right $B\otimes C$ module, their relative tensor product $m\otimes_{(B\otimes C)}n$ is isomorphic to the coequalizer of the morphism $B \otimes (m\otimes_C n) \to m\otimes_C n$ induced by the left $B^{\op}$-action on $m$ and the composite $B \otimes (m\otimes_C n)  \cong (m\otimes_C n) \otimes B \to m\otimes_C n$ where the first morphism is the braiding $\beta^{-1}_{m\otimes_C n, B}$ in $\cB$ and the second morphism is induced by the right ${}^\op B $ action on $n$.

In our case, the multiplication map $\mu_A:A \otimes A \to A$ descends to an isomorphism $A\otimes_A A \cong A$ and the induced left $(A^\op)^\op$ and right ${}^\op(A^\op) = A$ actions on $A$ become
\[\mu_A \cdot \beta_{A,A}^{-1} \cdot \gamma_A: (A^{\op})^{\op} \otimes A \to A \hspace{1cm} \mu_A: A\otimes A \to A.
\]
By the above discussion, this exhibits $A_{\gamma_A} \otimes_{(A^\op \otimes A)}A $ as the coequalizer of $\mu_A\cdot \beta_{A,A}^{-1} \cdot \gamma_A$ and $\mu_A \cdot \beta_{A,A}^{-1}$. 
\end{proof}
\begin{remark}\label{rem:Serre}Form the perspective of the braided monoidal $2$-category $\sHBA(\cB)$, the (invertible bimodule induced by the) half-braided twist $\theta_{A,\gamma}$, seen as a half-braided algebra isomorphism $(A^\op, \gamma^\op)  \to ({}^\op A, {}^\op\gamma)$, is isomorphic to the \emph{Serre isomorphism}, defined for any dualizable object $X$ in a braided monoidal $2$-category as the following $1$-equivalence $S_X: X^\vee \to {}^\vee X$: 
\[
\begin{tz}[std]
\draw[string] (0,0) to (0,0.5) to [out=up, in=up, in looseness=2.5] (-1, 1);
\draw[braid=main] (-1,1)  to [out=down, in=down, out looseness=2.5] (0, 1.5) to (0,2);
\node[label, below] at (0,0) {$ X^\vee$};
\node[label, above] at (0, 2) {${}^\vee X$};
\node[label, left] at (-1.05,1) {$X$};
\node[label, above] at (-0.75,1.6) {$\ev_X$};
\node[label, below] at (-0.75,0.4) {$\coev_{{}^\vee X}$};
\path[arrow data ={0.1}{>}] (-1, 1) to (-1, 2);
% [out=up, in=up, looseness=2] (-1,2) to (-1,1) to [out=down, in=down, looseness=2] (-0.5, 1) to [out=up, in=down] (0, 2) to (0,3);
\end{tz}
\]
In particular, the presentation of $\eta(A, \gamma)$ as the coequalizer~\eqref{eq:coeq} therefore follows more closely from the decomposition $\eta(X) \cong  \ev_{{}^\vee X} \circ (S_{X} \boxtimes \id_X)  \circ \coev_X$ of our framed circle.
\end{remark}

\begin{remark}\label{rem:invmate} 
For bimodules $\phi_*$ induced by  isomorphisms $\phi: A \to B$ of half-braided algebras, there is a convenient choice of mate given by the bimodule $_{A^\op}((\phi^{-1})_*)_{B^\op}$ induced by $\phi^{-1}$ thought of as an algebra isomorphism $\phi^{-1}: B^\op \to A^\op$. The rotation isomorphisms  $\rot^{\phi_*}$ and $\rot_{\phi_*}$ characterizing the mate (see Definition~\ref{defn:duals}) are given by the following two compositions:
\begin{gather*}
A \otimes_{(A\otimes A^{\op})} (\id_A \otimes \phi^{-1})_* \cong A_{(\id_A \otimes \phi^{-1}) }\overset{\phi}{\cong} B_{(\phi \otimes \id_{B^\op}) }
\cong B \otimes_{B\otimes B^{\op}} (\phi \otimes \id_{B^\op})_*  \\
 (\id_{A^\op} \otimes \phi)_* \otimes_{A^{\op} \otimes A}A  \cong _{( \id_{A^\op} \otimes \phi^{-1})}A \overset{\phi}{\cong} _{( \phi \otimes \id_B)}B \cong (\phi^{-1} \otimes \id_B)_* \otimes_{B^{\op} \otimes B}B 
\end{gather*}
\end{remark}

\begin{lemma} Let $\phi:(A, \gamma) \to (B, \zeta)$ be an isomorphism of half-braided algebras, equipped with the duality data from Lemma~\ref{lem:dualityHBA} and let $A \twoheadrightarrow \eta(A, \gamma)$ and $B \twoheadrightarrow \eta(B, \zeta)$ be the universal morphisms from Proposition~\ref{prop:etaiscenter}. Then, the following diagram commutes: 
\[\begin{tikzcd}
A \arrow[r, "\phi"]\arrow[d, twoheadrightarrow]& B \arrow[d, twoheadrightarrow] \\
\eta(A, \gamma) \arrow[r, "\eta(\phi_*)"] & \eta(B, \zeta) 
\end{tikzcd}
\]
\end{lemma}
\begin{proof}
According to Lemma~\ref{lemma:etafunctor}, to compute $\eta(\phi_*)$ we may choose a mate and a right adjoint for $_{B}(\phi_*)_A$. As a mate we choose the bimodule $_{A^\op}((\phi^{-1})_*)_{B^\op}$ induced by the half-braided algebra isomorphism $\phi^{-1}: B^\op \to A^\op$ defined in Remark~\ref{rem:invmate}. 
As a right adjoint (equivalence) we choose the bimodule $_{A}((\phi^{-1})_*)_B$, again induced by $\phi^{-1}$ but this time thought of as a half-braided algebra isomorphism $\phi^{-1}: B \to A$. 

Tracing through \eqref{eq:etaf}, we compute $\eta(\phi_*)$ as follows:
\begin{align*}
A_{\gamma_A}\otimes_{(A^\op \otimes A)} A 
& =A_{\gamma_A}\otimes_{(A^\op \otimes A)} {}_{(\id_{A^\op} \otimes \phi \phi^{-1})}A  \\
&  \cong\hspace*{-4ex}\overset{\id_A \otimes \phi}{\phantom\cong} A_{\gamma_A}\otimes_{(A^\op \otimes A)} {}_{(\phi \otimes \phi )}B\\
&  \cong A_{  \gamma_A (\phi^{-1} \otimes \phi^{-1})} \otimes_{B^\op \otimes B} B\\
& = A_{(\phi^{-1} \otimes \phi^{-1}) \zeta_B}\otimes_{(B^\op \otimes B)}B \\
& \cong\hspace*{-4ex}\overset{ \phi \otimes \id_B}{\phantom\cong} B_{\zeta_B} \otimes_{(B^\op \otimes B)} B
\end{align*}

 Except for the labelled isomorphisms $\id_A \otimes \phi$ and $\phi \otimes \id_A$, all isomorphisms in the above sequence lie under the trivial maps $\id_{A\otimes A}$ and $\id_{A\otimes B}$, respectively. 
Hence, $\eta(\phi_*)$ is given by the map $\phi \otimes \phi:A \otimes A \to B \otimes B $ descended to the coequalizers $A_{\gamma_A} \otimes_{(A^\op \otimes A)}A \to B_{\zeta_B} \otimes_{(B^\op \otimes B)} B$. In particular, this lifts further along $\nu_A \otimes \id_A:A \to A\otimes A$ to the map $\phi:A \to B$.
\end{proof}

\begin{remark}\label{rem:opmate}
Using the choices of right dual from Lemma~\ref{lem:dualityHBA}, a convenient choice of mate  for a half-braided bimodule $_{B}m_A : A \to B$ is the half-braided bimodule $_{A^\op\vphantom{B^\op}} (m^{\op})_{B^\op} :B^\op \to A^\op$ with the same underlying object $m$ in $\cB$ and ``rotated'' action $\mathrm{act}_m\cdot \beta^{-1}_{B, m\otimes A} \cdot (\beta^{-1}_{m, A} \otimes \id_B): A^\op \otimes m \otimes B^\op \to m$. 
With these choices, the canonical isomorphism from the mate of a composite $_{A^\op}(n\otimes_B m)^\op_{C^\op}$ to the composite of mates $_{A^\op} (m^\op) \otimes_{B^\op} (n^\op)_{C^\op}$ descends from the isomorphism $\beta_{n, m}: n \otimes m \to m \otimes n$.

In Remark~\ref{rem:invmate} we  discussed another choice of mate for bimodules $\phi_*$ arising from algebra isomorphisms $\phi:A \to B$ : The bimodule $(\phi^{-1})_*$ induced by the algebra-isomorphism $\phi^{-1}: B^\op \to A^\op$. 
The canonical isomorphism between these two choices of mate from  $(\phi^{-1})_*$ to $(\phi_*)^{\op}$  is given by $\phi$. 
\end{remark}

By Proposition~\ref{prop:etaiscenter}, both $\eta(A, \gamma)$ and $\eta({}^\op A, {}^\op \gamma)$ may be expressed as a quotient of the same object $A$ of $\cB$.

\begin{lemma}
Let $(A, \gamma)$ be a half-braided algebra equipped with the choice of left and right dual from Lemma~\ref{lem:dualityHBA} and let $A \twoheadrightarrow \eta(A, \gamma)$ and $A \twoheadrightarrow \eta({}^\op A,{}^\op\gamma)$ be the universal morphisms from Proposition~\ref{prop:etaiscenter}.
Then, the following diagram commutes: 
\[\begin{tikzcd}
&A \arrow[dl, twoheadrightarrow] \arrow[dr, twoheadrightarrow]&\\ \eta({}^\op A, {}^\op \gamma) \arrow[rr, "\fl_{A,\gamma}"]& & \eta(A, \gamma)
\end{tikzcd}
\] \end{lemma}
\begin{proof}
Besides the choice of right and left dual $(A^\op, \gamma^\op)$ and $({}^\op A, {}^\op \gamma)$ determined by Lemma~\ref{lem:dualityHBA}, we make the following three additional choices, precisely corresponding to three choices required in Remark~\ref{rem:complicatedfl}. For better readability, we make these choices more generally for half-braided algebras $(B, \zeta)$ and $(C, \xi)$ and specialize to the relevant case involving $(A, \gamma)$, $(A^\op, \gamma^\op)$ and $({}^\op A, {}^\op \gamma)$ in the computation below.
\begin{enumerate}
\item  \label{nucommutes1}
For the tensor product of half braided algebras $(B, \zeta) \boxtimes (C, \xi) = (B\otimes C, \zeta \boxtimes \xi)$, see~\eqref{eq:tensor}, we choose the right dual half-braided algebra $\left((B\otimes C)^{\op} ,(\zeta \boxtimes \xi)^{\op}\right)$ as described in Lemma~\ref{lem:dualityHBA}. 
\item In contrast to our choice in \ref{nucommutes1}., the canonical tensor product dual for $(B, \zeta) \boxtimes (C, \xi)$ is given by the half-braided algebra $(C^\op, \xi^\op) \boxtimes (B^\op ,\zeta^\op)$. As equivalence of duality data $((B \otimes C)^\op, (\zeta \boxtimes\xi)^\op) \to (C^\op \otimes B^\op, \xi^\op \otimes \zeta^\op)$ we choose the bimodule $(\beta_{B,C})_*$ induced by the half-braided algebra isomorphism $\beta_{B, C}: (B \otimes C)^\op \to C^\op \otimes B^\op$ given by the braiding of our underlying braided monoidal $1$-category $\cB$. 
\item For evaluation and coevaluation bimodules we choose as mate the ``rotated bimodules'' of Remark~\ref{rem:opmate}, for $\br^{\sHBA}_{(B, \zeta), (C, \xi)} = (\xi_B)_*$, we choose the bimodule $(\xi^{-1}_B)_*$. These mates are all with respect to the duality data on the tensor product as in \ref{nucommutes1}.\end{enumerate} 
With these choices, and using the notation from Remark~\ref{rem:complicatedfl} for $X=(A, \gamma)$,  the canonical isomorphisms relevant for the computation of $\fl$ are given by:
\begin{multline}\label{mateA1}
(\coev_{{}^\vee X})^\vee \circ \omega_L^{-1} = \left(_{A\otimes {}^\op A}{}^\op A\right)^{\op}\otimes_{(A\otimes {}^\op A)^\op}(\beta_{A,A}^{-1})_* \\
\cong \left(_{A\otimes {}^\op A}{}^\op A\right)^{\op}_{ \beta_{A,A}^{-1}} = A_{A\otimes A^{\op}} = \ev_{X}
\end{multline}
\begin{multline}\label{mateA2}
\omega_R \circ (\ev_{{}^\vee X})^\vee = (\beta_{A,A})_* \otimes_{({}^\op A \otimes A)^\op } \left( {}^\op A_{{}^\op A \otimes A} \right)^{\op}
\\
\overset{\beta_{A,A}^{-1}}\cong \left( {}^\op A_{{}^\op A \otimes A} \right)^{\op} = {_{A^\op \otimes A} A} = \coev_X
\end{multline}
\begin{multline}\label{mateA3}
\omega_L \circ (\br^{\sHBA(\cB)}_{X, {}^\vee X})^{\vee} \circ \omega_R^{-1}  = (\beta_{A,A})_* \otimes_{(A\otimes {}^\op A)^\op} ({}^\op\gamma_A^{-1})_* \otimes_{({}^\op A \otimes A)^\op} (\beta_{A,A}^{-1})_*
\\
\cong \left( \beta_{A,A} \cdot {}^{\op} \gamma_A^{-1} \cdot \beta^{-1}_{A,A}\right)_* = (\gamma_A)_* = \br^{\sHBA}_{X^\vee, X}
\end{multline}
We then compute 
\begin{align*}
  {}^\op A_{{}^\op \gamma} \otimes_{A \otimes{}^\op A} {}^\op A 
  & =   \left({}^\op A _{{}^\op \gamma}\otimes_{A \otimes {}^\op A} {}^\op A \right)^{\op} \\
  & \cong\hspace{-3.75ex}\overset{\beta_{A,A}}{\phantom\cong} \left( {}^\op A \right)^{\op} \otimes_{(A\otimes {}^\op A)^\op} \left( {}^\op A_{{}^\op \gamma_A}\right)^{\op} \\
& \cong  ({}^\op A)^\op_{({}^\op \gamma_A)^{-1}}  \otimes_{({}^\op A \otimes A)^\op}  ({}^\op A )^\op \\
& = ({}^\op A)^\op_{\beta_{A,A}^{-1} \beta^{\phantom{1}}_{A,A}({}^\op \gamma_A)^{-1}}  \otimes_{({}^\op A \otimes A)^{\op}}  {}_{\beta^{\phantom{1}}_{A, A}\beta^{-1}_{A,A}}({}^\op A )^\op \\
& \cong ({}^\op A)^\op_{\beta_{A,A}^{-1} (\beta^{\phantom{1}}_{A,A}({}^\op \gamma_A)^{-1} \beta_{A,A}^{-1})}  \otimes_{A^\op \otimes A}  {}_{\beta^{-1}_{A,A}}({}^\op A )^\op \\
& \cong A_{\gamma_A} \otimes_{A^\op \otimes A} A
\end{align*}
Here, the last isomorphism implements the canonical isomorphisms between choices of mates from (\ref{mateA1}--\ref{mateA3}), while $\beta_{A,A}$ implements the isomorphism between the mate of composites and the composite of mates (see Remark~\ref{rem:opmate}). Except for the labelled isomorphism $\beta_{A,A}$, all isomorphisms lie below the trivial map $\id_{A\otimes A}: A\otimes A \to A\otimes A$. In particular, the following square commutes 
\[\begin{tikzcd}
A\otimes A\arrow[d, twoheadrightarrow] \arrow[r, "\beta_{A,A}"] & A\otimes A\arrow[d, twoheadrightarrow] \\ 
\eta({}^\op A, {}^\op \gamma) \arrow[r, "\fl_{A,\gamma}"] & \eta(A, \gamma) 
\end{tikzcd}
\]
Using the  maps $A\twoheadrightarrow \eta({}^\op A, {}^\op \gamma)$ and $A\twoheadrightarrow \eta(A, \gamma)$ from Proposition~\ref{prop:etaiscenter}  and noting that $\beta_{A,A}\cdot (\id_A \otimes \nu_A) = \nu_A\otimes \id_A $ (where as above $\nu_A : I \to A$ denotes the multiplicative unit), it follows that $\fl$ lies under the identity map $\id_A:A\to A$. 
\end{proof}

\begin{corollary}\label{cor:mainHBA} Let $(A, \gamma)$ be a half-braided algebra equipped with the choice of left and right dual $({}^\op A, {}^\op \gamma)$ and $(A^\op, \gamma^\op)$ from Lemma~\ref{lem:dualityHBA},  and let $\phi: (A, \gamma) \to ({}^\op A, {}^\op\gamma) $ be a half-braided algebra isomorphism. Then, the following diagram commutes 
\[
\begin{tikzcd} 
A \arrow[rr, " \phi"] \arrow[d, twoheadrightarrow]  && A \arrow[d, twoheadrightarrow]\\
\eta(A, \gamma) \arrow[rr, "\fl_{A, \gamma} \cdot \eta(\phi)"] && \eta(A, \gamma) 
\end{tikzcd}\]
where $A\twoheadrightarrow \eta(A, \gamma)$ denotes the universal map from Proposition~\ref{prop:etaiscenter} identifying $\eta(A, \gamma)$ with the coequalizer~\eqref{eq:coeq}.

In particular, this identifies the induced automorphism $\underline{\eta}(\phi)\cdot \underline{\fl}_{A, \gamma}: \underline{\eta}(A, \gamma) \to \underline{\eta}(A, \gamma)$ from Definition~\ref{defn:underlineversions} with the automorphism of the vector space of twisted traces (see Definition~\ref{def:twistedcocenter}) which maps a twisted trace $\epsilon:A \to I$ to the twisted trace $\epsilon \cdot \phi$. 
\end{corollary}

\subsection{Twisted traces and \texorpdfstring{$\eta$}{eta} for the Lagrangian object \texorpdfstring{$\cL$}{L} in \texorpdfstring{$\cZ(\Sigma \cB)$}{Z(Sigma B)}}\label{subsec:etaL}

Recall from \S\ref{subsec.canonicalLagHBA} that $\cZ(\Sigma \cB)$ has a \emph{distinguished Lagrangian object} (in fact, Lagrangian braided monoidal object) $\cL$ represented by the braided $\cB$-module category $\cZ(\cB)$ or the half-braided algebra $L = \int^{b \in \cB} b \otimes b^*$ with half-braiding $\lambda$, as given in Definition~\ref{defn.lagHBA}. As in Section~\ref{subsec.canonicalLagHBA} we will define morphisms out of this object $L$ in terms of dinatural transformations. 

We first compute the half-braided twist of $(L, \lambda)$, as defined in Definition~\ref{def:hbtwist}.
\begin{lemma}The half-braided twist $\theta_{L,\lambda}:L \to L$ of the half-braided algebra $(L, \lambda)$ is defined in terms of the following dinatural transformation $b\otimes b^* \to L$:
\begin{equation}\label{eq:fulltwist}
\begin{tz}[std, scale=1,yscale=0.8]
\draw[string, dinat] (-0.75, 2) to [out=down, in=down, out looseness=1, in looseness=3] (0, 4);
\draw[string, dinat] (-1.25, 2) to [out=down, in=down, out looseness=1.5, in looseness=3] (0.5, 4);
\draw[braid, dinat] (0,0) to [out=up, in=up, in looseness=1, out looseness=3] (-0.75, 2);
\draw[braid, dinat] (0.5,0) to [out=up, in=up, in looseness=1.5, out looseness=3] (-1.25, 2);
\draw[string] (0.25, 4.3) to (0.25, 5.5);
\node[box=1cm] at (0.25,4.2) {$\iota_{{}^{**}b}$};
\node[label, below] at (0, 0) {$b$};
\node[label, above] at (0.25, 5.5){$L$};
\node[label, below] at (0.5, 0) {$b^*$};
\end{tz}
\end{equation}
\end{lemma}
Note that the full twist of~\eqref{eq:fulltwist} is not an endomorphism of $b\otimes b^*$, but rather an isomorphism $b\otimes b^* \to {}^{**}b \otimes {}^*b$. However, after composing with $\iota_{{}^{**}b}: {}^{**}b \otimes {}^*b \to L $, it defines a dinatural transformation $\{b\otimes b^* \to L\}_{b \in \cB}$ and hence induces an endomorphism $L\to L$. 
\begin{proof}Starting from Definition~\ref{def:hbtwist}, we compute $\theta_{L,\lambda} \cdot \iota_b : b\otimes b^* \to L$ as follows: 
\[
\begin{tz}[std,scale=0.8]
\clip (-1.5, -3.75) rectangle (1.5, 2.5);
\draw[braid] (-1, 2) to [out=-60,in=135] (0,1) to [out=-45, in=up]  (1,0) to [out=down, in=down, looseness=2 ] ( 0.25,0);
\draw[braid =main] (0.25,0) to [out=up, in=-135] (1,1);
\draw[string] (-1, -1.5) to [out =up, in=-135] (0,1) to [out=45, in=45, out looseness=3,in looseness=0.5] (1,1);
\node[smalldot] at(0,1){};
\draw[string, dinat] (-1.25, -1.5) to ++(0, -1.5);
\draw[string, dinat] (-0.75, -1.5) to ++(0, -1.5);
\node[box=0.7cm] at (-1,-1.5) {$\iota_{b}$};
\node[label, below] at (-1.25, -3) {$b$};
\node[label, below] at (-0.75, -3) {$b^*$};
\node[label, above] at (-1,2){$L$};
%\node[label, below] at (-1,-1){$X$};
\end{tz}
\quad= \quad
\begin{tz}[std,scale=0.8]
\clip (-2.2, -3.75) rectangle (2.5, 2.5);
\draw[string, dinat] (2,-2) to [out=45, in=-45] (0.1,0.1);
\draw[string, dinat] (1.75,-1.75) to [out=45, in=-45, out looseness=0.75] (-0.1,-0.1);
\draw[braid, dinat] (-0.5,0.5) to [out=45, in=-135, looseness=2, out looseness=3, in looseness=3] (2,-2);
\draw[braid, dinat] (-0.75,0.75) to [out=45, in=-135, looseness=2, out looseness=4, in looseness=2] (1.75,-1.75);
\draw[string, dinat] (-1.5, -3) to [out=up, in=-135] (-0.5, 0.5) ;
\draw[string, dinat] (-1.85, -3) to [out=up, in=-135] (-0.75, 0.75) ;
\draw[string] (0,0) to (-2, 2);
\node[smalldot] at (-0.5, 0.5) {};
\node[smalldot] at (-0.75, 0.75){};
\node[box=0.7cm, rotate=45] at (0,0) {$\iota_{b}$};
\node[label, below] at (-1.5, -3) {$b^*$};
\node[label, below] at (-1.85, -3) {$b$};
\node[label, above] at (-2,2){$L$};
\end{tz}
\quad= \quad
\begin{tz}[std,scale=0.8]
%\clip (-2.2, -3) rectangle (2.5, 2);
\draw[string, dinat] (0.5, -2) to (0.5, 1.5);
\draw[string, dinat] (1, -2) to (1, 1.5);
\draw[string, dinat] (1.5,1.5) to [out=down, in=0] (2.25, -1);
\draw[string, dinat] (2,1.5) to [out=down, in=0, in looseness=2] (2.25, -1.5);
\draw[ string] (1.75, 1.7) to (1.75, 3);
\draw[ braid, dinat] (2.25, -1) to [out=left, in=down] (3, 1.5);
\draw[braid, dinat] (2.25, -1.5) to [out=left, in=down, out looseness=2] (2.5, 1.5);
\node[box=2.8cm] at (1.75,1.6) {$\iota_{b\otimes b^* \otimes b}$};
\node[label, below] at (0.5, -2) {$b$};
\node[label, below] at (1, -2) {$b^*$};
\node[label, above] at (1.75,3){$L$};
\end{tz}
=\quad
\begin{tz}[std, scale=1, scale=0.8, yscale=0.9]
\draw[string, dinat] (-0.75, 2) to [out=down, in=down, out looseness=1, in looseness=3] (0, 4);
\draw[string, dinat] (-1.25, 2) to [out=down, in=down, out looseness=1.5, in looseness=3] (0.5, 4);
\draw[braid, dinat] (0,0) to [out=up, in=up, in looseness=1, out looseness=3] (-0.75, 2);
\draw[braid, dinat] (0.5,0) to [out=up, in=up, in looseness=1.5, out looseness=3] (-1.25, 2);
\draw[string] (1.25, 4.3) to (1.25, 5.5);
\draw[string, dinat] (1,4 ) to[out=down, in=down, looseness=4] (2.5, 4);
\draw[string, dinat] (1.5, 4) to [out=down, in=down, looseness=6] (2, 4);
\node[box=2.8cm] at (1.25,4.2) {$\iota_{{}^{**}b \otimes {}^*b \otimes b}$};
\node[label, below] at (0, 0) {$b$};
\node[label, below] at (0.5, 0) {$b^*$};
\node[label, above] at (1.25, 5.5){$L$};
\end{tz}
\]
The first equation is naturality of the half-braiding $\lambda$ on $L$ (denoted as a black square) and of the braiding of $\cB$, the second equation follows from unpacking the definition of the half-braiding $\lambda$ in terms of a dinatural transformation (see Definition~\ref{defn.lagHBA}), and the last equation uses dinaturality of $\iota:(-)\otimes (-)^{*} \to L$ to move the twist-on-two-strands over to the left. Lastly, we may use dinaturality with respect to $\coev_b = (\ev_{{}^*b})^* :I \to b^* \otimes b$ to arrive at the expression~\eqref{eq:fulltwist}.\end{proof}

Recall from Remark~\ref{rem:Serre} that the half-braided twist is an algebra isomorphism $\theta_{L,\lambda} : (L^\op, \lambda^\op) \to ({}^\op L, {}^\op \lambda)$  representing the Serre equivalence in $\cZ(\Sigma \cB)$ (here $(L^\op, \lambda^\op)$ and $({}^\op L, {}^\op \lambda)$ are the left and right opposite half-braided algebras as defined before Lemma~\ref{lem:dualityHBA}). As our half-braided twist on $(L,\lambda)$ looks like a ``full twist on two strands,'' it admits a canonical ``square root'' resulting in a self-duality $\cL \to {}^\op \cL$ of the Lagrangian object $\cL$ in $\cZ(\Sigma \cB)$ as follows. 

 \begin{prop}\label{prop:defineh} The following morphism $\psi:L \to L$ in $\cB$ defines a half-braided algebra isomorphism $ \psi: (L, \lambda) \to ({}^\op L, {}^\op \lambda) $:
\[ \begin{tz}[std]
\draw[string, dinat] (0.5,0) to [out=up, in=down] (1,1.5);
\draw[braid, dinat] (0, 1)  to [out=down, in=down, looseness=2] (0.5, 1.5);
\draw[braid, dinat] (1,0) to [out=up, in=up] (-0.0,1.);
\draw[string] (0.75, 1.8) to (0.75, 2.5);
\node[box=0.9cm] at (0.75, 1.75) {$\iota_{{{}^*b}}$};
\node[label, below] at (0.5, 0) {$b$};
\node[label, below] at (1,0) {$b^*$};
\node[label, above] at (0.75, 2.5) {$L$};
% [out=up, in=up] (-0.5, 1.5) to [out=down, in=down] (0, 2);
\end{tz}
\]
\end{prop}
\begin{proof} Unitality of $\psi$ is immediate. It follows from dinaturality of $\iota_x: x\otimes x^* \to L$  that $\iota_{c\otimes b} \cdot (\beta_{b,c} \otimes \id_{b^* \otimes c^*}) = \iota_{b\otimes c} \cdot (\id_{b \otimes c} \otimes \beta_{b^*, c^*})$ and hence that the multiplication $\mu_L$ can be rewritten as:
\begin{equation}
\label{eq:twomult}
\begin{tz}[std]
\clip (-0.5, -0.5) rectangle (2.5, 3);
\draw[string, dinat] (0,0) to (0,1.5);
\draw[string, dinat] (0.5, 0) to [out=up, in=down] (2,1.5);
\draw[braid, dinat] (1.5,0) to [out=up, in=down] (0.5,1.5);
\draw[braid, dinat] (2,0) to [out=up, in=down] (1.5,1.5);
\draw[string] (1, 1.7) to (1,2.5);
\node[box=2.4cm] at (1, 1.7) {$\iota_{b\otimes c}$};
\node[label, below] at (0,0) {$b$};
\node[label, below] at (0.5,0) {$b^*$};
\node[label, below] at (1.5, 0) {$c\vphantom{c^*}$};
\node[label, below] at (2,0) {$c^*$};
\node[label, above] at (1,2.5) {$L$};
\end{tz}
=
\begin{tz}[std, xscale=-1]
\clip (-0.5, -0.5) rectangle (2.5, 3);
\draw[braid, dinat] (1.5,0) to [out=up, in=down] (0.5,1.5);
\draw[braid, dinat] (2,0) to [out=up, in=down] (1.5,1.5);
\draw[braid, dinat] (0,0) to (0,1.5);
\draw[braid, dinat] (0.5, 0) to [out=up, in=down] (2,1.5);
\draw[string] (1, 1.7) to (1,2.5);
\node[box=2.4cm] at (1, 1.7) {$\iota_{c\otimes b}$};
\node[label, below] at (0,0) {$c^*$};
\node[label, below] at (0.5,0) {$c\vphantom{c^*}$};
\node[label, below] at (1.5, 0) {$b^*$};
\node[label, below] at (2,0) {$b$};
\node[label, above] at (1,2.5) {$L$};
\end{tz}
\end{equation}
Therefore, isotopy of the following two tangles implies that $\psi:L\to {}^\op L$ is indeed an algebra isomorphism, i.e.\ fulfills $\mu_L \cdot  \beta_{L, L} \cdot (\psi\otimes \psi) = \psi\cdot \mu_L $.
\def\tcut{3}
\[
\begin{tz}[std]
\draw[string, dinat] (0,0) to [out=up, in=down] (0.5,1.5);
\draw[braid, dinat] (-0.3, 1)  to [out=down, in=down, looseness=2] (0, 1.5);
\draw[braid, dinat] (0.5,0) to [out=up, in=up] (-0.3,1.);
\draw[string, dinat] (1.5,0) to [out=up, in=down] (2,1.5);
\draw[braid, dinat] (1.2, 1)  to [out=down, in=down, looseness=2] (1.5, 1.5);
\draw[braid, dinat] (2,0) to [out=up, in=up] (1.2,1.);
%Middle part
\draw[string, dinat] (0,1.5) to [out=up, in=down] (1.5,\tcut);
\draw[string, dinat] (0.5, 1.5) to [out=up, in=down] (2, \tcut);
\draw[braid, dinat] (1.5, 1.5) to [out=up, in=down] (0, \tcut);
\draw[braid, dinat] (2, 1.5) to [out=up, in=down, out looseness=1.3, in looseness=0.5] (0.5, \tcut);
%top part
\draw[string, dinat] (0, \tcut) to [out=up, in=down] (0.5, 4);
\draw[string, dinat] (0.5, \tcut) to [out=up, in=down] (1.5, 4);
\draw[braid, dinat] (1.5, \tcut) to [out=up, in=down] (0, 4);
\draw[braid, dinat] (2, \tcut) to [out=up, in=down] (2, 4);
\draw [decorate,line width =0.5pt,decoration={brace,amplitude=3pt},xshift=-4pt,yshift=0pt]
(-0.5,0) -- (-0.5,1.4) node [black,label, midway,left, xshift=-0.4cm] 
{$\psi \otimes \psi$};
\draw [decorate,line width =0.5pt,decoration={brace,amplitude=3pt},xshift=-4pt,yshift=0pt]
(-0.5,1.6) -- (-0.5,2.9) node [black,label, midway,left, xshift=-0.4cm] 
{$\beta_{L,L}$};
\draw [decorate,line width =0.5pt,decoration={brace,amplitude=3pt},xshift=-4pt,yshift=0pt]
(-0.5,3.1) -- (-0.5,4) node [black,label, midway,left, xshift=-0.4cm] 
{$\mu_L$};
\node[label, below] at(0,0) {$b$};
\node[label, below] at(0.5,0) {$b^*$};
\node[label, below] at(1.5,0) {$c\vphantom{b^*}$};
\node[label, below] at(2,0) {$c^*$};
\draw[string, dinat] (0,4) to ++(0, 0.5);
\draw[string, dinat] (0.5,4) to ++(0, 0.5);
\draw[string,dinat] (1.5,4) to ++(0, 0.5);
\draw[string, dinat] (2,4) to ++(0, 0.5);
\draw[string] (1, 4.5) to (1,5.5) ;
\node[label, above] at (1,5.5) {$L$};
\node[box=2.4cm] at (1, 4.5) {$\iota_{{}^*b\otimes {}^* c}$};
\end{tz}
\qquad = \qquad
\begin{tz}[std]
%\clip (-0.5, -0.5) rectangle (2.5, 3);
\draw[string, dinat] (0,0) to [out=up, in=down] (0.5,1.5);
\draw[string, dinat] (0,1.5) to [out=up, in=down] (1.5, 3.5) to (1.5, 4);
\draw[string, dinat] (0.5, 0) to [out=up, in=down] (1.5,1.5);
\draw[braid, dinat] (1.5,0) to [out=up, in=down] (0, 1.5);
\draw[braid, dinat] (2,0) to (2,1.5);
\draw[braid, dinat] (0.5,1.5) to [out=up, in=down] (2, 3.5) to (2,4);
\draw[string, dinat] (-0.5, 3) to [out=down, in=down, looseness=2] (0.5, 3.5) to (0.5,4);
\draw[string, dinat] (-0.25, 2.9) to [out=down, in=down, in looseness=2] (0.25, 3.5) to [out=up, in=down] (0, 4);
\draw[braid, dinat] (2,1.5)  to [out=up, in=up] (-0.5, 3);
\draw[braid, dinat] (1.5, 1.5) to [out=up, in=up] (-0.25, 2.9);
%\draw[string] (1, 1.7) to (1,2.5);
%\node[box=2.4cm] at (1, 1.7) {$i_{b\otimes c}$};
\node[label, below] at (0,0) {$b$};
\node[label, below] at (0.5,0) {$b^*$};
\node[label, below] at (1.5, 0) {$c\vphantom{c^*}$};
\node[label, below] at (2,0) {$c^*$};
%\node[label, above] at (1,2.5) {$L$};
\draw [decorate,line width =0.5pt,decoration={brace,amplitude=3pt,mirror, raise=4pt},xshift=-4pt,yshift=0pt]
(2.5,0) -- (2.5,1.4) node [black, midway, label, right,xshift=0.6cm] 
{$\mu_L$};
\draw [decorate,line width =0.5pt,decoration={brace,amplitude=3pt,mirror, raise=4pt},xshift=-4pt,yshift=0pt]
(2.5,1.6) -- (2.5,4) node [black, midway, label, right,xshift=0.6cm] 
{$\psi$};
\node[label, below] at(0,0) {$b$};
\node[label, below] at(0.5,0) {$b^*$};
\node[label, below] at(1.5,0) {$c\vphantom{b^*}$};
\node[label, below] at(2,0) {$c^*$};
\draw[string, dinat] (0,4) to ++(0, 0.5);
\draw[string, dinat] (0.5,4) to ++(0, 0.5);
\draw[string, dinat] (1.5,4) to ++(0, 0.5);
\draw[string, dinat] (2,4) to ++(0, 0.5);
\draw[string] (1, 4.5) to (1,5.5) ;
\node[label, above] at (1,5.5) {$L$};
\node[box=2.4cm] at (1, 4.5) {$\iota_{{}^*b\otimes {}^* c}$};
\end{tz}
\]
A similar computation shows that $\psi$ is also compatible with the half-braiding, i.e.\ $\lambda_x \cdot (\id_x \otimes \psi) = (\psi \otimes \id_x) \cdot {}^\op \lambda_x$ for all $x\in \cB$. 
\end{proof}

In Proposition~\ref{prop:etaiscenter}, we defined for any half-braided algebra $(A,\gamma)$ a map $A\to \eta(A,\gamma)$ identifying the object $\eta(A,\gamma) \in \cZ_{2}(\cB)$ with the coequalizer of $\mu_A\cdot (\theta_{A,\gamma}\otimes \id_A)$ and $\mu_A \cdot \beta_{A,A}^{-1}$.  In particular, the induced map $\underline{\eta}(A, \gamma) :=\Hom_{\cB}(\eta(A,\gamma), I) \to \Hom_{\cB}(A, I)$ identifies $\underline{\eta}(A,\gamma)$ with the subspace of $\Hom_{\cB}(A, I)$ of twisted traces on $A$ (see Definition~\ref{def:twistedcocenter}). 

\begin{lemma}\label{lem:everythingistwistedtrace} For the half-braided algebra $(L, \lambda)$, any morphism $\epsilon: L \to I$ is a twisted trace. 
In other words, the inclusion $\underline{\eta}(A, \gamma) \to \Hom_{\cB}(L, I)$ is an isomorphism. 
\end{lemma}
\begin{proof}
Using the two expressions for $\mu_L$ from equation~\eqref{eq:twomult}, it follows that $\mu_L \cdot \beta_{L,L}^{-1} $ is represented by the following dinatural transformation
\begin{equation}\label{eq:twomultop}
\begin{tz}[std, xscale=-1]
\clip (-0.5, -0.5) rectangle (2.5, 3);
\draw[braid, dinat] (0,0) to (0,1.5);
\draw[braid, dinat] (0.5, 0) to [out=up, in=down] (2,1.5);
\draw[braid, dinat] (1.5,0) to [out=up, in=down] (0.5,1.5);
\draw[braid, dinat] (2,0) to [out=up, in=down] (1.5,1.5);
\draw[string] (1, 1.7) to (1,2.5);
\node[box=2.4cm] at (1, 1.7) {$\iota_{c\otimes b}$};
\node[label, below] at (0,0) {$c^*$};
\node[label, below] at (0.5,0) {$c\vphantom{c^*}$};
\node[label, below] at (1.5, 0) {$b^*$};
\node[label, below] at (2,0) {$b$};
\node[label, above] at (1,2.5) {$L$};
\end{tz}
=
\begin{tz}[std]
\clip (-0.5, -0.5) rectangle (2.5, 3);
\draw[braid, dinat] (1.5,0) to [out=up, in=down] (0.5,1.5);
\draw[braid, dinat] (2,0) to [out=up, in=down] (1.5,1.5);
\draw[braid, dinat] (0,0) to (0,1.5);
\draw[braid, dinat] (0.5, 0) to [out=up, in=down] (2,1.5);
\draw[string] (1, 1.7) to (1,2.5);
\node[box=2.4cm] at (1, 1.7) {$\iota_{b\otimes c}$};
\node[label, below] at (0,0) {$b$};
\node[label, below] at (0.5,0) {$b^*$};
\node[label, below] at (1.5, 0) {$c\vphantom{c^*}$};
\node[label, below] at (2,0) {$c^*$};
\node[label, above] at (1,2.5) {$L$};
\end{tz}
\end{equation}
Let $\epsilon: L \to I$ be a morphism represented by a dinatural transformation $\{\epsilon_b: b\otimes b^* \to I\}_{b \in \cB}$. 
Using the second expression from~\eqref{eq:twomult} for $\mu_L$ and the expression~\eqref{eq:fulltwist} for $\theta_{L,\lambda}$, the morphism $\epsilon \cdot \mu_L\cdot (\theta_{L,\lambda} \otimes \id_L): L \otimes L \to I$ is represented by the following dinatural transformation $(b\otimes b^*) \otimes (c\otimes c^*) \to I$. 
\[
\begin{tz}[std, scale=1,scale=0.8]
\clip (-2, -0.5) rectangle (2.5,7);
\draw[string, dinat] (-0.75, 2) to [out=down, in=down, out looseness=1, in looseness=3] (0, 5);
\draw[string, dinat] (-1.25, 2) to [out=down, in=down, out looseness=2, in looseness=1] (1.5, 5);
\draw[braid, dinat] (0,0) to [out=up, in=up, in looseness=1, out looseness=3] (-0.75, 2);
\draw[braid, dinat] (0.5,0) to [out=up, in=up, in looseness=1.5, out looseness=3] (-1.25, 2);
\draw[braid, dinat] (1.5,  0) to [out=up, in=down](0.5, 5);
\draw[braid, dinat] (2,0) to [out=up, in=down] (1, 5);
\node[label, below] at (0, 0) {$b$};
\node[label, below] at (0.5, 0) {$b^*$};
\node[box=1.8cm] at (0.75,5.1) {$\vphantom{i}\epsilon_{{}^{**}b\otimes c}$};
\node[label, below] at (1.5, 0) {$c\vphantom{b^*}$};
\node[label, below] at (2, 0) {$c^*$};
\end{tz}
\quad = \quad 
\begin{tz}[std, scale=1,scale=0.8]
\clip (-2, -0.5) rectangle (3.5,7);
\draw[string, dinat] (-1.5,5) to [out=down, in=down, looseness=2, out looseness=6] (1.5, 5); 
\draw[string, dinat] (-0.75, 3) to [out=down, in=down, out looseness=1, in looseness=3] (0, 5);
%\draw[string, dinat] (-1.25, 2) to [out=down, in=down, out looseness=2, in looseness=1] (1.5, 5);
\draw[braid, dinat] (0,0) to (0,1) to [out=up, in=up, in looseness=1, out looseness=3] (-0.75, 3);
\draw[braid, dinat] (1.5,  0) to [out=up, in=down](0.5, 5);
\draw[braid, dinat] (2,0) to [out=up, in=down] (1, 5);
\draw[braid, dinat] (0.5,0) to [out=up, in=down] (3, 5) to [out=up, in=up, looseness=1.25] (-1.5,5);
\node[box=1.8cm] at (0.75,5.1) {$\vphantom{i}\epsilon_{{}^{**}b\otimes c}$};
\node[label, below] at (0, 0) {$b$};
\node[label, below] at (0.5, 0) {$b^*$};
\node[label, below] at (1.5, 0) {$c\vphantom{b^*}$};
\node[label, below] at (2, 0) {$c^*$};
\end{tz}
\quad = \quad 
\begin{tz}[std, scale=1,scale=0.8]
\clip (-1., -0.5) rectangle (3,7);
\draw[string, dinat] (-0.75, 3) to [out=down, in=down, out looseness=1, in looseness=3] (0, 5);
%\draw[string, dinat] (-1.25, 2) to [out=down, in=down, out looseness=2, in looseness=1] (1.5, 5);
\draw[braid, dinat] (0,0) to (0,1) to [out=up, in=up, in looseness=1, out looseness=3] (-0.75, 3);
\draw[braid, dinat] (1.5,  0) to [out=up, in=down](0.5, 5);
\draw[braid, dinat] (2,0) to [out=up, in=down] (1, 5);
\draw[braid, dinat] (0.5,0) to [out=up, in=down] (2.5, 2.5) to [out=up, in=up, looseness=3] (2, 3);
\draw[braid, dinat] (2,3) to [out=down, in=down, looseness=3] (2.5, 3.5) to [out=up, in=down] (1.5, 5);% to [out=down, in=down] (2,5); %(3, 5) to [out=up, in=up, looseness=1.25] (-1.5,5);
\node[box=1.8cm] at (0.75,5.1) {$\vphantom{i}\epsilon_{{}^{**}b\otimes c}$};
\node[label, below] at (0, 0) {$b$};
\node[label, below] at (0.5, 0) {$b^*$};
\node[label, below] at (1.5, 0) {$c\vphantom{b^*}$};
\node[label, below] at (2, 0) {$c^*$};
\end{tz}
\]
Using dinaturality of $\epsilon$ to cancel the twists results in the second expression~\eqref{eq:twomultop} for $m_L \cdot \beta_{L,L}^{-1}$ proving that $\epsilon:L \to I$ is a twisted trace. \end{proof}

By the universal property of the coend $L= \int^{b\in \cB} b \otimes b^*$, the vector space $\Hom_{\cB}(L, I)$ may be identified with the vector space  $\End(\id_{\cB})$ of (not necessarily monoidal) natural endomorphisms of the identity functor of $\cB$; any natural transformation $\{\alpha_b: b \to b\}_{b \in \cB}$ defines a dinatural transformation $\{\ev_b \cdot (\alpha_b \otimes \id_{b^*}): b \otimes b^*\to I\}_{b \in \cB}$ and vice versa.  By semisimplicity, this vector space may further be identified with the linear dual $\left(K_0(\cB) \otimes_{\mathbb{Z}} \bk \right)^*$ of the Grothendieck ring of $\cB$; any natural transformation $\{\alpha_b: b \to b\}_{b \in \cB}$ defines a linear function $K_0(\cB) \otimes_{\mathbb{Z}}\bk \to \bk$ mapping a simple object $b_i$ to the proportionality factor $\langle \alpha_{b_i} \rangle$ such that $\alpha_{b_i} = \langle \alpha_{b_i} \rangle \id_{b_i}$, and conversely any such linear function determines a natural transformation with these coefficients at simple objects.

Hence, it follows from Lemma~\ref{lem:everythingistwistedtrace} that the vector space  $\underline{\eta}(L,\lambda) = \Hom(\eta(L,\lambda), I) $ is itself identified with $ \Hom_{\cB}(L, I) \cong \End(\id_{\cB}) \cong \left(K_0(\cB)\otimes_{\mathbb{Z}}\bk\right)^*$.

\begin{cor} \label{cor:kappaforL}
Under the identification $\underline{\eta}(\cL) \cong  \Hom_{\cB}(L, I) \cong \End(\id_{\cB})$
the automorphism \[
 \underline{\eta}(\cL)\overset{\underline\fl_{\cL}}\longto  \underline{\eta}({}^\vee \cL) \overset{\underline\eta(\psi_*)}\longto \underline{\eta}(\cL)\] becomes the following automorphism on the vector space $\End(\id_{\cB})$ of natural endomorphisms of $\id_{\cB}$:
 \[ \{ \alpha_b :b \to b\}_{b \in \cB}~~ \mapsto~~ \{ (\alpha_{{}^* b})^*: b \to b \}_{b \in \cB} 
 \]
 Here, $(\alpha_{{}^*b})^*:= (\id_b \otimes \ev_{{}^*b})\cdot (\id_b \otimes \alpha_{{}^*b} \otimes \id_b) \cdot(\coev_{{}^*b} \otimes \id_b)$ is the mate of $\alpha_{{}^*b}:{}^*b \to {}^*b$.

 Further identifying $\End(\id_{\cB})$ with the linear dual $(K_0(\cB) \otimes _\mathbb{Z} \bk)^*$, this automorphism is the linear dual of the automorphism  
\[K_0(\cB) \otimes_{\mathbb{Z}}\bk \to K_0(\cB)\otimes_{\mathbb{Z}} \bk \hspace{1cm} b \mapsto {}^*b
\]
\end{cor}
\begin{proof}
By Corollary~\ref{cor:mainHBA}, the automorphism $\underline{\eta}(\psi_*) \cdot \underline{\fl}$ may be identified with the automorphism on the vector space of twisted traces mapping a twisted trace $\epsilon: L \to I$ to the twisted trace $\epsilon \cdot \psi$. 
By Lemma~\ref{lem:everythingistwistedtrace}, any morphism $\epsilon: L \to I$ is a twisted trace for $(L, \lambda)$ and so it suffices to compute the action of pre-composition with $\psi:L \to L$ from Proposition~\ref{prop:defineh} on $\Hom_{\cB}(L, I)$. 
Identifying $\Hom_{\cB}(L, I)$ with the vector space $\End(\id_{\cB})$ of (non-monoidal) natural transformations of the identity functor, this precomposition with $\psi$ maps a natural transformation $\{\alpha_b: b\to b\}_{b \in \cB}$ (corresponding to the dinatural transformation $\{\coev_b \cdot (\alpha_b \otimes \id_{b^*}):b \otimes b^* \to I \}_{b \in \cB}$) to the natural transformation corresponding to the following dinatural transformation

\[\begin{tz}[std]
\clip (-0.3, -0.5) rectangle (1.6, 3);
\draw[string, dinat] (0.5,0) to [out=up, in=down] (1,1.5);
\draw[braid, dinat] (0, 1)  to [out=down, in=down, looseness=2] (0.5, 1.5);
\draw[braid, dinat] (1,0) to [out=up, in=up] (-0.0,1.);
\draw[string, dinat] (0.5, 1.8) to (0.5, 2) to [out=up, in=up, looseness=2] (1,2) to (1,1.5);
\node[box=0.5cm] at (0.5, 1.73) {$\alpha_{{}^*b}$};
%\node[box=0.9cm] at (0.75, 1.75) {$i_{{{}^*b}}$};
\node[label, below] at (0.5, 0) {$b$};
\node[label, below] at (1,0) {$b^*$};
% [out=up, in=up] (-0.5, 1.5) to [out=down, in=down] (0, 2);
\end{tz}
~~=~~
\begin{tz}[std]
\clip (-0.3, -0.5) rectangle (1.6, 3);
\draw[string, dinat] (0.5,0) to [out=up, in=down] (1,1.5);
\draw[braid, dinat] (1,0) to [out=up, in=down] (1.5, 1.5) to (1.5, 2) to [out=up, in=up, looseness=2] (0, 2) to (0, 1.4) to [out=down, in=down, looseness=2] (0.5, 1.4) to (0.5, 1.5) ;% (-0.0,1.);
\draw[string, dinat] (0.5, 1.8) to (0.5, 2) to [out=up, in=up, looseness=2] (1,2) to (1,1.5);
\node[box=0.5cm] at (0.5, 1.73) {$\alpha_{{}^*b}$};
%\node[box=0.9cm] at (0.75, 1.75) {$i_{{{}^*b}}$};
\node[label, below] at (0.5, 0) {$b$};
\node[label, below] at (1,0) {$b^*$};
% [out=up, in=up] (-0.5, 1.5) to [out=down, in=down] (0, 2);
\end{tz}
~~=~~ \coev_{b} \cdot \left(( \alpha_{{}^*b})^* \otimes \id_{b^*}\right)
\]
which indeed corresponds to the natural transformation $(\alpha_{{}^*b})^*$. 

After identifying $\End(\id_{\cB})$ with the linear dual $(K_0(\cB) \otimes_{\mathbb{Z}}\bk)^*$ the above assignment maps a linear function $K_0(\cB)\ni b \mapsto \lambda_b$ to the function $b \mapsto \lambda_{{}^*b}$ (as $\langle (\alpha_{{}^*b})^* \rangle = \langle \alpha_{{}^*b}\rangle$).
\end{proof}

\subsection{The last step}
\label{subsec:laststep}

From now on, assume that $\cB$ is a slightly degenerate braided fusion category with $\cZ_{2}(\cB) \cong \sVec$. By Theorem~\ref{thm.BCs}, to prove our \MainTheorem, it suffices to show that the braided monoidal equivalence class of $\cZ(\Sigma\cB)$ is independent of $\cB$: indeed, $\cZ_{2}(\cB) \cong \sVec$ is a particular example of a slightly degenerate braided fusion category, and so we would find that $\cZ(\Sigma\cB)$ and $\cZ(\Sigma\sVec)$ are equivalent. By Theorem~\ref{thm:SandT}, $\cZ(\Sigma\cB)$ is equivalent to one of $\cS$ or $\cT$. We will show that it is always equivalent to~$\cS$.

As explained at the start of \S\ref{subsec:SandT},
$\cZ(\Sigma \cB)$ has precisely two components: the identity component  and the other, magnetic, component. In particular, any object $X$ in $\cZ(\Sigma \cB)$ uniquely decomposes into a direct sum $X_+ \boxplus X_-$ where $X_+$ is in the identity component and $X_-$ is in the magnetic component. As the Lagrangian object $\cL$ generates $\cZ(\Sigma \cB)$, both its summands $\cL_+$ and $\cL_-$ are non-zero and so, by Proposition~\ref{prop:fundamental},
to prove that $\cZ(\Sigma \cB) \cong \cS$,
 it suffices
 to show that $\underline{\Inv}(\cL_-)>0 $.  

As explained in Remark~\ref{rem:Rinvertible}, an object $Y \in \cZ(\Sigma \cB)$ is (a sum of simple objects all contained) in the identity component if and only if its full braiding $\br^{\cZ(\Sigma \cB)}_{e,Y}\cdot \br^{\cZ(\Sigma \cB)}_{Y,e}: \id_Y \boxtimes e \To \id_Y \boxtimes e$ with the non-trivial simple object $e\in \cZ_{2}(\cB) = \sVec$ is trivial. Analogously, it is in the magnetic component if and only if this full braiding is $(-1) \id_{\id_Y \boxtimes e}$.

Unpacked in terms of half-braided algebras, a half-braided algebra $(A,\gamma)$ is in the identity component if and only if $\gamma_{e} = \beta_{e,A}$ and it is in the magnetic component if and only if $\gamma_e =  - \beta_{e, A}$. 
In particular, recall the bimodule endomorphism $\widetilde{R}_{(A, \gamma),e} :{}_AA_A\To {}_AA_A$ from Example~\ref{ex:Rhba} and Remark~\ref{rem:Rinvertible}:  \[\widetilde{R}_{(A, \gamma), e}:=    (\id_A \otimes \ev_e) \cdot ((\gamma_e \cdot\beta_{A, e})\otimes \id_{e})\cdot (\id_A \otimes (\ev_e)^{-1}) ~~ = ~~
\begin{tz}[std]
\clip (-0.75, -0.5) rectangle (1.75, 2.5);
\draw[string] (0,0) to (0,2);
\draw[braid, hb] (1,1) to [out=down, in=-45, looseness=2] (0, 0.5) to [out=135, in=down] (-0.4, 1);
\draw[string, hb, arrow data={0.01}{>}, arrow data={0.8}{>} ] (-0.4,1)  to [out=up, in =-135] (0, 1.5)  to [out=45,  in=up, looseness=2] (1, 1);
\node[smalldot] at (0,1.5) {};
\node[label,left] at (-0.5,0.9) {$e$};
\node[label,below] at (0,-0.) {$A$};
\node[label, above] at (0.65, 1.8) {$\ev_e$};
\node[label, below] at (0.85, 0.2) {$(\ev_e)^{-1}$};
\end{tz}
 \]
As $e$ is invertible, it follows that $\widetilde{R}_{(A, \gamma), e} \cdot \widetilde{R}_{(A, \gamma), e} = \id_A$ (see Remark~\ref{rem:Rinvertible}) and hence that the splitting $(A, \gamma) \cong (A_+, \gamma_+) \boxplus (A_-, \gamma_-)$ of a half-braided algebra is induced by the following idempotents:
\[p_{A, \gamma}^\pm =\frac{1}{2}( \id_A \pm \widetilde{R}_{(A, \gamma), e})
\]
These idempotents are ``central orthogonal'' in the sense that they are $A$--$A$ bimodule morphisms $A\to A$ and fulfill 
\[
\bigl(p^+_{A, \gamma}\bigr)^2 = p^+_{A, \gamma}, \quad 
\bigl(p^-_{A, \gamma}\bigr)^2= p^-_{A, \gamma}, \quad
p_{A, \gamma}^+ p_{A, \gamma}^- = p_{A, \gamma}^- p_{A,\gamma}^+ = 0, \quad
p^+_{A, \gamma} + p^-_{A, \gamma} = \id_{A}.
\]

 \begin{remark}More abstractly, using the $2$-morphism from Remark~\ref{rem:Rinvertible}
 \[\widetilde{R}_{X, e}:= \left(\id_X \boxtimes \ev_e\right) \cdot \left(\left(\br^{\cZ(\Sigma \cB)}_{e,X}\cdot \br^{\cZ(\Sigma \cB)}_{X,e}\right) \boxtimes \id_e\right) \cdot \left(\id_X \boxtimes(\ev_e)^{-1}\right) :  \id_X \To \id_X,\]  the idempotent $2$-morphisms $p^\pm_X = \frac{1}{2}(\id_{\id_X} \pm \widetilde{R}_{X, e}) : \id_X \To \id_X$ can defined for any object $X$ in any model of $\cZ(\Sigma \cB)$ and precisely correspond to the direct sum decomposition $X\cong X_+ \boxplus X_-$. (The correspondence between families of idempotents in $\Omega^2_X\cC$ and direct sum decompositions of $X$ is explained in Remark~\ref{remark.2catDirectSums}.)  
 \end{remark}
 
 For a half-braided algebra $(A,\gamma)$, the idempotents $p^\pm_{A, \gamma}$  descend to the coequalizer~\eqref{eq:coeq} and induce the direct sum decomposition $\eta(A, \gamma) \cong \eta(A_+, \gamma_+) \oplus \eta(A_-, \gamma_-)$. This further induces a decomposition $\underline{\eta}(A, \gamma)  \cong \underline{\eta}(A_+, \gamma_+) \oplus \underline{\eta}(A_-, \gamma_-)$ of the vector space of twisted traces of $A$. Explicitly, the idempotents inducing this decomposition map a twisted trace $\epsilon:A \to I$ to the twisted trace $\epsilon \cdot p^\pm_{A, \gamma}$. For the canonical Lagrangian object $(L, \lambda)$ in $\sHBA(\cB)$, these idempotents on the space of twisted traces may be explicitly computed as follows. 
  \begin{lemma}\label{lem:almostdoneidempotent}
 Identifying the space of twisted traces of $(L, \lambda)$ with $(K_0(\cB)\otimes_\bZ \bk)^*$ as in Corollary~\ref{cor:kappaforL}, the induced idempotents on $(K_0(\cB) \otimes_{\bZ} \bk)^*$ are given by the linear dual of the idempotents 
 \[ K_0(\cB) \otimes_\bZ \bk \to K_0(\cB) \otimes_\bZ\bk, \hspace{1cm} b \mapsto  \left(\frac{1\pm e}{2}\right) b,
 \]
where the multiplication on the right  is in the Grothendieck ring $K_0(\cB) \otimes_\bZ \bk$.
 \end{lemma}
 \begin{proof}Recall from Lemma~\ref{lem:everythingistwistedtrace} that any morphism $\epsilon: L \to I$ is a twisted trace. Identify $\Hom_{\cB}(L, I)$ with the space  $\End(\id_{\cB})$  of natural endomorphisms of the identity functor on $\cB$, and let $\epsilon_b =\coev_b \cdot ( \alpha_b \otimes \id_{b^*}): b\otimes b^*\to I$ be the dinatural transformation associated to a natural endomorphism $\alpha$. Using Definition~\ref{defn.lagHBA} of the half-braiding $\lambda$,  $\epsilon \cdot \widetilde{R}_{(L, \lambda), e} $ unpacks to the dinatural transformation associated to the natural endomorphism  \[\{ \tr_e(\alpha_{e\otimes b}):= (\ev_e \otimes \id_b)\cdot (\id_e \otimes \alpha_{e \otimes b}) \cdot (\ev_e^{-1} \otimes \id_b): b \to b\}_{b \in \cB}.\]
 For simple objects $b$, the proportionality factors  $\langle \tr_e(\alpha_{e\otimes b}) \rangle = \langle \alpha_{e\otimes b}\rangle$ agree, and hence, if we further identify $\End(\id_{\cB})$ with $(K_0(\cB)\otimes_\bZ\bk)^*$, the assignment $\{\alpha_b\}_{b \in \cB} \mapsto \{\tr_e(\alpha_{e\otimes b})\}_{b \in \cB}$ turns into the linear dual of the endomorphism $b\mapsto e\otimes b$ on $K_0(\cB) \otimes_{\bZ} \bk$.   \end{proof}

 Suppose now that $(A,\gamma)$ is a half-braided algebra equipped with an algebra isomorphism $\phi: (A,\gamma) \to ({}^\op A, {}^\op \gamma)$. Observe that $\phi$ restricts to algebra isomorphisms $\phi_\pm: (A_\pm, \gamma_\pm) \to ({}^\op A_\pm, {}^\op \gamma_\pm)$ as $A_+$ and ${}^\op A_+$ are in a different component of $\sHBA(\cB)$ from $A_-$ and ${}^\op A_-$.
 Recall from Corollary~\ref{cor:mainHBA} that our invariant $\underline{\Inv}((A,\gamma), \phi)$ can be computed as the linear trace of the endomorphism of the vector space of twisted traces which maps a twisted trace $\epsilon:A \to I$ to $\epsilon\cdot \phi$. Therefore, the previous discussion can be summarized as follows:
\begin{lemma}\label{lem:Hpm} Given a half-braided algebra $(A,\gamma)$ with a half-braided algebra isomorphism $\phi: (A,\gamma) \to ({}^\op A, {}^\op \gamma)$, 
the invariant $\underline{\Inv}((A_\pm, \gamma_\pm), \phi_\pm)\in \bk$ is the linear trace of the following endomorphism of the vector space of twisted traces on $A$:
\\[6pt]
\mbox{} \hfill $ \displaystyle \epsilon:A\to I \hspace{1cm} \mapsto \hspace{1cm} \epsilon \cdot \phi \cdot p^\pm_{A, \gamma}$ \hfill \qed
\end{lemma}

 \begin{corollary} \label{cor:computingkappaL} Let $\cL_+$ and $\cL_-$ denote the trivial and magnetic summand of the canonical Lagrangian object $\cL$ in $\cZ(\Sigma \cB)$, respectively, and let $\psi_{\pm}: \cL_{\pm} \to {}^\vee \cL_{\pm}$ denote the restrictions of the self-duality $\psi_*:\cL \to {}^\vee \cL$ of $\cL$ to $\cL_{\pm}$.  Then, 
 \[
 \underline{\Inv}( \cL_{\pm}, \psi_\pm) = \frac{1}{2} \# \{b \in \pi_0 \cB~|~ b \cong {}^*b\} >0. 
 \]
 \end{corollary}
 \begin{proof}
Identify the space of twisted traces of the half-braided algebra $(L, \lambda)$ representing $\cL$ with the linear dual $(K_0(\cB) \otimes_{\bZ} \bk)^*$ as in Corollary~\ref{cor:kappaforL}, under which $\epsilon \mapsto \epsilon \cdot \psi$ becomes the linear dual to $b\mapsto {}^* b$.  By Lemma~\ref{lem:almostdoneidempotent}, $\epsilon \mapsto \epsilon \cdot p^{\pm}_{L, l}$ becomes the linear dual to $b\mapsto \frac{1\pm e}{2} b$. Thus we conclude from Lemma~\ref{lem:Hpm} that $\underline{\Inv}(\cL_{\pm}, h_\pm)$ is computed as the linear trace of the endomorphism 
\[K_0(\cB) \otimes_\bZ\bk \to K_0(\cB)\otimes_\bZ \bk\hspace{1cm} b \mapsto \frac{1\pm e}{2} {}^* b.
\]
Explicitly, this trace is
\[
\frac{1}{2} \#\{b\in \pi_0 \cB ~|~ b \cong {}^*b\} \pm \frac{1}{2} \# \{ b \in \pi_0 \cB~|~ b \cong e\otimes {}^* b\}.
\]
The second summand vanishes: a simple object $b$ can never be isomorphic to $e \otimes {}^* b$. This can for example be seen by computing (the $1$-categorical version of) $\eta(-)$ on both sides. As $e$ is transparent, $\eta(e \otimes {}^* b) = \eta(e) \eta({}^* b)$. As further $\eta(b) = \eta({}^*b) \neq 0$ (by the $1$-categorical version of $\fl$) and $\eta(e)=-1$, it follows that $b\not \cong e\otimes {}^*b$. 
\end{proof}

As $\cL_-$ is magnetic and $\underline{\Inv}(\cL_-) >0$, it follows from Proposition~\ref{prop:fundamental} that $\cZ(\Sigma \cB) \cong \cS$ which completes the proof of our \MainTheorem.

\begin{remark}\label{rem:refrequest}
  Corollary~\ref{cor:computingkappaL} does more than complete the proof of our main theorem: translated into the language of module categories, it provides information about the decomposition of $\cZ(\cB)$ into indecomposable $\cB$-module categories. Unpacking the proof of Proposition~\ref{prop:fundamental}, it shows that the number of indecomposable summands, in each of the two submodules $\cZ(\cB)_+$ and $\cZ(\cB)_-$, which are sent to themselves under the self-duality $(-)^*$ of $\cZ(\cB)$, is equal to half of the number of self-dual simple objects in $\cB$. We emphasize that, \emph{a priori}, $\underline\kappa$ reproduces this count only up to a sign (and would not reproduce this count at all if $\cZ(\cB)$ had summands equivalent to $\Mod_\cB(\mathrm{Cliff}(1))$), and it is a consequence of our proof that this sign is $+1$.
  
  This count can also be derived from \cite[Theorem 3.4 and Corollary 3.6]{DGNO}, using arguments similar to the ones we give in \S\ref{sec:Smatrix}.
  Let us write $\pi_0(\cZ(\cB)/\cB)$ for the set of indecomposable $\cB$-module summands of $\cZ(\cB)$. Note that the M\"uger centralizer of $\cB \subset \cZ(\cB)$ is $\cB^\rev$. Applied to this centralizer pair, \cite[Theorem 3.4 and Corollary 3.6]{DGNO} provides a perfect S-matrix pairing between $\pi_0(\cZ(\cB)/\cB)$ and $\pi_0(\cB^\rev) = \pi_0(\cB)$.
  This pairing intertwines the automorphism $b \mapsto b^*$ of $\pi_0 \cB^\rev$ with the automorphism of $\pi_0(\cZ(\cB)/\cB)$ which sends the summand of $\cZ(\cB)$ containing a simple object $x$ to the summand containing $x^*$. (This automorphism of  $\pi_0(\cZ(\cB)/\cB)$ is well-defined: if simple objects $x,y \in \cZ(\cB)$ are in the same $\cB$-module summand, then $x^*$ and $y^*$ are in the same $\cB$-module summand.) The pairing furthermore intertwines the automorphism $b \mapsto e\otimes b$ of $\pi_0 \cB^\rev$ with the automorphism of the vector space spanned by $\pi_0(\cZ(\cB)/\cB)$ which acts diagonally by $+1$ on the summands inside $\cZ(\cB)_+$ and by $-1$ on the summands inside $\cZ(\cB)_-$.
  
  The count of self-dual indecomposable $\cB$-module summands of $\cZ(\cB)_\pm$ follows from comparing traces.
\end{remark}

\section{Outlook: a complete obstruction theory for nondegenerate extensions} \label{subsec.completeobstructions}

This paper provides the ingredients necessary for a complete obstruction theory for minimal nondegenerate extensions, similar to the obstruction theory developed in \cite{DmitriMSRI}. In this section, we outline such an obstruction theory, but will leave out many details whose justification would require a more fully developed higher Morita theory of fusion higher categories and their semisimple higher-categorical modules. In particular, the results in~\S\ref{sec:Tan} and~\S\ref{sec:superTan} which rely on these unproven assumptions will be labelled ``Conjectures''.

\subsection{Obstruction theories via Witt groups of braided fusion $1$-categories}
Fix a symmetric fusion 1-category $\cE$, and let $\NBFC(\cE)$ denote the set of equivalence classes of braided fusion $1$-categories $\cB$ equipped with a symmetric monoidal equivalence $\cE \isom \cZ_2(\cB)$. By \cite[Proposition 4.3]{MR3022755}, $\NBFC(\cE)$ is a commutative monoid under the balanced tensor product $\boxtimes_\cE$. It contains a submonoid of categories of the form $\cZ_\cE(\cF)$, where $\cF$ is a fusion 1-category equipped with a braided monoidal inclusion $\cE \subset \cZ(\cF)$ and $\cZ_\cE(\cF) := \cZ_2(\cE \subset \cZ(\cF))$. The quotient of $\NBFC(\cE)$ by this submonoid is a group $\Witt(\cE)$, defined in~\cite[Definition 5.1, Lemma 5.2]{MR3022755} and known as the \define{$\cE$-Witt group}. By convention, $\Witt := \Witt(\Vec)$.

Let us say that an \define{obstruction theory for minimal nondegenerate extensions} is a commutative monoid $\Obs(\cE)$ together with a monoid homomorphism $O(-) : \NBFC(\cE) \to \Obs(\cE)$ such that $\cB \in \NBFC(\cE)$ admits a minimal nondegenerate extension if and only if $O(\cB)=0$. We will say that an obstruction theory is \define{complete} if additionally $O(-)$ is surjective. 

The following proposition was probably known to the authors of \cite{MR3022755} (see in particular Question 5.15 therein), and is essentially Theorem~3.2 of~\cite{2105.01814}.
  Our proof is essentially the same as the proof used in \cite[Proposition 5.15]{MR3613518} to produce a surjective homomorphism onto $\ker\bigl([-\boxtimes\cE] : \Witt \to \Witt(\cE)\bigr)$ from the group $\Mext(\cE)$ of minimal nondegenerate extensions of $\cE$.

\begin{proposition}\label{prop:Wittobstruction}
A braided fusion category $\cB \in \NBFC(\cE)$ admits a minimal nondegenerate extension if and only if its Witt class $[\cB] \in \Witt(\cE)$ is in the image of the map $[-\boxtimes\cE] : \Witt \to \Witt(\cE)$ constructed in~\cite[Proposition 5.13]{MR3022755}. In particular, $\coker\bigl([-\boxtimes\cE] : \Witt \to \Witt(\cE)\bigr)$, with its canonical map from $\NBFC(\cE)$, is a complete obstruction theory for minimal nondegenerate extensions.
\end{proposition}

\begin{proof}
  To prove the ``only if'' direction, suppose that $\cB \in \NBFC(\cE)$ admits a minimal nondegenerate extension $\cM$. Recall that the Drinfeld centre $\cZ(\cM)$ is equivalent to $\cM \boxtimes \cM^\rev$, and the inclusion of $\cM$ to the first factor is the one coming from the braiding on $\cM$. Consider the inclusion $\cE \subset \cB \subset \cM \subset \cM \boxtimes \cM^\rev$.  Then there are canonical equivalences
  $$ \cZ_\cE(\cM) \cong \cZ_2(\cE \subset \cM \boxtimes \cM^\rev) \cong \cZ_2(\cE \subset \cM) \boxtimes \cM^\rev \cong \cB \boxtimes \cM^\rev \cong \cB \boxtimes_\cE (\cE \boxtimes \cM^\rev).$$
  It follows that $[\cB] = [\cM \boxtimes \cE]$ in $\Witt(\cE)$.

  To prove the ``if'' direction, suppose that $[\cB] = [\cN \boxtimes \cE]$ for some nondegenerate braided fusion category $\cN$. Then there exists a fusion category $\cA$ with an inclusion $\cE \subset \cZ(\cA)$ such that
  \begin{equation*} \label{eqn:characterizationofA}
   \cB \boxtimes \cN^\rev \cong \cB \boxtimes_\cE (\cE \boxtimes \cN^\rev) \cong \cZ_\cE(\cA).
  \end{equation*}
  There is therefore an inclusion of nondegenerate braided fusion categories $\cN^\rev \subset \cB \boxtimes \cN^{\rev} \cong \cZ_\cE(\cA) \subset \cZ(\cA)$. Hence, by  \cite[Theorem 3.13(i)]{DGNO} there is a braided equivalence $\cZ(\cA) \cong \cN^\rev \boxtimes \cM$ for the nondegenerate braided fusion category $\cM:= \cZ_2(\cN^\rev \subset \cZ(\cA))$. But $\cB$ and $\cN^\rev$ commute in $\cZ_\cE(\cA) \subset \cZ(\cA)$, providing an inclusion $\cB \subset \cM$. Counting Frobenius--Perron dimensions confirms that this inclusion is a minimal nondegenerate extension.
\end{proof}

\begin{corollary}
  The complete obstruction theory for minimal nondegenerate extensions is unique: any complete obstruction theory $(\Obs(\cE), O: \NBFC(\cE) \to \Obs(\cE))$ is canonically isomorphic to $\coker\bigl([-\boxtimes\cE] : \Witt \to \Witt(\cE)\bigr)$ with its canonical map from $\NBFC(\cE)$.\end{corollary}
\begin{proof}
  The defining property of a complete obstruction theory is that it is a surjection of commutative monoids with prescribed kernel $K \subseteq M$. In general, prescribing the kernel does not determine (up to unique isomorphism) a surjection of monoids, but it does if there exists a surjection $M \twoheadrightarrow G$ with kernel $K$ and image a group $G$. \end{proof}

\subsection{Obstruction theory via braided fusion $2$-categories}\label{subsec:obstructionviaBF2C}

It remains to actually describe the group $\Obs(\cE) \cong \coker\bigl([-\boxtimes\cE] : \Witt \to \Witt(\cE)\bigr)$ and the homomorphism $O(-)$ in a computationally useful way.

Given $\cB \in \NBFC(\cE)$, we can compare the braided fusion 2-categories $\cZ(\Sigma \cB)$ and $\cZ(\Sigma \cE)$ which are both nondegenerate in the sense of Theorem~\ref{thm:invertibleSmatrix}. Theorem~\ref{thm.BCs} implies that if they are braided equivalent under $\cE$ (i.e.\ via an equivalence which restricts to the canonical equivalence $\Omega\cZ(\Sigma\cB) = \cZ_2(\cB) \cong \cE = \Omega\cZ(\Sigma\cE)$), then $\cB$ admits a minimal nondegenerate extension. As in Remark~\ref{rem:BCs}, we expect that Theorem~\ref{thm.BCs} extends to an equivalence between the $2$-groupoid of equivalences $\cZ(\Sigma \cB) \cong\cZ(\Sigma \cE)$ under~$\cE$ and the $2$-groupoid of minimal nondegenerate extensions of $\cB$. This in particular produces a canonical bijection between the set $\Mext(\cB)$ of isomorphism classes of minimal nondegenerate extensions of $\cB$ and the set of isomorphism classes of braided equivalences $\cZ(\Sigma \cB) \cong\cZ(\Sigma \cE)$ under~$\cE$.

This perspective reveals yet another description of $\Obs(\cE)$.
Let $\mathrm{NBF2C}(\cE)$ denote the set of (braided equivalence classes under $ \cE$ of) nondegenerate braided fusion $2$-categories $\cA$ equipped with a symmetric equivalence $\cE \cong \Omega \cA$. This set has an abelian group structure, defined analogously to the abelian group structure on $\Mext(\cE)$ in~\cite{MR3613518}. 
The previous paragraph implies that the obstruction group $\Obs(\cE)$ is precisely the subgroup of $\mathrm{NBF2C}(\cE)$ on those $\cA$ which happen to be of the form $\cZ(\Sigma \cB)$ for a braided fusion $1$-category $\cB$. The monoid homomorphism $O: \NBFC(\cE) \to \Obs(\cE) \subseteq \mathrm{NBF2C}(\cE)$ takes a braided fusion $1$-category $\cB$ to $\cZ(\Sigma \cB)$.

In other words, the question of minimal nondegenerate extensions reduces to the question of classifying nondegenerate braided fusion 2-categories, a classification of which is outlined in \cite{1704.04221,PhysRevX.9.021005}, see also \cite{2003.06663}. At its heart, this classification relies on a categorification of Section 4.4 of~\cite{DGNO}.  Given any braided fusion 2-category $\cX$, there  is a higher-categorical group $\Inn(\cX)$ of ``inner automorphisms'' of $\cX$, equivalent to the $4$-group (i.e. monoidal $3$-groupoid which is ``groupal'' in the sense that all objects are $\otimes$-invertible) $(\Sigma \cX)^\times$ of invertible $\cX$-module $2$-categories (and their higher isomorphisms). An ``inner action'' of a finite group $G$ on a nondegenerate braided fusion 2-category $\cX$ is a monoidal $3$-functor $\mu : G \to \Inn(\cX)$, and there is a way to ``gauge'' an inner action  $G\to \Inn(\cX)$ to produce a nondegenerate braided fusion 2-category $\cX\quot G$ with a symmetric embedding $\Rep(G) \hookrightarrow \Omega(\cX\quot G)$, analogous to the construction for braided fusion $1$-categories developed in~\cite{MR3555361}. Categorifying Theorem 4.44 of~\cite{DGNO}, we expect that gauging induces a bijection between the set of (equivalence classes of) nondegenerate braided fusion $2$-categories equipped with an inner $G$-action and the set of (equivalence classes under $\Rep(G)$) of nondegenerate braided fusion $2$-categories $\cA$ with a symmetric embedding $\Rep(G) \hookrightarrow \Omega \cA$.
The inverse of gauging is known as ``condensing'' the symmetric fusion category $\Rep(G) \hookrightarrow \Omega \cA$.  

\subsection{The Tannakian case}\label{sec:Tan}
Under this gauging/condensing equivalence,  $\mathrm{NBF2C}(\Rep(G))$ becomes the set of nondegenerate braided fusion $2$-categories $\cX$ with $\Omega \cX \cong \Vec$ and equipped with an inner $G$-action. (The induced abelian group structure on this set is given by the ordinary tensor product $\boxtimes$ with diagonal inner $G$ action.)  By nondegeneracy, any such $\cX$ is equivalent to the unit $2$-category $\Sigma \Vec$ for which the classifying space $\rB \Inn(\Sigma \Vec)$ is a $K(\bk^\times, 4)$. Therefore, an inner action on $\Sigma \Vec$ amounts to a class in $\H^4(\rB G; \bk^\times)$. Hence, $\mathrm{NBF2C}(\Rep(G)) \cong \H^4(\rB G; \bk^\times)$ and it turns out that the map $O: \NBFC(\Rep(G)) \to \mathrm{NBF2C}(\Rep(G)) \cong \H^4(\rB G; \bk^\times)$ is precisely the obstruction class $O_4(\cB)$ constructed in \cite{MR2677836}. We therefore recover the Tannakian part of Theorem~4.8 of \cite{1712.07097};  $O_4$ induces an injection $ \Obs(\Rep(G)) \hookrightarrow \H^4(\rB G; \bk^\times)$.

In fact, this injection is an isomorphism, providing a complete obstruction theory. To see this one must produce for every $\alpha \in \H^4(\rB G; \bk^\times)$ a braided fusion category $\cB$ with $\alpha = O_4(\cB)$. This can be done following ideas in \cite{PhysRevX.8.031048,EvansGannon}. First, find a surjection of finite groups $f : \widetilde{G} \to G$ so that $f^* \alpha$ may be trivialized. The data of such a trivialization, when restricted to $K=\ker(f)$, provides the data of a $K$-graded extension of the fusion category $\Vec$.  Explicitly, a trivialization of $f^*\alpha$ is a 3-cochain $\beta$ on $\widetilde{G}$ solving $\d\beta = f^*\alpha$; restricted to $K$ this gives a 3-cocycle which is used as the associator. The remaining data on the rest of $\widetilde{G}$ provides an action of $G$ on the Drinfeld centre $\cZ(\Vec^\beta[K])$. The categorical fixed points of this $G$-action is the desired braided fusion category with obstruction $\alpha$. 

We have therefore outlined a proof of the following conjecture, subject to the assumptions made above.
\begin{conjecture}The obstruction class $O_4(\cB)$ induces an isomorphism
\begin{equation} \label{eq:obsbos2}
\Obs(\Rep(G)) \cong \coker\bigl(\Witt \to \Witt(\Rep(G))\bigr) \cong \H^4(\rB G; \bk^\times). 
\end{equation}
\end{conjecture}
Since trivializations of a cohomology class form a torsor for the cohomology one degree lower, we also recover the calculation $\Mext(\Rep(G)) \cong \H^3(\rB G; \bk^\times)$ from \cite[Theorem~4.22]{MR3613518}.

\begin{remark}
As a corollary of the isomorphism~\eqref{eq:obsbos2} it follows that all nondegenerate braided fusion 2-categories $\cA$ with $\Omega \cA \cong \Rep(G)$ are of the form $\cZ(\Sigma \cB)$, since $\Obs(\Rep(G)) \subseteq \mathrm{NBF2C}(\Rep(G)) \cong \H^4(\rB G; \bk^\times)$ is by definition the subgroup on those $\cA$ of this form. 
\end{remark}

\subsection{The super Tannakian case}\label{sec:superTan}

By Deligne's theorem on the existence of fibre functors~\cite{MR1944506}, any symmetric fusion $1$-category is either Tannakian, i.e.\ equivalent to $\Rep(G)$ for a finite group $G$, or \define{super-Tannakian}, i.e.\ equivalent to $\sRep(G, \varpi)$ for a finite group $G$ together with a $2$-cocycle $\varpi \in \H^2(\rB G; \bZ_2)$. (The latter data $(G, \varpi)$ is often alternatively encoded in terms of the finite \define{super group} --- a finite group $H$ together with a non-trivial central element $z$ of order two --- given by the central extension of $G$ associated to $\varpi$.) Explicitly, the classifying space of the $2$-group of symmetric monoidal automorphisms of $\sVec$ is $\rB \Aut(\sVec) \cong K(\bZ_2,2)$ and hence the class $\varpi$ encodes a monoidal functor $G\to \Aut(\sVec)$, or equivalently an action of $G$ on $\sVec$. The category $\sRep(G, \varpi)$ arises as the categorical fixed points of this action. 

In the super-Tannakian case, we may condense the (maximal Tannakian) subcategory $\Rep(G)$ inducing an isomorphism between $\mathrm{NBF2C}(\sRep(G,\varpi))$ and the abelian group of nondegenerate braided fusion $2$-categories $\cX$ equipped with a symmetric equivalence $\Omega \cX \cong \sVec$ and an inner $G$-action $G \to  \Inn(\cX)$ together with an identification of the composite $G \to \Inn(\cX) \to \Aut(\cX) \to \Aut(\Omega \cX) \cong \Aut(\sVec)$ with $\varpi$. Equivalently, and somewhat more efficiently, the latter two pieces of data may be described as a lift of $\varpi: \rB G \to K(\bZ_2, 2)$ along $\rB \Inn(\cX) \to K(\bZ_2, 2)$. 

Our Theorem~\ref{thm:SandT} precisely classifies nondegenerate braided fusion $2$-categories $\cX$ with $\Omega \cX \cong \sVec$:
There are exactly two such 2-categories up to braided equivalence, ``$\cS$'' and ``$\cT$,'' with $\cS \cong \cZ(\Sigma\sVec)$. Starting with a braided fusion $1$-category $\cB$ with $\cZ_2(\cB) \cong \sRep(G, \varpi)$ and condensing the maximal Tannakian subcategory $\Rep(G)$ in $\cZ(\Sigma \cB)$ results in the braided fusion $2$-category $\cZ(\Sigma \cB_G)$, where $\cB_G$ is the slightly degenerate braided fusion 1-category from~\cite[Theorem~4.8]{1712.07097} of the same name. The heart of the proof of our \MainTheorem\ shows that $\cZ(\Sigma\cB_G)$ is (noncanonically) equivalent to $\cS$, and that such equivalences correspond to minimal nondegenerate extensions of the slightly degenerate braided fusion category $\cB_G$.
Hence, any braided fusion $2$-category $\cZ(\Sigma \cB)$ with $\Omega \cZ(\Sigma \cB) \cong \cZ_2(\cB) \cong \sRep(G, \varpi)$ must arise from gauging an inner $G$-action on $\cS$ and $\Obs(\sRep(G, \varpi)) $ must be contained in the (proper) subgroup of $\mathrm{NBF2C}(\sRep(G, \varpi))$ on those nondegenerate braided fusion $2$-categories $\cX$ (with $\Omega \cX \cong \sVec$ and lift of $\varpi: \rB G \to K(\bZ_2,2)$ along $\rB \Inn(\cX) \to K(\bZ_2, 2)$) which \emph{happen to be equivalent to $\cS$}.

This subgroup of $\mathrm{NBF2C}(\sRep(G, \varpi))$ is precisely the cokernel of the group homomorphism from the group $\pi_0 \Aut_{\sVec}(\cS) \cong \Mext(\sVec) \cong \bZ_{16}$ of equivalence classes of braided automorphisms of $\cS$ over $\sVec$ (see the second paragraph of \S\ref{subsec:obstructionviaBF2C})  to the group of lifts of $\rB G \to K(\bZ_2, 2) $ along $\rB \Inn(\cS) \to K(\bZ_2, 2)$. This cokernel is most efficiently described and computed in terms of a certain generalized cohomology theory $\SH^\bullet$ called ``extended supercohomology'' which was introduced in the physics literature, and given its name, in~\cite{WangGu2017}. 
The underlying spectrum of $\SH^\bullet$ is (a shift of) the Picard spectrum $(\Sigma \sVec)^\times$ of the symmetric monoidal $2$-category $\Sigma \sVec$; its nonvanishing homotopy groups are  $\SH^{-2}(\mathrm{pt}) \cong \SH^{-1}(\mathrm{pt}) \cong \bZ_2$ and $\SH^0(\mathrm{pt}) \cong \bk^\times$. 
Being built from $\sVec$, the Picard spectrum $(\Sigma \sVec)^\times$ has a (nontrivial) action by $\Aut(\sVec) = \rB \bZ_2$.
Hence, supercohomology of a space $X$ may be twisted by classes $\varpi \in \H^2(X; \bZ_2)$, and we will denote the $\varpi$-twisted supercohomology of $X$ by $\SH^{\varpi + \bullet}(X)$. In terms of $\SH$, the group $\pi_0 \Aut_{\sVec}(\cS) \cong \Mext(\cS) \cong \bZ_{16}$ turns out to be isomorphic to $\SH^{\id+4}(K(\bZ_2, 2))$ and the above cokernel  turns out to agree with the cokernel of the map $\varpi^*: \SH^{\id+4}(K(\bZ_2, 2)) \to \SH^{\varpi + 4} (\rB G)$.

 A super variant of the construction in~\cite{PhysRevX.8.031048,EvansGannon} shows that $\Obs(\sRep(G, \varpi))$ is in fact isomorphic to this subgroup of $\mathrm{NBF2C}(\cE)$.
 
Subject to our assumptions, this yields a sketch of a proof of the following conjecture. 
 \begin{conjecture}The obstruction group $\Obs(\sRep(G, \varpi))$ is isomorphic to 
\begin{equation} \label{eqn.superobstruction} \coker\left(\vphantom{\SH^{\id }}\Witt \to \Witt\left(\sRep(G, \varpi)\right)\right)\cong \coker \left( \varpi^*:  \SH^{\id+4}(K(\bZ_2, 2)) \to \SH^{\varpi + 4} (\rB G) \right).
\end{equation}
Moreover, $\Mext(\sVec) \cong \SH^{\id+4}(K(\bZ_2, 2))$ and there is a left exact sequence
\begin{equation} \label{eqn.supermext}
0 \to  \SH^{\varpi+3}(\rB G) \to \Mext(\sRep(G, \varpi)) \to \Mext(\sVec) \to[\varpi^*] \SH^{\varpi+4}(\rB G).
\end{equation}
\end{conjecture}
This exact sequence~\eqref{eqn.supermext} is the (relative, twisted) super-cohomological analogue of the recognition $\Mext(\Rep(G)) \cong \H^3(\rB G)$ as the set of trivializations of a degree-4 ordinary cohomology class.

\bibliographystyle{initalpha}
\bibliography{MME}

\end{document}